\documentclass{amsart}

\usepackage[utf8]{inputenc}
\usepackage[english]{babel}
\usepackage{graphicx}
\usepackage{amssymb}
\usepackage{amsmath}
\usepackage{tikz-cd}
\usepackage{adjustbox}
\usepackage{mathabx}
\usepackage{hyperref}
\usepackage{framed}
\usepackage{mathtools}
\usepackage[backend=bibtex,
style=alphabetic,
bibencoding=ascii
]{biblatex}
\usepackage{rotating}

\newtheorem{theorem}{Theorem}[section]
\newtheorem{lemma}[theorem]{Lemma}
\newtheorem{proposition}[theorem]{Proposition}
\newtheorem{corollary}[theorem]{Corollary}
\newtheorem{conjecture}[theorem]{Conjecture}
\newtheorem{construction}[theorem]{Construction}
\theoremstyle{definition}
\newtheorem{definition}[theorem]{Definition}
\newtheorem{assumption}[theorem]{Assumption}

\newtheorem{fact}[theorem]{Fact}

\theoremstyle{remark}
\newtheorem{remark}[theorem]{Remark}

\numberwithin{equation}{section}

\newcommand{\mcSS}{\mathcal{SS}}
\newcommand{\G}{\mathrm{G}}
\newcommand{\ov}{\overline}
\newcommand{\mcEQ}{\mathcal{EQ}}
\newcommand{\Int}{\mathrm{Int}}
\newcommand{\Inn}{\mathrm{Inn}}
\newcommand{\der}{\mathrm{der}}
\newcommand{\mc}{\mathcal}
\newcommand{\Out}{\mathrm{Out}}
\newcommand{\J}{\mathrm{J}}
\newcommand{\Q}{\mathbb{Q}}
\newcommand{\mb}{\mathbf}
\newcommand{\mf}{\mathfrak}
\newcommand{\Gr}{\mathrm{Groth}}
\newcommand{\inv}{\mathrm{inv}}

\newcommand{\A}{\mathbb{A}}
\newcommand{\Sc}{\mathrm{sc}}
\newcommand{\ad}{\mathrm{ad}}
\newcommand{\iso}{\mathrm{iso}}
\newcommand{\bas}{\mathrm{bas}}
\newcommand{\alg}{\mathrm{alg}}
\newcommand{\bb}{\mathbb}
\newcommand{\C}{\mathbb{C}}
\newcommand{\R}{\mathbb{R}}
\newcommand{\Gal}{\mathrm{Gal}}
\newcommand{\Z}{\mathbb{Z}}
\newcommand{\Sh}{\mathrm{Sh}}
\newcommand{\ms}{\mathsf}
\newcommand{\KT}{\mathrm{KT}}

\newcommand{\im}{\mathrm{im}}
\newcommand{\Aut}{\mathrm{Aut}}
\newcommand{\el}{\mathrm{ell}}
\newcommand{\SL}{\mathrm{SL}}
\newcommand{\tr}{\mathrm{tr}}
\newcommand{\Trans}{\mathrm{Trans}}
\newcommand{\Red}{\mathrm{Red}}
\newcommand{\Groth}{\mathrm{Groth}}
\newcommand{\Ind}{\mathrm{Ind}}
\newcommand{\Irr}{\mathrm{Irr}}
\newcommand{\Jac}{\mathrm{Jac}}

\newcommand{\LG}{{}^LG}
\newcommand{\LH}{{}^LH}
\newcommand{\Leta}{{}^L\eta}
\newcommand{\Lalpha}{{}^L\alpha}
\newcommand{\LM}{{}^LM}

\newcommand{\Ext}{\mathrm{Ext}}
\newcommand{\Mant}{\mathrm{Mant}}
\newcommand{\Ig}{\mathrm{Ig}}
\newcommand{\Res}{\mathrm{Res}}
\newcommand{\LL}{{}^L}

\title{An averaging formula for the cohomology of PEL-type Rapoport--Zink spaces}

\author{Alexander Bertoloni Meli}

\begin{document}

\begin{abstract}
    We prove under certain assumptions a formula for the cohomology of PEL-type Rapoport--Zink spaces that ``averages'' over the Kottwitz set. Our formula generalizes that of Shin's beyond the EL-type case and is proven by combining Mantovan's formula with descriptions of the cohomology of Shimura and Igusa varieties. We then use this averaging formula to derive a conjectural description of the cohomology of Rapoport--Zink spaces, generalizing our earlier work for EL-type spaces. Along the way, we give a description of the cohomology of Igusa varieties in terms of the Langlands correspondence, generalizing work in Shin's thesis. 
\end{abstract}

\maketitle

\tableofcontents

\section{Introduction}

The interactions between Rapoport--Zink spaces and other related moduli spaces, such as Igusa varieties and Shimura varieties, have been considered by many authors. In many cases, global methods have been used to prove powerful local results. For instance, in the pioneering work \cite{HT1}, Harris and Taylor were able to use these methods to prove the local Langlands correspondence for general linear groups over $p$-adic fields. Fargues (\cite{Far1}) and Shin (\cite{Shi1}) used these techniques to give formulas for the cohomology of EL-type Rapoport--Zink spaces in terms of the local Langlands correspondence. Strikingly, in \cite{Shi2}, Shin was able to reverse the method and use knowledge of the cohomology of certain Rapoport--Zink spaces to deduce powerful global results: namely the construction of a global Langlands correspondence for certain automorphic representations of $GL_n$. This paper further develops the method of using Shimura varieties and Igusa varieties to study the cohomology of Rapoport--Zink spaces. 

The goal of this paper is to give a framework that generalizes the averaging formula for the cohomology of Rapoport--Zink spaces proven by Shin (\cite[Theorem 1.2]{Shi1}). Shin's averaging formula is quite remarkable. On the one hand, it implies the Kottwitz conjecture on the supercuspidal part of the cohomology of non-basic Rapoport--Zink spaces and also allows one to conjecturally describe the non-supercuspidal part of the cohomology (cf. \cite{Shi1}, \cite{abm1}). On the other hand, there is currently no local proof of this formula despite the fact that its statement is purely local. 

We now describe the context in which this formula appears. Let $(G, b, \{\mu\})$ be a \emph{local Shimura datum} consisting of a connected reductive group $G$ over $\Q_p$, a conjugacy class of minuscule cocharacters $\{\mu\}$ of $G_{\ov{\Q_p}}$, and an element $b \in \mb{B}(\Q_p, G)$ where $\mb{B}(\Q_p, G)$ is the Kottwitz set. It was conjectured in \cite{RV1} that there is a tower of rigid analytic spaces $\bb{M}_{G, b, \mu, K}$ defined over $\breve{E_{\mu}}$, the completion of the maximal unramified extension of the reflex field $E_{\mu}$ of $\mu$, and indexed by compact open subgroups $K$ of $G(\Q_p)$. These towers were constructed in \cite{RZ1} in the case of Rapoport--Zink spaces of PEL-type and generally in \cite{BerkLect}.

In this paper, we are interested in the study of the $\ell$-adic cohomology $H_c^{i}(\bb{M}_{G, b, \mu, K}, \overline{\Q}_{\ell})$ of the spaces $\bb{M}_{G,b, \mu, K}$. These spaces carry an action of $G(\Q_p) \times J_b(\Q_p) \times W_{E_{\mu}}$ where $J_b(\Q_p)$ is a certain inner form of a Levi subgroup of $G$ associated to $b$. Hence, one can consider the spaces
\[
H^{i,j} (G, b, \mu)[\rho] := \varinjlim_K  \Ext^j_{J_b(\Q_p)}( H_c^{i}(\bb{M}_{G, b, \mu, K}, \overline{\Q}_{\ell}), \rho),  
\]
for $\rho$ an admissible representation of $J_b(\Q_p)$.

In this paper we will be interested in studying the homomorphisms of Grothendieck groups of admissible representations
\begin{equation*}
    \Mant_{G, b, \mu}: \Groth(J_b(\Q_p)) \to \Groth(G(\Q_p) \times W_{E_{\mu}}),
\end{equation*}
given by
\[
\Mant_{G, b, \mu} (\rho) := \sum_{i,j} (-1)^{i+j} H^{i,j} (G, b, \mu)[\rho](- \dim \bb{M}).
\]
These maps are considered by many authors (\cite{Far1}, \cite{Man2}, \cite{RV1}, \cite{Shi1}) and arise in the statement of the Kottwitz conjecture of \cite{RV1}. This construction also arises naturally in the cohomology theory of moduli spaces of local shtukas (cf \cite[Theorem 3.5.1]{KW}).

In many cases, such as for Rapoport--Zink spaces of PEL-type, these spaces are associated to the local geometry of Shimura varieties. Suppose now that we have a global Shimura datum $(\mb{G}, X)$ of PEL-type and such that $\mb{G}$ is a connected reductive group over $\Q$ that is anisotropic modulo center and $\mb{G}_{\Q_p}=G$.  In \cite{Man2}, Mantovan used the geometric relationship between Shimura varieties and Rapoport--Zink spaces to prove the following formula:
\begin{equation}{\label{mantform}}
    H^*_c(\Sh, \mc{L}_{\xi}) =  \sum\limits_{b \in \mb{B}(\Q_p, G, -\mu)} \Mant_{G,b, \mu}(H^*_c(\Ig_b, \mc{L}_{\xi})), 
\end{equation}
in $\Groth(G(\Q_p) \times W_{E_{\mu}})$. We denote respectively by $H^*_c(\Sh, \mc{L}_{\xi})$ and $H^*_c(\Ig_b, \mc{L}_{\xi})$ the alternating sums of the compactly supported cohomology of Igusa and Shimura varieties evaluated at the $\ell$-adic sheaf $\mc{L}_{\xi}$ associated to some irreducible algebraic representation $\xi$ of $\mb{G}$. 

For both Shimura varieties and Igusa varieties, there are, in many cases, trace formulas describing their cohomology (\cite{kot7}, \cite{Shi4}, \cite{Shi3}). In theory, one can combine these trace formulas and the above formula of Mantovan to derive an ``averaging formula'' for the cohomology of Rapoport--Zink spaces. In \cite{Shi1}, Shin does this for $\Res_{E/\Q_p} GL_n$ for $E/\Q_p$ an unramified Galois extension. In this case, he proves (up to an explicit character twist) 
\begin{equation}{\label{shinavg}}
    \sum\limits_{b \in \mb{B}(\Q_p, G, \mu)} \Mant_{G,b, -\mu}(\Red_b(\pi)) = \pi \boxtimes r_{-\mu} \circ \phi_{\pi},
\end{equation}
for $\pi$ an \emph{acceptable} representation of $G(\Q_p)$ and where $\phi_{\pi}$ is the semi-simplified Langlands parameter attached to $\pi$. The map $\Red_b: \Groth(G(\Q_p)) \to \Groth(J_b(\Q_p))$ shows up in Shin's analysis of the cohomology of Igusa varieties and is a combination of a Jacquet module and Jacquet-Langlands map. Loosely, an admissible representation $\pi$ is acceptable if it appears as the $p$-component of a representation appearing in the cohomology of a Shimura variety. Shin proves that such representations are dense in the Bernstein variety of $G(\Q_p)$.

In \cite{Shi1}, Shin used his formula to prove the Kottwitz conjecture for the groups he considers. In \cite{abm1}, we showed that Shin's formula can be combined with the Harris--Viehmann conjecture (\cite[Conjecture 8.4]{RV1}), to give a conjectural description of $\Mant_{G, b, \mu} \circ \Red_b$ for general admissible representations of $G(\Q_p)$.

The goal of this paper is to generalize Shin's formula and the formulas of \cite{abm1} to groups other than $GL_n$ .  In particular, we are interested in understanding what the above formula looks like for groups with nontrivial $L$-packets.

The most naive generalization of Equation \eqref{shinavg} would be to replace $\pi$ with the stable character sum over an $A$-packet, and in fact, such a formula follows from our work. On its own however, this formula is not sufficient to describe $\Mant_{G,b, \mu}$. In particular, it does not allow one to understand the  behavior of $\Mant_{G, b, \mu}$ on a single representation inside a packet.

To resolve this, we derive endoscopic versions of Equation \eqref{shinavg}. For each elliptic endoscopic datum $\mc{H}^{\mf{e}}=(H, s, \Leta)$ of $G$, we define a map 
\begin{equation*}
    \Red^{\mc{H}^{\mf{e}}}_b: \Groth^{st}(H(\Q_p)) \to \Groth(J_b(\Q_p)),
\end{equation*}
where $\Groth^{st}(H(\Q_p))$ is the Grothendieck group of of stable virtual representations of $H(\Q_p)$. Our main result is then as follows. 
\begin{theorem}[cf. Theorem \ref{finalformula}]{\label{mainthm}}
Fix $G$ a connected reductive group over $\Q_p$ and minuscule cocharacter $\mu$ such that $(G, b, \{\mu\})$ is PEL-type for $b \in \mb{B}(\Q_p, G, \mu)$. Then under a number of additional assumptions, we have the following formula in $\Groth(G(\Q_p) \times W_{E_{\mu}})$ associated to an Arthur parameter $\psi$ of $G$ that factors as $\Leta \circ \psi^H$.

\begin{equation*}
\sum\limits_{b \in \mathbf{B}(\Q_p, G, -\mu)} \mathrm{Mant}_{G,b,\mu}(
    \mathrm{Red}^{\mc{H}^{\mf{e}}}_b(S\Theta_{\psi^H}))
\end{equation*}
\begin{equation*}
    = \sum\limits_{\rho}\sum\limits_{\pi_p \in \Pi_{\psi}(G, \varrho_p)}  \langle \pi_p, \eta(s)\rangle \frac{\tr( \eta(s) | V_\rho)}{\dim \rho} \pi_p \boxtimes [\rho \otimes | \cdot |^{-\langle \rho_G, \mu \rangle}].
\end{equation*}
The first sum on the right-hand side is over irreducible factors of the representation $r_{-\mu} \circ \psi$ and $V_{\rho}$ is the $\rho$-isotypic part of $r_{-\mu} \circ \psi$. The term $S\Theta_{\psi^H}$ is the stable character sum for $\psi^H$.
\end{theorem}

Unfortunately, our derivation of the above formula makes a number of assumptions on $G$ and $\psi$, some of which likely need to be checked on a group by group basis. We comment on each of these assumptions at the end of the introduction. We expect that the above formula is true for all local Shimura varieties appearing in the $p$-adic uniformization of global Shimura varieties. However, we have stated Theorem \ref{mainthm} in the PEL-type case since this seems to be the most general case in the literature where the trace formulas for Shimura and Igusa varieties are well understood. However, the expected works of \cite{KSZ} and \cite{Mack-Crane} should allow the above formula to be extended to Hodge-type cases.

In the companion paper joint with Kieu Hieu (\cite{BMN1}), we verify all of these assumptions for odd unitary similitude groups unramified over $\Q_p$ and supercuspidal parameters $\psi$. Thus in this case, we have Theorem \ref{mainthm} unconditionally. In that paper, we deduce the following corollary:
\begin{corollary}[\cite{BMN1}]
The Kottwitz conjecture of \cite{RV1} is true for unramified unitary similitude groups in an odd number of variables.
\end{corollary}

As we have indicated, the proof of Theorem \ref{mainthm} is by combining the trace formulas for Shimura and Igusa varieties with the Mantovan formula (Equation \eqref{mantform}). To derive our result, we must necessarily deal with all the endoscopic terms of these trace formulas, which makes our analysis considerably more complex than that of \cite{Shi1}.

\section*{Outline of the paper}
We now describe the contents of the paper in detail. Section 2 of the paper describes in detail the theory of \emph{refined endoscopy} which we adopt in this paper. A refined endoscopic datum  $(H,s,\eta)$ of a connected reductive group $G$ over a local or global field $F$ differs from the standard notion of endoscopy in that we require our element $s$ to be in $Z(\widehat{H})^{\Gamma_F}$ as opposed to just $Z(\widehat{H})$ in the standard definition. Our use of this form of endoscopy is necessitated by our choice to normalize the global automorphic multiplicity formula and local Langlands correspondences using the \emph{isocrystal normalization} as in \cite[\S 4.6]{KalTai} and \cite[Conjecture F]{Kal1} respectively. 

In Section 2 we discuss refined endoscopy for Levi subgroups. Given a Levi subgroup $M$ of a connected reductive group $G$, there is a standard procedure (see for instance \cite[\S 2.4]{Mor1}) for producing an endoscopic datum $(H, s, \eta)$ for $G$ from a datum $(H_M, s_M, \eta_M)$ for $M$. In our notation, this is a map
\begin{equation*}
    Y: \mc{E}^r(M) \to \mc{E}^r(G),
\end{equation*}
where $\mc{E}^r(M)$ and $\mc{E}^r(G)$ are the sets of isomorphism classes of refined endoscopic data for $M$ and $G$ respectively.

However, for representation theoretic applications, one often needs to keep track of an embedding $H_M \hookrightarrow H$. Our notion of an \emph{embedded endoscopic datum} keeps track of this embedding and we show in Proposition \ref{refemb} that there is a natural bijection, $X$, between isomorphism classes of refined and embedded endoscopic data that fits into a commutative diagram:

\begin{equation}
\begin{tikzcd}
&\mathcal{E}^r(G)&\\
\mathcal{E}^e(M,G) \arrow[ur, " Y^e"] \arrow[rr, "X"] && \mathcal{E}^r(M) \arrow[ul, swap, "Y"].
\end{tikzcd}
\end{equation}

The notion of embedded endoscopic datum simplifies some of the constructions and arguments in Shin's stabilization of the trace formula for Igusa varieties (\cite{Shi3}) which we review in Section 4. In Proposition \ref{endparam} we give an explicit parametrization of the map $Y^e$ above in terms of double cosets of Weyl groups. This allows one to explicitly compute the more mysterious sets appearing in Shin's stabilization. Interestingly, analogous parametrizations appear in work of Xu (\cite{Xu1}) and Hiraga (\cite{Hir1}) in their formulas for the compatibility of endoscopic transfer and Jacquet modules.

In Section $3$ we review the form of the local and global Langlands correspondences we will use in this paper. Note that in this paper we use the formalism of the conjectural global Langlands group. This is done because it greatly simplifies the exposition. In practice, since our goal is to prove local formulas, one can work with Arthur's version of global parameters in terms of cuspidal automorphic representations as in (\cite{Arthurbook}). This is the approach followed by the author and Nguyen in \cite{BMN1}.

In Section $4$, we describe the stabilization and destabilization of the trace formula for the cohomology of Shimura varieties following Kottwitz (\cite{kot7}). Our task in this section is primarily to re-intepret Kottwitz's work using the normalizations of the local and global Langlands correspondences and of endoscopy that we have made in this paper.  The trace formula for Shimura varieties associated to a connected reductive group $\mb{G}$ over $\Q$ includes terms indexed by semisimple but not strongly-regular elements $\gamma \in \mb{G}(\Q)$. In the stabilization of these terms, one needs an explicit formula for the invariant relating the value of a transfer factor $\Delta(\gamma_{\mb{H}}, \gamma)$ to that of $\Delta(\gamma_{\mb{H}}, \gamma^*)$ where $\gamma^* \in \mb{G}^*(\Q)$ for a fixed quasisplit inner form of $\mb{G}$ and $\gamma^*$ and $\gamma_{\mb{H}}$ both transfer to $\gamma$. These formulas require a modest extension of Kottwitz's theory of $\mb{B}(\Q, \mb{G})$ and are proven in \cite{BM3}.

In Section $5$ we repeat our analysis of the trace formula for Shimura varieties in the case of Igusa varieties. The stabilization is previously known by work of Shin (\cite{Shi3}), but prior to this work, the destabilization had only been studied in a simplified case in Shin's thesis (for instance he only works with PEL data of A-type and assumes his groups are products of inner forms of $GL_n$ at $p$). The results of Section 2 and our work on transfer factors lead to some simplifications of Shin's stabilization. Most notably, the constant $c_{M_H}$ that appears in the definition of the function $h^{\mb{H}}$ in Shin's main theorem \cite[Theorem 7.2]{Shi3} is equal to $1$ under our normalizations. This constant is in general quite difficult to compute (cf \cite[\S6.4, \S8]{Shi3}) and so our normalizations make Shin's formula more explicit.

In Section 6, we combine the cohomology formulas for Igusa and Shimura varieties with the Mantovan formula. Under several assumptions, we derive Theorem \ref{mainthm}.

Finally, in Section 7 we generalize the results of \cite{abm1} and derive a conjectural formula for $\Mant_{G, b, \mu} \circ \Red^{\mc{H}^{\mf{e}}}_b$. In \cite{abm1}, we were able to show by assuming the Harris-Viehmann conjecture, that for $E/\Q_p$ a finite unramified extension and $G=\Res_{E/F} GL_n$ and $\mu$ of weights $0$ and $1$, that
\begin{equation*}
    \Mant_{G, b, \mu} \circ \Red^{\mc{H}^{\mf{e}}_{triv}}_b = [\mc{M}_{G, b, \mu}]
\end{equation*}
in $\Groth(G(\Q_p) \times W_{E_{\mu}}$. In the  above expression, $\mc{H}^{\mf{e}}_{triv}$ refers to the trivial endoscopic datum and the term $ [\mc{M}_{G, b, \mu}]$ is an explicit combination in the Grothendieck group of maps that up to character twists are of the form
\begin{equation*}
\begin{tikzcd}
\Groth(G(\Q_p)) \arrow[dd, "\Jac^G_P"] & & & \Groth(G(\Q_p) \times W_{E_{\mu}})\\
&&& \\
\Groth(M(\Q_p)) \arrow[rrr, "\pi \mapsto \pi \boxtimes r_{-\mu} \circ \phi_{\pi}"] &&& \Groth(M(\Q_p) \times W_{E_{\mu}}) \arrow[uu, "\Ind^G_P"].
\end{tikzcd}    
\end{equation*}
In Section 7, we generalize this to case of nontrivial endoscopy. In particular we derive the following conjectural formula:
\begin{conjecture}
We have
\begin{equation*}
    \Mant_{G,b,\mu} \circ \Red^{\mc{H}^{\mf{e}}}_b = [\mc{M}_{\mc{H}^{\mf{e}}, G, b, \mu}],
\end{equation*}
in $\Groth(G(\Q_p) \times W_{E_{\{\mu_G\}}})$.
\end{conjecture}
In the above conjecture, the term $[\mc{M}_{\mc{H}^{\mf{e}}, G, b, \mu}]$ is an explicit combination in the Grothendieck group of maps that are up to character twists are of the form
\begin{equation*}
\begin{tikzcd}
\Groth^{st}(H(\Q_p)) \arrow[dd, "\Jac^H_{P_H}"] & & & \Groth(G(\Q_p) \times W_{E_{\mu}}) \\
&&&\\
\Groth^{st}(H_M(\Q_p)) \arrow[rrr, "{[\mu_{H_M, M}]}"] &&&  \Groth(M(\Q_p) \times W_{E_{\mu}}) \arrow[uu, "\Ind^G_P"],
\end{tikzcd}
\end{equation*}
where $[\mu_{H_M, M}]$ is a composition of the transfer map
\begin{equation*}
 \Trans^{H_{M}}_{M} : \Groth^{st}(H_{M}(\Q_p)) \to \Groth(M(\Q_p))
\end{equation*}
and the ``weighted local Langlands map''
\begin{equation*}
    LL: \Groth(M(\Q_p)) \to \Groth(M(\Q_p) \times W_{E_{\mu}}),
\end{equation*}
defined on tempered $\pi$ by
\begin{equation*}
    \pi \mapsto \langle \pi, \eta(s) \rangle \pi \boxtimes \sum\limits_{\rho \in \Irr(r_{- \mu} \circ \psi_{\pi})} \frac{ \tr( \eta(s) | V_{\rho})  }{\dim \rho} [\rho].
\end{equation*}

\subsection*{Discussion of assumptions}

We now comment on the assumptions appearing in this work. 

We often assume that $\mb{G}_{\der}$ is simply connected. This assumption shows up in two places. The first is that it guarantees that centralizers of semisimple elements are connected and the second is that it allows one to lift the maps $\eta: \widehat{\mb{H}} \to \widehat{\mb{G}}$ appearing in endoscopic data to $L$-maps $\Leta: \LL\mb{H} \to \LL\mb{G}$. This assumption should be removable using the theory of $z$-extensions to reduce statements for general $\mb{G}$ to the simply connected case. In practice, much of the existing literature (\cite{Shi3}, \cite{kot7}) uses this assumption and so it is expositionally simpler to do so.

There are a few assumptions we make relating to our normalization of the local Langlands correspondence using isocrystals. In particular, our assumption that $\mb{G}$ is an extended pure inner twist of $\mb{G}^*$ and our use of the Hasse principle arise because of this normalization.

Following \cite{kot7} we make the assumption that the maximal $\Q$-split torus in $Z(\mb{G})$  coincides with the maximal $\R$-split torus in $Z(\mb{G})$.

As we mentioned above, we only consider spaces of PEL-type since this is the generality in which the trace formulas for Igusa varieties and Shimura varieties are known. 

The Mantovan formula uses $\ell$-adic cohomology of Shimura varieties whereas the trace formula for Shimura varieties computes the intersection cohomology. We assume that $\mb{G}$ is anisotropic modulo center so that the relevant Shimura varieties are compact and these cohomology spaces coincide. In general one can use the fact that the cuspidal parts of the $\ell$-adic cohomology and intersection cohomology of Shimura varieties agree to avoid this assumption. In fact, we do this in \cite{BMN1}.

Assumption \ref{STELLA} is expected to be true and we verify this in \cite{BMN1} for unitary similitude groups using the work of \cite{Mor1}. One can likely avoid this assumption if one works with the full stable trace formula for Shimura varieties rather than restricting to the $P=\mb{G}$ term as we do in this paper.

In Section $3$ and in Assumption \ref{unramassump}, we make a number of assumptions about properties of the Langlands correspondence. These are expected to hold generally. In this paper, we found it simpler to describe our formula in terms of the conjectural global Langlands group. In practice, since our final formula is purely local, one can avoid this by working with Arthur's global automorphic parameters as in \cite{Arthurbook}. This is the approach we take in \cite{BMN1}.

Finally, we have assumptions on the existence of certain globalizations of local representations. Namely, given an admissible representation $\pi_p$ of $G(\Q_p)$, we need to assume the existence of global representations $\pi$ of $\mb{G}(\A_f)$ that are isomorphic to $\pi_p$ at $p$, appear in the cohomology of Shimura varieties, and satisfy certain other conditions. Assumption \ref{unicityassump} is of this form but can usually be resolved by placing conditions on $\pi$ to guarantee its $A$-parameter is trivial on the Arthur $SL_2(\C)$-factor. 

The most serious assumption is Assumption \ref{liftingassump}. For a supercuspidal parameter $\psi_p$ of an odd unitary similitude group, this condition amounts to constructing for each nontrivial elliptic endoscopic datum $(H, s, \Leta)$ through which $\psi_p$ factors, a global elliptic endoscopic datum $(\mb{H}, s, \Leta)$ lifting $(H,s,\Leta)$ and a parameter $\mb{\psi}$ lifting $\psi_p$ that only factors through the trivial endoscopic datum and $(\mb{H},s,\Leta)$ and contains a representation appearing in the cohomology of our chosen Shimura variety attached to $\mb{G}$ (cf \cite{BMN1}).  For $\psi_p$ with non-singleton packet, one will be forced to consider Shimura varieties whose trace formula has non-trivial endoscopic contributions. In particular, the Shimura varieties of \cite{Kot2} have no non-trivial endoscopic contributions and so will not suffice to satisfy Assumption \ref{liftingassump}.

\subsection*{Acknowledgements}
We wish to thank Sug Woo Shin for inspiring us to undertake this project and for his continual support and encouragement. We thank Rahul Dalal, Tasho Kaletha, Sander Mack-Crane, Kieu Hieu Nguyen, Masao Oi, and Alex Youcis for numerous helpful conversations and suggestions.

\section{Endoscopic preparations}
The theory of endoscopy is central to the following work. Moreover, it will be necessary for several of our constructions to differ subtly from those appearing in much of the literature.  Thus, we choose to devote a substantial amount of care to the following standard notions with the hope of avoiding confusion.
For this section, we adopt the following notations.
\begin{itemize}
    \item Let $F$ be a local or global field of characteristic $0$.\\
    \item Denote $\mathrm{Gal}(\overline{F}/F)$ by $\Gamma_F$.\\
    \item Let $G$ be a connected reductive group defined over $F$. Let $Z(G)$ be the center of $G$ and if $S \subset G$ is a subset, let $Z_G(S)$ denote the centralizer of $S$ in $G$.\\
    \item Let $\Psi_0(G)$  denote the canonical based root datum $\Psi_0(G_{\ov{F}})$ equipped with its natural $\Gamma_F$-action \cite[\S1]{Kot5}.\\
    \item Let $\widehat{G}$ denote a complex dual group of $G$. Thus, $\widehat{G}$ is a connected reductive group over $\C$ equipped with an action of $\Gamma_F$ that fixes some splitting of $\widehat{G}$ and a $\Gamma_F$-equivariant isomorphism between $\Psi_0(\widehat{G})$ and the dual of $\Psi_0(G)$. \\
    \item Let $^LG$ denote the Langlands dual group $\widehat{G} \rtimes W_F$ with the action of $W_F$ induced by the canonical map $W_F \to \Gamma_F$.\\
    \item The action of $W_F$ on $\widehat{G}$ induces a map $\rho_G: W_F \to \mathrm{Out}(\widehat{G})$.\\
    \item For convenience, we fix a $\Gamma_F$-invariant splitting $(\bb{B}, \bb{T}, \{X_{\alpha}\})$ of $\widehat{G}$.\\
    \item Fix a quasisplit reductive group $G^*$. Then an \emph{inner twist} of $G^*$ is an isomorphism $\Psi: G^* \to G$ defined over $\ov{F}$ and such that for each $\sigma \in \Gamma_F$, we have that $\Psi^{-1} \sigma(\Psi)=\Int(g_{\sigma})$ for some $g_{\sigma} \in G^*(\ov{F})$. Associated to $\Psi$ is the $1$-cocycle of $\Gamma_F$ valued in $G^*_{\ad}(\ov{F})$ and given by $\sigma \mapsto g_{\sigma}$. Inner twists $\Psi_1: G^* \to G_1$ and $\Psi_2: G^* \to G_2$ are equivalent if the corresponding cocycles are equal in $H^1(F, G^*_{\ad})$.\\
    \item A \emph{form} of $G$ is a reductive group $G'$ over $F$ such that $G_{\ov{F}} \cong G'_{\ov{F}}$. The forms of $G$ are classified by $H^1(F, \Aut(G))$ and a form $G'$ is inner to $G$ if its class lies in the image of $H^1(F, G_{\ad})$.
\end{itemize}

\subsection{Inner twists and dual groups}
In this subsection, we record some useful lemmas that we will need on inner twists. To begin, we record the following standard fact.
\begin{fact}{\label{qsforms}}
Let $G$ be a split reductive group. Then there is a bijection between isomorphism classes of quasisplit forms of $G$ and $H^1(F, \Out(G))$, where $\Out(G)$ has the trivial $\Gamma_F$-action.
\end{fact}

Fix a reductive group $H$ over $\C$. By the classification theorem for reductive groups, there is, up to isomorphism, a unique split reductive group $G$ over $F$ such that the dual of $\Psi_0(G)$ is isomorphic to $\Psi_0(H)$. We will often use the following lemma. 
\begin{lemma}{\label{qsbijlem}}
There is a natural bijection between 
\begin{enumerate}
\item the set of isomorphism classes of quasisplit forms $G^*$ of $G$,
\item the set of $\Out(H)$-conjugacy classes of morphisms $\rho: \Gamma_F \to \Out(H)$.
\end{enumerate}
\end{lemma}
\begin{proof}
We have an isomorphism
\begin{equation*}
    \Out(G) \cong \Aut(\Psi_0(G)) = \Aut(\Psi_0(G)^{\vee})=\Aut(\Psi_0(H))=\Out(H).
\end{equation*}
Then we observe that the first set is identified with $H^1(F, \Out(G))$ and the second with $H^1(F, \Out(H))$.
\end{proof}

Now let $G$ be a connected reductive group over $F$ and take $g \in G(\ov{F})$. Then $\Int(g): G \to G$ induces the identity automorphism of the canonical based root datum. We note the following consequences of this observation:
\begin{fact}{\label{equivarenum}}
\begin{enumerate}
\item  If $\Psi: G^* \to G$ is an inner twist then $\Psi$ induces a $\Gamma_F$-equivariant map on canonical based root data.

\item If $\eta: \widehat{G} \to \widehat{G'}$ is an isomorphism whose conjugacy class is $\Gamma_F$-stable, then $\eta$ induces a $\Gamma_F$-equivariant map of canonical based root data. Moreover, by choosing $F$-splittings of $\widehat{G}$ and $\widehat{G'}$ and a $\widehat{G}$-conjugate $\eta_1$ of $\eta$ that takes one $F$-splitting to the other, we see that $\eta_1^{-1} \circ \gamma \cdot \eta_1$ will act as the identity on our $F$-splitting of $\widehat{G}$ and hence $\eta_1$ is $\Gamma_F$-equivariant.
\end{enumerate}
\end{fact}

We deduce the following result:
\begin{lemma}{\label{itdg}}
Fix a quasisplit group $G^*$ and an inner form $G$ of $G^*$. Let $\widehat{G^*}$ and $\widehat{G}$ be dual groups of $G^*$ and $G$ respectively. Then there is a natural bijection between 
\begin{enumerate}
    \item Conjugacy classes of inner twists $\Psi: G^* \to G$,
    \item $\widehat{G}^{\Gamma_F}$-conjugacy classes of $\Gamma_F$-equivariant isomorphisms $\xi: \widehat{G} \to \widehat{G^*}$.
    \end{enumerate}
\end{lemma}
\begin{proof}
We claim the second item is the same as the set of  $\Gamma_F$-equivariant isomorphisms $\Psi_0(\widehat{G}) \to \Psi_0(\widehat{G^*})$. Indeed, the set of $\Gamma_F$-equivariant isomorphisms of based root data is in bijection with conjugacy classes of isomorphisms $\widehat{G} \to \widehat{G^*}$ containing a $\Gamma_F$-equivariant element by Fact \ref{equivarenum}.  On the other hand, if $\xi, \xi'$ are $\Gamma_F$-equivariant and in the same conjugacy class, then by \cite[Lemma 1.6]{Kot5} they are in the same $\widehat{G}^{\Gamma_F}$-conjugacy class. This proves the claim.

Now, composing with the fixed isomorphisms $\Psi_0(\widehat{G}) \cong \Psi_0(G)^{\vee}$ and similarly for $G^*$, we see that this set is the same as the set of $\Gamma_F$-equivariant isomorphisms $\Psi_0(G)^{\vee} \to \Psi_0(G^*)^{\vee}$. Dually, this set is naturally identified with $\Gamma_F$-equivariant isomorphisms $\Psi_0(G^*) \to \Psi_0(G)$ which is identified with the set of equivalence classes of inner twists $\Psi: G^* \to G$ again by Fact \ref{equivarenum}.
\end{proof}

\subsection{Standard endoscopic data}
In this paper, we primarily use the following definition of endoscopic datum. This definition is due to Kottwitz and appears in work of Shin \cite[Definition 2.1]{Shi3}.
\begin{definition}{\label{endtrip}}
An \emph{endoscopic triple of $G$} is a tuple $(H,s,\eta)$ such that $H$ is a quasisplit reductive group defined over $F$, the element $s \in Z(\widehat{H})$, and $\eta: \widehat{H} \to \widehat{G}$ is an embedding. Further, the following holds.
\begin{itemize}
    \item The $\widehat{G}$-conjugacy class of $\eta$ is stable under the action of $\Gamma_F$. Recall that $\gamma \in \Gamma_F$ acts by $\gamma \cdot \eta= \gamma \circ \eta \circ \gamma^{-1}$.\\
    \item The image of $s$ in $Z(\widehat{H})/Z(\widehat{G})$ is $\Gamma_F$-invariant and its image under the connecting homomorphism, 
    \begin{equation}{\label{endhom}}
 (Z(\widehat{H})/Z(\widehat{G}))^{\Gamma_F} \to H^1(F, Z(\widehat{G})),
    \end{equation}
    arising from the $\Gamma_F$-invariant short exact sequence
    \begin{equation*}1 \to Z(\widehat{G}) \to Z(\widehat{H}) \to Z(\widehat{H})/Z(\widehat{G}) \to 1 \end{equation*}
    is trivial if $F$ is local and locally trivial if $F$ is global.
\end{itemize}
\end{definition}
\begin{remark}
Note that the first condition implies that the map $Z(\widehat{G}) \hookrightarrow Z(\widehat{H})$ is $\Gamma_F$-equivariant so that the second condition makes sense. After all, the first condition implies that for $\gamma \in \Gamma_F$, there exists a $g_{\gamma} \in \widehat{G}$ so that 
\[
\gamma \cdot \eta= \gamma \circ \eta \circ \gamma^{-1}= \mathrm{Int}(g_{\gamma}) \circ \eta.
\]
In particular, for $z \in Z(\widehat{H})$ such that $\eta(z) \in Z(\widehat{G})$, we have 
\[
\gamma(\eta(z))= \mathrm{Int}(g_{\gamma})(\eta(\gamma(z)))=\eta (\gamma(z)),
\]
as desired.
\end{remark}
\begin{definition}{\label{isoendtrip}}
An \emph{isomorphism} of endoscopic triples of $G$,
\[
i: (H, s, \eta) \to (H', s', \eta'),
\]
is an isomorphism $\alpha: H \to H'$ over $F$ such that 
\begin{itemize}
    \item the maps $\eta \circ \widehat{\alpha}$ and $\eta'$ are conjugate by an element of $\widehat{G}$,
    \item the elements $s$ and $\widehat{\alpha}(s')$ are equal in $Z(\widehat{H})/Z(\widehat{G})$.
\end{itemize}
We denote the set of isomorphism classes of endoscopic data by $\mathcal{E}(G)$.
We define the outer automorphism group of $\mathrm{Out}(H, s, \eta)$ to be $\mathrm{Aut}(H,s,\eta)/H_{\ad}(F)$.
\end{definition}
\begin{remark}
We need to explain what we mean by $\widehat{\alpha}$. The map $\alpha$ induces an isomorphism between the root data of $H$ and $H'$ and therefore between $\widehat{H}$ and $\widehat{H'}$. Such an isomorphism determines a $\Gamma_F$-equivariant isomorphism $\widehat{\alpha}: \widehat{H} \cong \widehat{H'}$ unique up to $\widehat{H}^{\Gamma_F}$-conjugacy.
\end{remark}
\begin{remark}
The above definitions are equivalent to their respective notions in \cite[\S 2.1]{KS} (see Appendix\ref{appendixA}).
\end{remark}

We will routinely use that the groups $\Out(H,s,\eta)$ are finite, which follows from the following easy lemma.
\begin{lemma}{\label{finoutim}} Let $X$ be a complex reductive group. Let $s\in X(\C)$ be semisimple and set $Y:=Z_X(s)^\circ$. Then, the map $N_X(Y)\to \Out(Y)$ given on $\C$-points by sending $x\in N_X(Y)(\C)$ to $\Int(x)_{\mid Y}$ has finite image. 
\end{lemma}

\begin{proof} We first note that $Z_X(Z(Y))^\circ$ is contained in the kernel of the map $N_X(Y)\to \Out(Y)$. Indeed, it suffices to show that $Z_X(Z(Y))^\circ \subseteq Y$.  We first observe that $s\in Z(Y)$. Evidently $s\in Z(Z_X(s))\subseteq Z_X(s)$ so the only non-trivial statement is that $s$ is actually in $Z_X(s)^\circ=Y$. But, note that since $s$ is semisimple, we have $s \in T(\C)$ for $T$ a maximal torus of $X$. Hence $s \in T(\C) \subset Y$ and so  $s\in Y$ and thus $s\in Z(Y)$. Therefore, $Z_X(Z(Y))\subseteq Z_X(s)$ and thus $Z_X(Z(Y))^\circ\subseteq Z_X(s)^\circ=Y$. 

To finish the proof, it suffices to show that $N_X(Y)/Z_X(Z(Y))^\circ$ is finite. But, since $Z_X(Z(Y))^\circ$ is finite index in $Z_X(Z(Y))$ it suffices to show that $N_X(Y)/Z_X(Z(Y))$ is finite. Note though that $N_X(Y)\subseteq N_X(Z(Y))$ since $Z(Y)$ is a characteristic subgroup of $Y$. Thus, we get an inclusion
\begin{equation*}
    N_X(Y)/Z_X(Z(Y))\hookrightarrow N_X(Z(Y))/Z_X(Z(Y)),
\end{equation*}
and thus it suffices to show this latter group is finite. Of course, this is equivalent to showing that $N_X(Z(Y))^\circ$ and $Z_X(Z(Y))^\circ$ coincide. Since $Z(Y)$ is multiplicative (since $Y$ is reductive by \cite[\S2.2]{Hum1}) this claim follows from \cite[Corollary, \S16.3]{Hum2}.
\end{proof}

\subsection{Refined endoscopic data}
The Definitions \ref{endtrip} and \ref{isoendtrip} suffice for the study of endoscopy when $F$ is global. In general it is desirable that the transfer factors are invariant under isomorphisms of endoscopic triples.  For local fields this will not be the case using the above definitions (see \cite[\S 3]{Art4}). This motivates the following definitions which were suggested to us by Tasho Kaletha.

\begin{definition}
A \emph{refined endoscopic triple of $G$} is an endoscopic triple $(H,s,\eta)$ such that in addition, $s \in Z(\widehat{H})^{\Gamma}$.

An \emph{isomorphism} of refined endosopic triples of $G$,
\begin{equation}
    i: (H, s, \eta) \to (H', s', \eta')
\end{equation}
is an isomorphism $\alpha: H \to H'$ such that
\begin{itemize}
    \item the maps $\eta \circ \widehat{\alpha}$ and $\eta'$ are conjugate by an element of $\widehat{G}$,
    \item the elements $s$ and $\widehat{\alpha}(s')$ are equal in $Z(\widehat{H})^{\Gamma}$.
\end{itemize}
We denote the set of isomorphism classes of refined endoscopic data by $\mathcal{E}^r(G)$.
\end{definition}
\begin{definition}
A refined or standard endoscopic triple $(H,s,\eta)$ is \emph{elliptic} if $(Z(\widehat{H})^{\Gamma_F})^{\circ} \subset Z(\widehat{G})$.
\end{definition}
\begin{remark}
In the case that $F$ is local, any isomorphism class of endoscopic triples of $G$ has a refined representative. In the global case, this is true if $G$ satisfies the Hasse principle.
\end{remark}

\subsection{Maps of \texorpdfstring{$L$}{L}-groups}
A group map $\xi: \LG_1 \to \LG_2$ is an $L$-\emph{map} if the induced map $\widehat{G_1} \to \widehat{G_2}$ is a map of algebraic groups and $\xi$ commutes with the projections of $\LG_1$ and $\LG_2$ to $W_F$. An $L$-map that is also an embedding is called an $L$-\emph{embedding}.

If we assume $G_{\der}$ is simply connected, it is a fact that any embedding $\eta: \widehat{H} \to \widehat{G}$ extends to an $L$-embedding $\Leta: \LH \to \LG$ (\cite[Prop 1]{Lan1}).

We now show that we can extend isomorphisms of endoscopic data to give maps of $L$-groups.

\begin{lemma}{\label{alphalem}}
Suppose that $(H_1 , s_1, \eta_1)$ and $(H_2, s_2, \eta_2)$ are refined endoscopic data and suppose there exist lifts $^L\eta_1$ and $^L\eta_2$ of $\eta_1$ and $\eta_2$ respectively. Suppose further that $\alpha: H_2 \to H_1$ gives an isomorphism of endoscopic data (refined or standard) and $g \in \widehat{G}$ is such that $\Int(g) \circ \eta_1 = \eta_2 \circ \widehat{\alpha}$. Then for each choice of $\widehat{\alpha}$, there exists a lift $^L\alpha$ of $\alpha$ such that the following diagram commutes:
\begin{equation}{\label{endisodiagram}}
\begin{tikzcd}
^LH_1 \arrow[r, "^L\eta_1"] \arrow[d, swap, "^L\alpha"] & ^LG \arrow[d, "\Int(g)"] \\
^LH_2 \arrow[r, "^L\eta_2"]& ^LG.
\end{tikzcd}
\end{equation}
Moreover, the $\widehat{H_1}$-conjugacy class of $^L\alpha$ does not depend on the choice of $\widehat{\alpha}$ or $g$.
\end{lemma}
\begin{proof}
We want to define $^L\alpha$ to equal $^L\eta^{-1}_2 \circ \, \Int(g) \circ \, ^L\eta_1$. For this to make sense, we need to show that the image of $\Int(g) \circ \, ^L\eta_1$ is contained in the image of $^L\eta_2$.

Now there exists for each $w \in W_F$ and $i \in \{1, 2\}$, elements $g_{w,i} \in \widehat{G}$ so that $^L\eta_i(1,w)=(g_{w,i},w)$. We observe that for any $h_i \in \widehat{H_i}$, we have 
\begin{align*}
    (g_{w,i} (w \cdot \eta_i)(h_i),w)&=\,^L\eta_i(1,w)\, ^L\eta_i(w^{-1}(h_i),1) \\
    &=\, ^L\eta_i(h_i,w)\\
    &=\, ^L\eta_i(h_i,1)\, ^L\eta_i(1,w)\\
    &=(\eta_i(h_i)g_{w,i},w),
 \end{align*}
so that
\begin{equation}{\label{cocycleid}}
    \Int(g^{-1}_{w,i})(\eta_i(h_i))=(w \cdot \eta_i)(h_i).
\end{equation}

Now, it suffices to check that for each $(1,w) \in \, ^LH_1$ there exists an $(h_2,w) \in \, ^LH_2$ such that 
\begin{equation*}
    \Int(g)(\Leta_1(1,w)) =(gg_{w,1}w(g^{-1}),w)=(\eta_2(h_2)g_{w,2},w)= \Leta_2(h_2, w).
\end{equation*}
Hence we need to check that $gg_{w,1}w(g^{-1})g^{-1}_{w,2} \in \eta_2(\widehat{H_2})$. It suffices to show that this element lies in $Z_{\widehat{G}}(\eta_2(\widehat{H_2}))$ since for any maximal torus $T$ of $\widehat{H_2}$, we have $\eta_2(T)$ is a maximal torus of $\widehat{G}$ and so 
\begin{equation}{\label{centralizertrick}}
    Z_{\widehat{G}}(\eta_2(\widehat{H_2})) \subset Z_{\widehat{G}}(\eta_2(T))=\eta_2(T) \subset \eta_2(\widehat{H_2}).
\end{equation}
Now pick $h_2 \in \widehat{H_2}$. We observe that using Equation \eqref{cocycleid}, we have 
\begin{align*}
    \Int(gg_{w, 1}w(g^{-1})g^{-1}_{w,2})(\eta_2(h_2))&= \Int(gg_{w,1}w(g^{-1}))((w \cdot \eta_2)(h_2))\\
    &=(\Int(gg_{w, 1}) \circ w \circ  \Int(g^{-1}) \circ \eta_2)(w^{-1}(h_2))\\
    &=(\Int(gg_{w,1}) \circ w \circ \eta_1 \circ \widehat{\alpha}^{-1})(w^{-1}(h_2))\\
    &=(\Int(gg_{w,1}) \circ w \cdot \eta_1 \circ \widehat{\alpha}^{-1})(h_2)\\
    &=(\Int(g) \circ \eta_1 \circ \widehat{\alpha}^{-1})(h_2)\\
    &=\eta_2(h_2),
\end{align*}
as desired.

Now we show the second statement of the lemma. As above, we have that the map $\widehat{\alpha}$ is unique up to $\widehat{H_1}^{\Gamma_F}$-conjugacy. For a fixed choice of $\widehat{\alpha}$ if we pick two different $g, g' \in \widehat{G}$ such that the requisite diagram commutes, then $\Int(g^{-1}g')$ fixes $\eta_1(\widehat{H_1})$ pointwise and so by Equation \eqref{centralizertrick} we have $g^{-1}g' \in \eta_1(Z(\widehat{H_1}))$. Hence any two $^L\alpha$ will differ at most up to conjugacy by an element of $\widehat{H_1}$.
\end{proof}
\subsection{Endoscopic data and Levi subgroups}
We now discuss a key map of endoscopic data.
\begin{construction}{\label{Ymap}}
Given a Levi subgroup $M$ of a reductive group $G$, there exists a natural map from refined endoscopic data of $M$ to refined endoscopic data of $G$. This map preserves isomorphism classes and hence gives a map
\begin{equation*}
    Y : \mathcal{E}^r(M) \to \mathcal{E}^r(G).
\end{equation*}
\end{construction}
\begin{proof}
The construction uses ideas from Appendix \ref{appendixA} and is essentially equivalent to the one appearing after Definition 2.4.1 of \cite{Mor1}.

Given a refined endoscopic datum $(H_M, s_M, \eta_M)$ of $M$, we have a natural candidate for $\widehat{H}$: namely $Z_{\widehat{G}}(s_M)^0$. What we need to do is define a continuous map $\Gamma_F \to \mathrm{Out}(\widehat{H})$ which then allows us to define a quasisplit group $H$ over $F$ whose dual group is $\widehat{H}$ by Lemma \ref{qsbijlem}. 

Define $\mathcal{H}_M \subset \, ^LM$ as in Construction \ref{triptodat}. This is a split extension
\begin{equation}
    1 \to \eta_M(\widehat{H_M}) \to \mathcal{H}_{M} \to W_F \to 1.
\end{equation}
Pick a splitting $c: W_F \to \mathcal{H}_M$. Then by definition of $\mc{H}_M$, there exists $(h, w) \in \, ^LH_M$ so that we have
\begin{equation*}
    \mathrm{Int}(c(w))(\eta_M(s_M))=\eta_M(\mathrm{Int}(h,w)(s_M))=\eta_M(s_M),
\end{equation*}
where the last equality is because $s_M \in Z(\widehat{H_M})^{\Gamma_F}$ and $W_F$ acts on $\widehat{H_M}$ through a finite quotient of $\Gamma_F$. Via the inclusion $\LM \subset \LG$, we have $c$ induces a map $c: W_F \to \mathrm{Aut}(\widehat{H}) \to \mathrm{Out}(\widehat{H})$. We note that since the action of $W_F$ on $\widehat{G}$ factors through a finite quotient, we have that $\Int(c(w))$ is an inner automorphism of $\widehat{H}$ for $w \in W_K$ for $K/F$ some finite Galois extension. Hence $c: W_F \to \Out(\widehat{H})$ factors through the projection to $\Gal(K/F)$  and hence induces a map $c: \Gamma_F \to \Out(\widehat{H})$ as desired. We define $H$ to be a representative of the corresponding isomorphism class of quasisplit groups. We let $\eta$ be the natural inclusion of $\widehat{H}$ into $\widehat{G}$ and define $s' := \eta^{-1}(\eta_M(s_M))$. Then $(H, s', \eta)$ is a refined endoscopic datum of $G$.
\end{proof}
In the context of the above construction, we record two useful lemmas.
\begin{lemma}{\label{endlevi}}
Let $G$ be a connected reductive group and $M$ a Levi subgroup. Suppose $s \in \widehat{M} \subset \widehat{G}$. Then $Z_{\widehat{M}}(s)^{\circ}=Z_{\widehat{G}}(s)^{\circ} \cap \widehat{M}$ and is a Levi subgroup of $Z_{\widehat{G}}(s)^{\circ}$.
\end{lemma}
\begin{proof}
Since $\widehat{M}$ is a Levi subgroup of $\widehat{G}$, we have $\widehat{M}= Z_{\widehat{G}}(Z(\widehat{M})^{\circ})$. In particular, 
\begin{equation}
Z_{\widehat{G}}(s)^{\circ} \cap \widehat{M}=Z_{Z_{\widehat{G}}(s)^{\circ}}(Z(\widehat{M})^{\circ}),
\end{equation}
which is a Levi subgroup of $Z_{\widehat{G}}(s)^{\circ}$. Hence the second assertion follows from the first.

Moreover, by the above equation, we see that since the centralizer  of a torus in a reductive group is connected, $Z_{\widehat{G}}(s)^{\circ} \cap \widehat{M}$ is connected and hence $Z_{\widehat{G}}(s)^{\circ} \cap \widehat{M} \subset Z_{\widehat{M}}(s)^{\circ}$. Conversely, $Z_{\widehat{M}}(s)^{\circ}$ is connected so clearly contained within $Z_{\widehat{G}}(s)^{\circ} \cap \widehat{M}$.
\end{proof}

\begin{lemma}{\label{levidersc}}
Suppose that $G$ is a connected reductive group and $G_{\der}$ is simply connected. If $M \subset G$ is a Levi subgroup then $M_{\der}$ is simply connected.
\end{lemma}
\begin{proof}
Fix a maximal torus $T \subset M \subset G$. Then we have  $M_{\der} \subset G_{\der}$ and tori $T_{M_{\der}} := T \cap M_{\der}$ and $T_{G_{\der}} := T \cap G_{\der}$ respectively. Let $X_M$ and $X_G$ be the image of the natural inclusions $X_*(T_{M_{\der}}) \subset X_*(T)$ and $X_*(T_{G_{\der}}) \subset X_*(T)$ respectively. By definition, we have $X_M \subset X_G$. Then $G_{\der}$ is simply connected if and only if the $\Z$-span of the absolute coroots of $G$ is equal to $X_G$ and the analogous statement is true for $M$.

Suppose $M_{\der}$ is not simply connected. Then we can find some $x \in X_M$ that is not in the integer span of the coroots of $M$. Make some choice of simple coroots of $G$. Then $x \in X_G$ and we claim that $x$ cannot be in the integer span of the simple coroots of $G$. Indeed if it were, then we could write $x= a_1\alpha_1 + ...+a_k\alpha_k + a_{k+1}\alpha_{k+1}+...+a_n\alpha_n$ where $\alpha_1, ..., \alpha_k$ are simple coroots of $M$ and $\alpha_{k+1},...,\alpha_n$ are simple coroots of $G$ but not $M$. Further, $a_i \in \Z$ for all $i$ and at least some $a_i \neq 0$ for $i >k$ by our assumption on $x$. But then $0 \neq a_{k+1}\alpha_{k+1} + ... + a_n \alpha_n \in X_M \subset X_{M, \R}$. This is impossible because the simple coroots of $G$ are linearly independent in $X_*(T)_{\R}$ and ${X_{M,\R}}$ is equal to the $\R$ span of the simple coroots of $M$.
\end{proof}

\subsection{Embedded endoscopic data}
Fix an element $[(H,s, \eta)] \in \mathcal{E}^r(G)$ and a Levi subgroup $M$ of $G$. Our eventual goal is to compute $Y^{-1}( [(H,s,\eta)])$. As a first step to computing these fibers, we describe a seemingly more rigidified notion of endoscopic datum for a Levi subgroup. This notion will be useful for describing maps of representations between groups such as Jacquet modules.

\begin{definition}{\label{embdat}}
Let $G$ be a reductive group over $F$ and $M$ an $F$-rational Levi subgroup. Fix a $\Gamma_F$-invariant splitting $(\widehat{T}, \widehat{B}, \{X_{\alpha}\})$ and $\widehat{M} \subset \widehat{G}$ be the Levi subgroup of a parabolic subgroup containing $\widehat{B}$ and corresponding to $M$. Then an \emph{embedded} endoscopic datum of $M$ relative to a fixed splitting $(\widehat{T_H}, \widehat{B_H}, \{X^H_{\alpha}\})$ is a quadruple $(H, H_M, s, \eta)$ such that $(H,s,\eta)$ is a refined endoscopic triple for $G$, the group $H_M$ is a Levi subgroup of $H$ and the restriction of $\eta$ to $\widehat{H_M}$ gives a refined endoscopic triple $(H_M, s, \eta|_{\widehat{H_M}})$ of $M$.

An isomorphism of embedded endoscopic triples $(H, M_H, s, \eta), (H', H'_M, s', \eta')$ is an isomorphism $\alpha: H \to H'$ of $(H,s,\eta)$ and $(H', s', \eta')$ over $F$ that restricts to give an isomorphism $\alpha_M: H_M \to H'_M$ of $(H_M, s, \eta|_{\widehat{H_M}})$ and $(H'_M, s, \eta|_{\widehat{H'_M}})$. We say an automorphism of $(H, H_M, s, \eta)$ is \emph{inner} if $\alpha_M$ is inner. Note that this implies $\alpha$ is an inner automorphism of $H$ since if $\alpha_M$ is given by conjugation by $h \in H_M(\ov{F})$, then $\alpha \circ \Int(h)^{-1}$ fixes $H_M$ pointwise, hence acts trivially on a root datum of $H$, hence is inner. We then define $\Out_e(H, H_M, s, \eta) := \Aut_e(H, H_M, s, \eta)/ \Inn(H, H_M, s, \eta)$.

We denote the set of isomorphism classes of embedded endoscopic data by $\mathcal{E}^e(M, G)$. 
\end{definition}
\begin{remark}{\label{compatibilityremark}}
Note that if $(H_1, {H_M}_1, s_1, \eta_1)$ and $(H_2, {H_M}_2, s_2, \eta_2)$ are isomorphic via $\alpha: H_2 \to H_1$, then any $g \in \widehat{G^*}$ making the following diagram commute:
\begin{equation*}
\begin{tikzcd}
\widehat{H_1} \arrow[d, swap, "\widehat{\alpha}"] \arrow[r, "\eta_1"] & \widehat{G} \arrow[d, "\Int(g)"] \\
\widehat{H_2} \arrow[r, "\eta_2"]  & \widehat{G}
\end{tikzcd}    
\end{equation*}
must be an element of $\widehat{M}$ and gives an isomorphism of embedded data.

Indeed, we need only show that $g \in \widehat{M}$. But by definition there exists an $m \in \widehat{M}$ such that $\Int(m) \circ \eta_1|_{\widehat{{H_M}_1}} = \eta_2|_{\widehat{{H_M}_2}} \circ \widehat{\alpha_M}$. Then if $\widehat{T}$ is a maximal torus of $\eta(\widehat{{H_M}_1})$ then $\Int(g^{-1}m)$ fixes $\widehat{T}$ pointwise hence $g^{-1}m \in \widehat{T} \subset \widehat{M}$. 
\end{remark}

We have a natural forgetful map 
\begin{equation}
    Y^e: \mathcal{E}^e(M,G) \to \mathcal{E}^r(G).
\end{equation}
\begin{proposition}{\label{refemb}}
The map $X: \mathcal{E}^e(M,G) \cong \mathcal{E}^r(M)$ defined by $(H, H_M, s, \eta) \mapsto (H_M, s, \eta|_{\widehat{H_M}})$ is a bijection and the following diagram commutes.
\begin{equation}{\label{refembdiag}}
\begin{tikzcd}
&\mathcal{E}^r(G)&\\
\mathcal{E}^e(M,G) \arrow[ur, " Y^e"] \arrow[rr, "X"] && \mathcal{E}^r(M) \arrow[ul, swap, "Y"].
\end{tikzcd}
\end{equation}
Moreover, $X: (H, H_M, s, \eta) \mapsto (H_M, s, \eta|_{\widehat{H_M}})$ induces an isomorphism $\Out_e(H,H_M, s,\eta) \cong \Out_r(H_M, s, \eta|_{\widehat{H_M}})$ via $\alpha \mapsto \alpha_M$.
\end{proposition}
\begin{proof}
The map $X$ extends to a map of isomorphism classes. Indeed, if $(H, H_M, s, \eta)$ and $(H', H'_M, s', \eta')$ are isomorphic by $\alpha$, then $\alpha_M$ gives an isomorphism of the corresponding refined triples of $M$.

Now, we check that Diagram \eqref{refembdiag} commutes. Given $(H, H_M, s, \eta)$, we get two refined endoscopic triples $(H,s, \eta)$ and $(H', s', \eta')$ from applying $Y^e$ and  $Y \circ X$ respectively. We must show these are isomorphic.

We have the groups $\mathcal{H}, \mathcal{H}_M \subset \, ^LG$ corresponding to $(H, s, \eta)$ and $(H_M, s, \eta|_{\widehat{H_M}})$ as defined in Construction \ref{triptodat}. We claim that $\mathcal{H}_M \subset \mathcal{H}$. To prove the claim, pick $(m, w) \in \mathcal{H}_M$. Then there exists $(h, w) \in \, ^LH_M$ such that restricted to $\widehat{H_M}$ we have
\begin{equation}
    \mathrm{Int}(m,w) \circ \eta = \eta \circ \mathrm{Int}(h,w).
\end{equation}
Now consider the two sides of the above equation as maps from $\widehat{H}$ (utilizing the natural embedding $\LH_M \subset \LH$). In particular, we see that composing one with the inverse of the other, we get an automorphism $A$ of $\widehat{H}$ that acts trivially on $\widehat{H_M} \subset \widehat{H}$. Fix a maximal torus $\widehat{T} \subset \widehat{H_M}$. 

Now we make the subclaim that $A$ is an inner automorphism by an element of $\widehat{T}$. After all, $A$ acts by the identity on $\widehat{T}$, hence trivially on $\psi_0(\widehat{H})$. Thus, $A$ lies in the kernel of the map $\mathrm{Aut}(\widehat{H}) \to \mathrm{Out}(\widehat{H})$, so is inner. But now $A$ centralizes $\widehat{T}$ which implies the subclaim.

The subclaim implies that there exists a $t \in \widehat{T} \subset \widehat{H_M}$ such that
\begin{equation}
    \mathrm{Int}(m,w) \circ \eta = \eta \circ \mathrm{Int}(th,w)
\end{equation}
on $\widehat{H}$. This proves that $\mathcal{H}_M \subset \mathcal{H}$ which was the original claim.

Now, we complete the proof that $(H,s, \eta)$ and $(H', s', \eta')$ are isomorphic. The strategy is to show that $\widehat{\alpha} := \eta'^{-1} \circ \eta$ has $\Gamma_F$-stable conjugacy class. We then use Fact \ref{equivarenum} to find a $\Gamma_F$-equivariant conjugate of $\widehat{\alpha}$ and define $\alpha$ to be the unique isomorphism dual to this and preserving a choice of $F$-splittings of $H, H'$. This will then imply that $\alpha$ is defined over $F$ and is the desired isomorphism.

Now, the action of $\Gamma_F$ on $\widehat{H}, \widehat{H'}$ factors through a finite quotient $\mathrm{Gal}(K/F)$ so pick $\gamma \in \mathrm{Gal}(K/F)$ and choose a lift to some $w \in W_F$. Then by definition of $\eta'$, we have that (using the notation of Construction \ref{Ymap}) $\mathrm{Int}(c(w)) \circ \eta'$ and $\eta' \circ w$ agree up to conjugacy by an element of  $\eta'(\widehat{H'})$. On the other hand, by the claim we proved above, $c(w) \in \mathcal{H}$ so there exists $(x,w) \in \, ^LH$ such that $\mathrm{Int}(c(w)) \circ \eta=\eta \circ \mathrm{Int}(x,w)$ on $\widehat{H}$. Hence $\Int(c(w)) \circ \eta$ and $\eta \circ w$ agree up to conjugacy by an element of $\eta(H)$. These two statements imply that $\widehat{\alpha} \circ w$ and $w \circ \widehat{\alpha}$ agree up to conjugation by an element of $\widehat{H}$ as desired. Finally, we observe that by construction $\widehat{\alpha}(s)=s'$ so that $\alpha$ indeed gives an isomorphism of refined endoscopic data.

We now show that the map $X: \mathcal{E}^e(M,G) \to \mathcal{E}^r(M)$ is a bijection. To prove this, we construct an inverse. Take an element of $\mathcal{E}^r(M)$ and choose a representative $(H_M,s_M, \eta_M)$. By Construction \ref{Ymap}, we get a datum $(H',s', \eta')$. Restricting $\eta'$ to $\eta(\widehat{H_M})$ we get a refined endoscopic datum $(H'_M, s', \eta'|_{\widehat{H'_M}})$ which is isomorphic to $(H_M, s_M, \eta_M)$ by construction. Then $(H', H'_M, s', \eta')$ is an embedded endoscopic datum and this defines a map $X'$. By construction we have $X \circ X'$ is the identity so it suffices to show the same for $X' \circ X$. In other words, starting with an embedded datum $(H, H_M, s, \eta)$ we need to show an isomorphism between $X'(X(H,H_M, s, \eta))$ and $(H', H'_M, s', \eta')$. But by our argument that the diagram in the proposition commutes, we constructed an isomorphism $\alpha$ of $(H, s, \eta)$ and $(H', s', \eta')$ that clearly restricts to an isomorphism of $(H_M, s, \eta|_{\widehat{H_M}})$ and $(H'_M, s', \eta'_{\widehat{H'_M}})$.

It remains to show that $\alpha \mapsto \alpha_M$ induces an isomorphism of outer automorphism groups. The map is well-defined and injective by definition since $\alpha$ is inner as an automorphism of embedded data precisely when $\alpha_M$ is inner. For any automorphism $\alpha_M$ of $(H_M, s, \eta|_{\widehat{H_M}})$, we get an $m \in \widehat{M}$ such that $\Int(m) \circ \eta|_{\widehat{H_M}} = \eta|_{\widehat{H_M}} \circ \widehat{\alpha_M}$. Then it is easy to check that if $\gamma \cdot \eta = \Int(g_{\gamma}) \circ \eta$ and $\widehat{T} \subset \widehat{H_M}$ is $\Gamma_F$-invariant, then $\Int(m^{-1}g^{-1}_{\gamma}\gamma(m) g_{\gamma})$ preserves $\eta(s)$ and  acts as the identity on $\eta(\widehat{T})$. Hence $m^{-1}g^{-1}_{\gamma}\gamma(m)g_{\gamma} \in \eta(\widehat{H})$ and it follows that $\eta^{-1} \circ \Int(m) \circ \eta: \widehat{H} \to \widehat{H}$ has $\Gamma_F$-stable conjugacy class and hence by Fact \ref{equivarenum} we get an automorphism $\alpha: H \to H$ well-defined up to conjugacy lifting $\alpha_M$.
\end{proof}

\subsection{Parametrization of endoscopic data}
Let $(H,s,\eta)$ be a refined  endoscopic datum of $G$, fix a Levi subgroup $M$, and choices of $\Gamma_F$-invariant splittings $(\widehat{T}, \widehat{B}, \{X^{\widehat{G}}_{\alpha}\})$ and $(\widehat{T_H}, \widehat{B_H}, \{ X^{\widehat{H}}_{\alpha}\})$ of $\widehat{G}$ and $\widehat{H}$ respectively. We assume that  $\eta(\widehat{T_H})=\widehat{T}$ and that $\eta(\widehat{B_H}) \subset \widehat{B}$. Our goal is to parametrize the set of isomorphism classes of embedded endoscopic data in the fiber $Y^{-1}([(H,s, \eta)])$. 

To do so, we understand  which endoscopic data $(H', s', \eta')$ isomorphic to $(H, s, \eta)$ can be restricted to the pre-image of $\widehat{M}$ to give a refined endoscopic datum for $M$. If this is possible, we say that $(H', s', \eta')$ ``restricts'' to give an embedded datum $(H', H'_M, s', \eta')$. We make a series of reductions on the endoscopic data $(H', s', \eta')$ that we must consider. We use several ideas we learned from \cite{Hir1} and \cite{Xu1}.

First, since any such $H$ and $H'$ are isomorphic, it suffices to fix a group $H$ and only consider endoscopic data of the form $(H, s', \eta')$ in the same isomorphism class. Now pick a quasisplit inner form $G^*$ and an inner twist $\Psi: G^* \to G$. Then $\Psi$ induces an identification of dual groups by Lemma \ref{itdg} and hence it suffices to consider the parametrization for $G^*$ and $M^*$.

Second, fix $F$-splittings  $(T_H, B_H, \{X^H_{\alpha}\})$ and $(T, B, \{X^{G^*}_{\alpha}\})$ of $H$ and  $G^*$, noting that $T_H$ and $T$ will have maximal split rank in their respective groups. Then the following lemma shows that we may assume $\eta'( \widehat{T_H})=\widehat{T}$ and furthermore, that if $\alpha: H \to H$ induces an isomorphism of embedded data $(H, {H_{M^*}}_1, s_1, \eta_1)$ and $(H, {H_{M^*}}_2, s_2, \eta_2)$ such that for $i=1,2$ we have $\eta_i(\widehat{T_H})=\widehat{T}$, then we may assume $\alpha(T_H)=T_H$ and that there exists $g \in \widehat{G^*}$ such that $\Int(g)(\widehat{T})=\widehat{T}$ and the diagram
\begin{equation*}
\begin{tikzcd}
\widehat{H} \arrow[r, "\eta_1"] \arrow[d, swap,  "\widehat{\alpha}"] & \widehat{G^*} \arrow[d, "\Int(g)"] \\
\widehat{H} \arrow[r, "\eta_2"] & \widehat{G^*}
\end{tikzcd}    
\end{equation*}
commutes.
\begin{lemma}{\label{endweyl}}
Let $(H, H_{M^*}, s, \eta)$ be an embedded endoscopic datum. Let $T_H, T, \widehat{T_H}, \widehat{T}$ be as above. Then we have the following.
\begin{enumerate}
    \item The datum $(H, H_{M^*}, s, \eta)$ is isomorphic to an embedded datum $(H, H'_{M^*}, s', \eta')$ such that $T_H \subset H'_{M^*}$ and $\eta'(\widehat{T_H})=\widehat{T}$.\\
    \item If $(H, H'_{M^*}, s', \eta')$ is isomorphic to $(H, H_{M^*}, s, \eta)$ and both $H_{M^*}$ and $H'_{M^*}$ contain $T_H$ and $\eta(\widehat{T_H})=\eta'(\widehat{T_H})=\widehat{T}$, then there is an automorphism $\alpha: H \to H$ with dual $\widehat{\alpha}$ and an element $g \in \widehat{G^*}$ such that $\alpha$ induces an isomorphism of embedded endoscopic data, $\alpha(T_H)=T_H$, $\widehat{\alpha}(\widehat{T_H})=\widehat{T_H}$ and $\mathrm{Int}(g)(\widehat{T})=\widehat{T}$, and the following diagram commutes.
    \begin{equation}
    \begin{tikzcd}
    \widehat{H} \arrow[r, "\eta"] \arrow[d, swap, "\widehat{\alpha}"] & \widehat{G^*} \arrow[d, "\mathrm{Int}(g)"] \\
    \widehat{H} \arrow[r, swap,  "\eta'"]& \widehat{G^*}
    \end{tikzcd}
    \end{equation}
\end{enumerate}
\end{lemma}
\begin{proof}
For the first assertion, fix a maximal torus $T_1 \subset H_{M^*}$ of maximal split rank and a Borel subgroup $B_1$ of $H$ such that $H_{M^*}$ is a Levi subgroup of a parabolic subgroup containing $B_1$. Similarly, fix a Borel subgroup $B_H$ containing $T_H$ and defined over $F$. Then the pair $(T_H, B_H)$ is conjugate to $(T_1, B_1)$ in $H$ by some inner automorphism $\beta$ defined over $F$. Let $H'_{M^*}=\beta^{-1}(H_M)$. Then let $\widehat{\beta}: \widehat{H} \to \widehat{H}$ be a dual morphism, chosen so that $\widehat{\beta}$ is $\Gamma_F$-equivariant and so that $\widehat{\beta}(\widehat{H_{M^*}})=\widehat{H'_{M^*}}$ (where both $\widehat{H_{M^*}}$ and $\widehat{H'_{M^*}}$ are defined via the fixed splitting of $\widehat{H})$. Now consider the maximal torus of $\widehat{M^*}$ given by $(\eta \circ \widehat{\beta}^{-1})(\widehat{T_H})$. Since all maximal tori are conjugate, there exists an $m \in \widehat{M^*}$ such that $\mathrm{Int}(m)(\eta(\widehat{\beta}^{-1}(\widehat{T_H})))=\widehat{T}$. Then consider the tuple $(H, H'_{M^*}, \widehat{\beta}(s), \mathrm{Int}(m) \circ \eta \circ \widehat{\beta}^{-1})$. This is an embedded endoscopic datum and isomorphic to $(H, H_M, s, \eta)$ by $\beta$.

To prove the second assertion, we pick an automorphism $\alpha': H \to H$ inducing an isomorphism of $(H, H'_{M^*}, s', \eta')$ and $(H,H_{M^*},s, \eta)$. Then $T'_H := \alpha'(T_H)$ is some maximal torus of maximal split rank in $H_{M^*}$ so there exists an $h \in H_{M^*}(F)$ so that $\Int(h)(T'_H)=T_H$ and hence $\alpha := \mathrm{Int}(h) \circ \alpha'$ satisfies $\alpha(T_H)=T_H$. Now choose an automorphism $\widehat{\alpha}$ inducing a $\Gamma_F$-equivariant automorphism of $\widehat{H}$ dual to $\alpha$ and such that $\widehat{\alpha}(\widehat{T_H})=\widehat{T_H}$. Then $\alpha$ gives the desired isomorphism. Since $\widehat{\alpha}$ will be conjugate in $\widehat{H}$ to $\widehat{\alpha'}$, we can pick a $g \in \widehat{G^*}$ such that the following diagram commutes:
 \begin{equation}
    \begin{tikzcd}
    \widehat{H} \arrow[r, "\eta"] \arrow[d, swap, "\widehat{\alpha}"] & \widehat{G^*} \arrow[d, "\mathrm{Int}(g)"] \\
    \widehat{H} \arrow[r, swap,  "\eta'"]& \widehat{G^*}.
    \end{tikzcd}
\end{equation}
This then implies that $\mathrm{Int}(g)(\widehat{T})=\widehat{T}$ and $\alpha$ induces an isomorphism of embedded endoscopic data as required.
\end{proof}

Third, fix a refined endoscopic datum $(H, s_1, \eta_1)$ and suppose we have an embedded datum $(H, {H_{M^*}}_2,  s_2, \eta_2)$ such that $\eta_i(\widehat{T_H})=\widehat{T}$ and we have an isomorphism of endoscopic data $\alpha: H \to H$ between $(H, s_1, \eta_1)$ and $(H, s_2, \eta_2)$. We can and do choose $\alpha$ and $g \in \widehat{G^*}$ such that $\alpha(T_H)=T_H$, $\Int(g)(\widehat{T})=\widehat{T}$, and $\Int(g) \circ \eta_1 = \eta_2 \circ \widehat{\alpha}$. We now note that $(H, \alpha({H_{M^*}}_2), \widehat{\alpha}^{-1}(s_2), \eta_2 \circ \widehat{\alpha})$ is isomorphic via $\alpha$ to $(H, {H_{M^*}}_2, s_2, \eta_2) $ as embedded endoscopic data. Then the former embedded endoscopic datum is equal to $(H, \alpha({H_{M^*}}_2), s_1, \Int(g) \circ \eta_1)$. In particular, to compute $Y^{-1}[(H, s, \eta)]$, it suffices to fix a datum $(H,s, \eta)$ and compute the set of endoscopic data of the form $(H, s, \Int(g) \circ \eta)$ for $g \in N_{\widehat{G^*}}(\widehat{T})$ which restrict to give embedded endoscopic data.

We now record the following trivial fact. 
\begin{lemma}{\label{centermlem}}
Suppose that $(H, s, \Int(g) \circ \eta)$ restricts to an embedded endoscopic datum  $(H, H_{M^*}, s, \Int(g) \circ \eta)$ and satisfies $(\Int(g) \circ \eta)(\widehat{T_H}) = \widehat{T}$. Then
\begin{equation*}
    (Z(\widehat{M^*})^{\Gamma_F})^{\circ} \subset (\Int(g) \circ \eta)((\widehat{T_H}^{\Gamma_F})^{\circ}).
\end{equation*}
\end{lemma}
\begin{proof}
We note the endoscopic datum $(H_{M^*}, s, \Int(g) \circ \eta|_{\widehat{H_{M^*}}})$ induces a $\Gamma_F$-equivariant embedding $(Z(\widehat{M^*})^{\Gamma_F})^{\circ} \hookrightarrow (Z(\widehat{H_{M^*}})^{\Gamma_F})^{\circ} \subset (\widehat{T_H}^{\Gamma_F})^{\circ}$.
\end{proof}

Let $\Aut_{T_H}(H,s,\eta)$ denote the set of automorphisms $\alpha: H \to H$ of $(H,s, \eta)$ that stabilize $T_H$. We construct a homomorphism $\Aut_{T_H}(H,s,\eta) \hookrightarrow W(\widehat{T}, \widehat{G^*})$ that is canonical up to our fixed datum $(H,s,\eta)$ and our chosen splittings. Our fixed choice of splittings give an identification of $W(T_H, H)^{\Gamma_F}$ and $W(\widehat{T_H}, \widehat{H})^{\Gamma_F}$. Let  $\Aut_{T_H}(H,s,\eta)$ denote the set of automorphisms $\alpha: H \to H$ of $(H,s,\eta)$ that stabilize $T_H$. Then we can uniquely write $\alpha=\alpha'\circ w$ such that $w \in W(T_H, H)^{\Gamma_F}$ and $\alpha'$ preserves the chosen splitting of $H$. Then we construct $\widehat{\alpha}: \widehat{H} \to \widehat{H}$ to be equal to $w' \circ \widehat{\alpha'}$ whereby $w' \in W(\widehat{T_H}, \widehat{H})^{\Gamma_F}$ is the element corresponding to $w$ and $\widehat{\alpha'}$ is a dual of $\alpha'$ chosen so that it preserves our fixed splitting of $\widehat{H}$. Now, for each $\widehat{\alpha}$, we choose $g \in N_{\widehat{G^*}}(\widehat{T})$ such that the following diagram commutes:
\begin{equation*}
\begin{tikzcd}
\widehat{H} \arrow[r, "\eta"] \arrow[d, swap, "\widehat{\alpha}"]& \widehat{G} \arrow[d, "\Int(g)"] \\
\widehat{H} \arrow[r, "\eta"] & \widehat{G}
\end{tikzcd}    
\end{equation*}
Two different choices $g', g \in N_{\widehat{G^*}}(\widehat{T})$ in the above diagram differ by an element of $\eta(Z(\widehat{H}))$ and hence induce the same element of $W(\widehat{T}, \widehat{G^*})$.  Hence we have constructed our embedding $\Aut_{T_H}(H,s,\eta) \hookrightarrow W(\widehat{T}, \widehat{G^*})$. The map $\eta$ also induces an embedding of $W(\widehat{T_H}, \widehat{H})$ into $W(\widehat{T}, \widehat{G^*})$. We let $\Aut(\widehat{H}, \eta)$ be the subgroup of $W(\widehat{T}, \widehat{G^*})$ generated by the images of $W(\widehat{T_H}, \widehat{H})$ and $\Aut_{T_H}(H,s,\eta)$. Note that the quotient of  $\Aut(\widehat{H}, \eta)$ by the image of $W(\widehat{T_H}, \widehat{H})$ in $W(\widehat{T}, \widehat{G^*})$ is identified with $\Out_r(H,s,\eta)$.

We now have the following definition.
\begin{definition}
We define $W(M^*, H) \subset W(\widehat{T}, \widehat{G^*})$ as the elements $w$ such that $(w \circ \eta)^{-1}( (Z(\widehat{M^*})^{\Gamma_F})^{\circ})$ is pointwise $\Gamma_F$-invariant up to conjugacy in $\widehat{H}$. In other words, for each $\gamma \in \Gamma_F$ there is an $h_{\gamma} \in \widehat{H}$ so that $\Int(h_{\gamma}) \circ \gamma$ centralizes $(w \circ \eta)^{-1}( (Z(\widehat{M^*})^{\Gamma_F})^{\circ})$. Equivalently, $Z_{\LH}((w \circ \eta)^{-1}( (Z(\widehat{M^*})^{\Gamma_F})^{\circ}))$ is a \emph{full} subgroup of $\LH$ (its projection to $W_F$ is surjective). 
\end{definition}

We will now see that $\Aut(\widehat{H}, \eta)$ acts on $W(M^*, H)$ on the right and $W(\widehat{T}, \widehat{M^*})$ acts on the left. Indeed, if $w_1 \in W(\widehat{T}, \widehat{G^*})$ is in the image of $w_h \in W(\widehat{T_H}, \widehat{H})$ and $w_2 \in W(\widehat{T}, \widehat{M^*})$ and $w \in W(M^*, H)$, then 
\begin{equation*}
    (w_2ww_1 \circ \eta)^{-1}( (Z(\widehat{M^*})^{\Gamma_F})^{\circ}) = (\eta^{-1} \circ w^{-1}_1w^{-1})((Z(\widehat{M^*})^{\Gamma_F})^{\circ})=(w^{-1}_h \circ (w \circ \eta)^{-1})((Z(\widehat{M^*})^{\Gamma_F})^{\circ}).
\end{equation*}
If the centralizer of $(w \circ \eta)^{-1}((Z(\widehat{M^*})^{\Gamma_F})^{\circ})$ in $\LH$ is full then so is that of $w^{-1}_h \circ (w \circ \eta)^{-1}((Z(\widehat{M^*})^{\Gamma_F})^{\circ})$. Finally, to show that the image of $\Aut_{T_H}(H,s,\eta)$ in $W(\widehat{T}, \widehat{G^*})$ acts on $W(M^*, H)$, we observe that if $\widehat{\alpha}$ is the dual of $\alpha \in \Aut_{T_H}(H,s, \eta)$ as constructed above, then $\widehat{\alpha}$ is $\Gamma_F$-equivariant and so if $\Int(h_{\gamma})(\gamma(z))=z$ for all $z \in (w \circ \eta)^{-1}((Z(\widehat{M^*})^{\Gamma_F})^{\circ})$, then $\Int(\widehat{\alpha}(h_{\gamma})) \gamma (\widehat{\alpha}(z))= \widehat{\alpha}(z)$ for all $z$.

We are now ready for the key result. 

\begin{proposition}{\label{WMHlem}}
Let $(H,s,\eta)$ be an endoscopic datum for $G$ and fix maximal tori $T_H \subset H$ and $T \subset G^*$ of maximal split rank in their respective groups. We assume that $\eta(\widehat{T_H})=\widehat{T}$. Let $g \in N_{\widehat{G^*}}(\widehat{T})$. Then $g$ projects to an element of $W(M^*,H)$ if and only if there exists an $h \in N_{\widehat{G^*}}(\widehat{T})$ that projects to an element of the image of $W(\widehat{T_H}, \widehat{H})$ such that $(H, s, \Int(gh) \circ \eta)$ restricts to give an embedded endoscopic datum.
\end{proposition}
\begin{proof}
Suppose first that $g, h \in N_{G^*}(\widehat{T})$ are such that $h$ projects to an element in the image of $W(\widehat{T_H}, \widehat{H})$ and $(H, s, \Int(gh) \circ \eta)$ restricts to an embedded endoscopic datum $(H, H_{M^*}, s, \Int(gh) \circ \eta)$. Then by Lemma \ref{centermlem}, it follows that $gh \in W(M^*, H)$ and therefore $g \in W(M^*, H)$.

We now prove the converse. Suppose that $(H, s, \Int(g) \circ \eta)$ is as above and that $g$ projects to an element of $W(M^*,H)$. Define $Z' := (\Int(g) \circ \eta)^{-1}((Z(\widehat{M^*})^{\Gamma_F})^{\circ})$ and $\widehat{H_{M^*}}' := Z_{\widehat{H}}(Z')$. By Lemma \ref{endlevi}, $\widehat{H_{M^*}}'$ is a Levi subgroup of $\widehat{H}$. By \cite[Lemma 3.5]{Bor2}, $Z_{^LH}(Z')$ is a Levi subgroup of some parabolic subgroup of $\LH$. Since all parabolic subgroups of $\LH$ are conjugate in $\widehat{H}$ to a standard parabolic subgroup relative to $\widehat{B_H}$, there is some $h' \in \widehat{H}$ such that $Z := \Int(h')(Z')$ satisfies $Z_{\LH}(Z)$ is a standard Levi subgroup of $\LH$, which we denote by $\widehat{H_{M^*}} \rtimes W_F$. We may assume that $h' \in N_{\widehat{H}}(\widehat{T_H})$.

We claim that $Z$ is $\Gamma_F$-invariant. To prove this, pick $\gamma \in \Gamma_F$ and find a $w \in W_F$ that acts as $\gamma$ on $\widehat{H}$. Then observe that $(1,w) \in Z_{\LH}(Z)$ which shows that the action of $w$ and therefore $\gamma$ on $Z$ is trivial. 

Let $h= \eta({h'}^{-1})$. Then we want to show that $(H, s, \Int(gh) \circ \eta)$ restricts to give an embedded endoscopic datum. We have that $\widehat{H_{M^*}} = (\Int(gh) \circ \eta)^{-1}(\widehat{M^*})$. Since $\widehat{H_{M^*}}$ is a $\Gamma_F$-stable standard Levi subgroup, there is a standard Levi subgroup $H_{M^*} \subset H$ relative to $B_H$ whose coroots correspond to the roots of $\widehat{H_{M^*}}$. We can now define a candidate embedded endoscopic datum $(H, H_{M^*}, s, \Int(gh) \circ \eta)$. To check this is an embedded endoscopic datum, it suffices to show that the conjugacy class of $\Int(gh) \circ \eta$ is $\Gamma_F$-invariant in $\widehat{M^*}$. To do so, pick $\gamma \in \Gamma_F$ and choose $w \in W_F$ that acts as $\gamma$ on $\widehat{H}$ and $\widehat{G^*}$. Then let $(g_w, w) \in \mc{H} \subset \LG$ (defined relative to $\Int(gh) \circ \eta$) be such that $\Int((g_w, w)) \circ (\Int(gh) \circ \eta) = (\Int(gh) \circ \eta) \circ \Int((1,w))$ on $\widehat{H}$. Since $(1,w)$ commutes with $Z$, it follows $(g_w, w)$ commutes with $(\Int(gh) \circ \eta)(Z)= (Z(\widehat{M^*})^{\Gamma_F})^{\circ}$. Hence $(g_w, w) \in \LM^*$ so $g_w \in \widehat{M^*}$. Finally, we have $\Int(\gamma^{-1}(g_w)) \circ (
\Int(gh) \circ \eta) = \gamma^{-1} \cdot (\Int(gh) \circ \eta)$ as desired.
\end{proof}

To complete our parametrization of the fiber $Y^{-1}([H,s,\eta])$ it remains to understand which embedded data $(H, H_{M^*}, s, \Int(g) \circ \eta)$ of the above reduced form are isomorphic. If $(H, s, \Int(g) \circ \eta)$ restricts to an embedded endoscopic datum, then $g$ determines a double coset in $W(\widehat{T}, \widehat{M^*}) \setminus W(M^*, H) / \Aut(\widehat{H}, \eta)$. We show that this induces a bijection between $Y^{-1}([H,s,\eta])$ and $W(\widehat{T}, \widehat{M^*}) \setminus W(M^*, H) / \Aut(\widehat{H}, \eta)$.

We need to show that if $(H, s, \Int(g) \circ \eta)$ and $(H, s, \Int(g') \circ \eta)$ restrict to isomorphic embedded data, then they correspond to the same double coset and conversely, that if $(H,s, \Int(g) \circ \eta)$ and $(H, s, \Int(g') \circ \eta)$ restrict to embedded data and yield the same double coset, then these embedded data are isomorphic.

Suppose $(H, s, \Int(g) \circ \eta)$ and $(H, s, \Int(g') \circ \eta)$ restrict to isomorphic embedded data and $\alpha: H \to H$ induces an isomorphism such that $\alpha(T_H)=T_H$. Then we can find $g_{\alpha} \in N_{\widehat{G^*}}(\widehat{T})$ projecting to an element of $W(\widehat{T}, \widehat{M^*})$ and such that  $\Int(g) \circ \eta \circ \widehat{\alpha} = \Int(g_{\alpha}) \circ \Int(g') \circ \eta$. Additionally, since $\Int(g) \circ \eta$ and $\Int(g') \circ \eta$ and $\eta$ are all conjugate in $\widehat{G^*}$ and $\alpha$ induces an isomorphism of $(H, s, \Int(g) \circ \eta)$ and $(H, s, \Int(g') \circ \eta)$, it follows that $\alpha$ also induces an automorphism of $(H,s, \eta)$ and hence that we can find $g_1 \in N_{\widehat{G^*}}(\widehat{T})$ such that $g_1$ projects to an element of $W(\widehat{H}, \eta)$ and $\eta \circ \widehat{\alpha} = \Int(g_1) \circ \eta$. Putting these equations together, we have
\begin{equation*}
    \Int(g_{\alpha}g') \circ \eta = \Int(gg_1) \circ \eta.
\end{equation*}
Taking projections to $W(\widehat{T}, \widehat{G^*})$, we get that $g_{\alpha}g'=gg_1 \in W(\widehat{T}, \widehat{G^*})$ and hence $g$ and $g'$ are in the same $W(\widehat{T}, \widehat{M^*}) \setminus W(M^*, H) / \Aut(\widehat{H}, \eta)$ double coset as desired.

Now assume that $(H, s, \Int(g) \circ \eta)$ and $(H, s, \Int(g') \circ \eta)$ restrict to embedded endoscopic data and yield the same double coset. Then we can write $g'=mgh$ where $h$ induces an element of $\Aut(\widehat{H}, \eta)$ and $m$ induces an element of $W(\widehat{T}, \widehat{M^*})$. Clearly $(H, s, \Int(gh) \circ \eta)$ and $(H,s, \Int(mgh) \circ \eta)$ restrict to give isomorphic embedded endoscopic data so it suffices to show that if $(H, s, \Int(gh) \circ \eta)$ and $(H, s, \Int(g) \circ \eta)$ restrict to give embedded endoscopic data with corresponding Levi subgroups $H'_{M^*}, H_{M^*} \subset H$ respectively, then these data are isomorphic. Choose an automorphism $\widehat{\alpha}: \widehat{H} \to \widehat{H}$ such that $\Int(h) \circ \eta = \eta \circ \widehat{\alpha}$. Then since $\widehat{\alpha}$ is $\Gamma_F$-equivariant up to conjugacy, we can pick a dual $\alpha': H \to H$ defined over $F$ such that $\alpha'(H_{M^*})$ and $H'_{M^*}$ are conjugate in $H_{\ov{F}}$. We want to show we can pick a dual $\alpha$ of $\widehat{\alpha}$ so that $\alpha(H_{M^*})=H'_{M^*}$. It suffices to show that $\alpha'(H_{M^*})$ and $H'_{M^*}$ are conjugate over $F$. This follows from \cite[Theorem A]{Sol1}.

Hence we have proven:
\begin{proposition}{\label{endparam}}
Let $(H,s, \eta)$ be a refined endoscopic datum of $G$. Then $Y^{-1}([H,s, \eta])$ is parametrized by elements of  $W(\widehat{T}, \widehat{M^*}) \backslash W(M^*,H) / W(\widehat{H}, \eta)$.
\end{proposition}

We wish to compare this parametization to the works of Xu \cite[C.4]{Xu1}and Hiraga \cite[5.6]{Hir1} on the compatibility of endoscopic transfer and Jacquet modules. These authors use a subtly different notion of isomorphism of endoscopic data where all isomorphisms are inner. In particular, we have the following definition.
\begin{definition}
Fix an endoscopic datum $(H, s, \eta)$ of $G$. We define the set $\mc{E}^{e}(M,G;H)$ to be the set isomorphism classes of embedded endoscopic data whose image under 
\begin{equation}
Y^e: \mc{E}^{e}(M,G) \to \mc{E}^r(G)
\end{equation}
is $[H,s,\eta]$. We define the set of \emph{inner classes of embedded endoscopic data} relative to $H$, $\mc{E}^{i}(M, G ; H)$ to be the set of embedded endoscopic data for the group $H$ whose isomorphism class lies in $\mc{E}^{e}(M,G;H)$ and where two such data are equivalent if they are isomorphic by an  inner isomorphism $\alpha$ of the group $H$ inducing an isomorphism of embedded endoscopic data. Note that $\alpha$ need not induce an inner isomorphism of embedded endoscopic data.
\end{definition}
By similar reasoning to the above argument, the number of isomorphism classes of $\mc{E}^{i}(M,G;H)$ is $W(\widehat{T}, \widehat{M^*}) \backslash W(M^*,H) / W(\widehat{T_H}, \widehat{H})$.

We now count the number of  isomorphism classes in $\mc{E}^{i}(M,G;H)$ mapping to a given isomorphism class of $\mc{E}^{e}(M,G;H)$ under the natural map.
\begin{proposition}{\label{fibercard}}
Pick $[H, H_M, s, \eta] \in \mc{E}^{e}(M,G;H)$. Then the number of isomorphism classes in $\mc{E}^{i}(M,G; H)$ mapping to $[H, H_M, s, \eta]$ under the natural map is $|\Out_r(H,s , \eta)/\Out_r(H_M,s, \eta)|$.
\end{proposition}
\begin{proof}
To prove the result, we fix an embedded datum $(H, H_M, s, \eta)$ and construct parametrizations of $\mc{E}^e(M,G;H)$ and $\mc{E}^i(M,G;H)$ as above, relative to the datum $(H,s,\eta)$. Then we simply need to compute how many $W(\widehat{T}, \widehat{M^*}) \times W(\widehat{T_H}, \widehat{H})$-double cosets are contained in the identity double coset of $W(\widehat{T}, \widehat{M^*}) \times \Aut(\widehat{H}, \eta)$. 

Suppose $w_{\alpha} \in W(\widehat{H}, \eta)$ so that $1, w_{\alpha}$ give the same $W(\widehat{T}, \widehat{M^*}) \times \Aut(\widehat{H}, \eta)$ double coset but potentially different $W(\widehat{T}, \widehat{M^*}) \times W(\widehat{T_H}, \widehat{H})$ double cosets. We claim that these $W(\widehat{T}, \widehat{M^*}) \times W(\widehat{T_H}, \widehat{H})$ double cosets contain the same number of distinct $W(\widehat{T}, \widehat{M^*})$-cosets. Indeed the number of $W(\widehat{T}, \widehat{M^*})$ cosets are respectively
\begin{equation*}
    W(\widehat{T_H}, \widehat{H})/ (W(\widehat{T_H}, \widehat{H}) \cap W(\widehat{T}, \widehat{M^*}) ) \text{ and } W(\widehat{T_H}, \widehat{H})/ ( w_{\alpha}W(\widehat{T_H}, \widehat{H})w^{-1}_{\alpha} \cap W(\widehat{T}, \widehat{M^*}))
\end{equation*}
but these are the same since $w_{\alpha}$ stabilizes $W(\widehat{T_H}, \widehat{H})$.

Now, the identity $W(\widehat{T}, \widehat{M^*}) \times \Aut(\widehat{H}, \eta)$ and $W(\widehat{T}, \widehat{M^*}) \times W(\widehat{T_H}, \widehat{H})$ double cosets consist of
\begin{equation*}
    |\Aut(\widehat{H}, \eta)/(W(\widehat{T}, \widehat{M^*}) \cap  \Aut(\widehat{H},\eta))| \text{ and } |W(\widehat{T_H}, \widehat{H})/(W(\widehat{T}, \widehat{M^*}) \cap W(\widehat{T_H}, \widehat{H}))|
\end{equation*}
many $W(\widehat{T}, \widehat{M^*})$-cosets respectively. Hence the number of $W(\widehat{T}, \widehat{M^*}) \times W(\widehat{T_H}, \widehat{H})$ double cosets in the identity $W(\widehat{T}, \widehat{M^*}) \times \Aut(\widehat{H},\eta)$-double coset is
\begin{equation*}
    |\Aut(\widehat{H}, \eta)/W(\widehat{T_H}, \widehat{H})| \cdot |W(\widehat{T}, \widehat{M^*}) \cap W(\widehat{T_H}, \widehat{H}) / W(\widehat{T}, \widehat{M^*}) \cap \Aut(\widehat{H},\eta)|.
\end{equation*}
The first factor is in bijection with $|\Out_r(H,s,\eta)|$. Hence, to complete the proof, we need only show that 
\begin{equation*}
   |W(\widehat{T}, \widehat{M^*}) \cap W(\widehat{T_H}, \widehat{H})| / |W(\widehat{T}, \widehat{M^*}) \cap \Aut(\widehat{H},\eta)| = 1 / |\Out_r(H_M,s,\eta)|. 
\end{equation*}

First, we note (identifying  $W(\widehat{T_H}, \widehat{H_M})$ and $W(\widehat{T_H}, \widehat{H})$ with their image under $\eta$) that $W(\widehat{T_H}, \widehat{H_{M^*}})=W(\widehat{T_H}, \widehat{H}) \cap W(\widehat{T}, \widehat{M^*})$.  Similarly, we claim that $\Aut(\widehat{H},\eta) \cap W(\widehat{T}, \widehat{M^*})$ can be identified with $\Aut(\widehat{H_{M^*}},\eta)$. Indeed if we take a lift $g_{\alpha}$ of an element in the first set, then $g_{\alpha}$ normalizes $\eta(\widehat{H})$ and $\widehat{M^*}$ hence also their intersection. Conversely, if $m_{\alpha} \in \widehat{M^*} \subset \widehat{G}$ is a lift of an element in the second set, then $m_{\alpha}$ normalizes $\widehat{H}$ because  by construction of $\Aut(\widehat{H_{M^*}},\eta)$, the element $m_{\alpha}$ must commute with $s$. This finishes the proof.
\end{proof}

\subsection{Semisimple conjugacy classes and endoscopy}

In this section we show how to relate refined and embedded endoscopic data to certain constructions involving semisimple conjugacy classes. To simplify the discussion, we assume throughout that $\G_{\der}$ is simply connected. We begin by showing an analogue of \cite[lemma 9.7]{Kot6}. We need the following definitions.

\begin{definition}
Let $F$ be a local or global field. We define $\mcSS^r(\G)$ to be equivalence classes of pairs $(\gamma_0, \lambda)$ such that $\gamma_0 \in \G(F)$ is a semisimple element and $\lambda \in Z(\widehat{I^G_{\gamma_0}})^{\Gamma_F}$ where we define $I^G_{\gamma_0} := Z_{\G}(\gamma_0)^0$. Two pairs  $(\gamma_0, \lambda), (\gamma'_0, \lambda')$ are said to be equivalent if $\gamma_0 \sim_{st} \gamma'_0$ (where $\sim_{st}$ denotes stable conjuguacy) and $\lambda$ corresponds to $\lambda'$ under the canonical isomorphism $Z(\widehat{I^G_{\gamma_0}})^{\Gamma_F} \cong Z(\widehat{I^G_{\gamma'_0}})^{\Gamma_F}$ \cite[\S 3]{Kot6}.
\end{definition}

\begin{definition}
We define $\mcEQ^r(\G)$ to be the set of equivalence classes of quadruples $(H, s, \eta, \gamma_H)$ such that $(H,s, \eta)$ is a refined endoscopic datum and $\gamma_H \in H(F)$ is $(\G,H)$-regular (see \cite[\S 2.3]{Shi3}) and transfers to $G$. Two quadruples $(H ,s, \eta, \gamma_H)$ and $(H', s', \eta', \gamma_{H'})$ are equivalent if there exists an isomorphism $\alpha: H \to H'$ inducing an isomorphism of refined endoscopic data and such that $\alpha(\gamma_H)$ is stably conjugate to $\gamma'_H$.
\end{definition}

We have a natural map 
\begin{equation}
 \mcEQ^r(\G) \to \mcSS^r(G),   
\end{equation}
defined as follows. Given $(H, s, \eta, \gamma_H) \in \mcEQ^r(\G)$ we let $\gamma_0 \in \G(F)$ be a transfer of $\gamma_H$. Then by the $(\G,H)$-regular condition, we have a natural $\Gamma_F$-equivariant isomorphism of $Z(\widehat{I^H_{\gamma_H}})$ and $Z(\widehat{I^G_{\gamma_0}})$. In particular, since $s \in Z(\widehat{H})^{\Gamma_F}$ it gives an element of $Z(\widehat{I^G_{\gamma_0}})^{\Gamma_F}$. If $(H', s', \eta', \gamma_{H'})$ is equivalent to $(H, s, \eta, \gamma_H)$, then one can check we get up to equivalence the same element of $\mcSS^r(\G)$.

We now have the following lemma which is highly analogous to \cite[Lemma 2.8]{Shi3} but note the differences coming from our use of refined endoscopic data.
\begin{lemma}{\label{SSeqEQ}}
The above defined map gives a bijection $\mcEQ^r(\G) \cong \mcSS^r(\G)$.
\end{lemma}
\begin{proof}
We largely follow \cite[Lemma 9.7]{Kot6} with a few differences coming from the fact that we consider refined endoscopic data.

We first prove the map is surjective. Given a pair $(\gamma_0, \lambda)$, we fix a maximal torus $T$ of $G$ such that $\gamma_0 \in T(F)$. Then $T$ is a maximal torus of $I^G_{\gamma_0}$ so we have a canonical $\Gamma_F$-equivariant embedding $Z(\widehat{I^G_{\gamma_0}}) \subset \widehat{T}$.  Hence by this inclusion and a choice of embedding $\widehat{T} \subset \widehat{G}$, the element $\lambda$ gives us an element $s \in \widehat{G}$. We define $\widehat{H} := Z_{\widehat{G}}(s)^0$.

We need to define a Galois action on $\widehat{H}$ so that we can consider it as the dual group of a quasisplit reductive group over $F$ (see Lemma \ref{qsbijlem}). Note that $\widehat{T}$ has an action of $\Gamma_F$ induced from the action of $\Gamma_F$ on $T_{\ov{F}}$. Since $\Gamma_F$ stabilizes the root system associated to $(T_{\ov{F}}, G)$, we have that $\Gamma_F$ also preserves the root system associated to $(\widehat{T}, \widehat{G})$. Then by the theory of centralizers (for instance see \cite[\S 2.2]{Hum1} )the roots of $\widehat{H}$ are precisely the roots $\alpha$ of $\widehat{G}$ such that $\alpha(s)=1$. In particular, since for each $\sigma \in \Gamma_F$, we have $\sigma(s)=s$, this implies that $\Gamma_F$ acts on the root system of $(\widehat{T}, \widehat{H})$. Now, fix a splitting $\Sigma$ of $\widehat{H}$ and this gives an action of $\Gamma_F$ on $\widehat{H}$ as desired. Define $\eta$ to be the natural embedding $\eta: \widehat{H} \subset \widehat{G}$. 

To show that $(H, s, \eta)$ is a refined endoscopic datum, it remains to check that the conjugacy class of $\eta$ is stable under the action of $\Gamma_F$. Extend the splitting $\Sigma$ of $\widehat{H}$ to a splitting $\Sigma_G$ of $\widehat{G}$. The action of $\Gamma_F$ on $\widehat{G}$ is determined by some other splitting $\Sigma'$ of $\widehat{G}$. Hence for each $\sigma \in \Gamma_F$, the actions of $\sigma$ on $\widehat{G}$ induced by $\Sigma_G$ and $\Sigma'$ differ by an element of $\mathrm{Inn}(\widehat{G})$ as desired. 

Since the conjugacy class of $\widehat{T} \subset \widehat{H}$ is $\Gamma_F$-stable, we can find an $F$-torus $T_H \subset H$ that transfers to $T \subset G$ and hence produce a $(G,H)$-regular $\gamma_H \in H(F)$ that transfers to $\gamma_0$ such that $(H, s, \eta, \gamma_H)$ maps to $(\gamma_0, \lambda)$.

Now we need to show the map is injective. Suppose that $(H', s', \eta', \gamma_{H'})$ and $(H, s, \eta, \gamma_H)$ map to $(\gamma_0, \lambda) \in \mcSS^r(\G)$. We wish to show $(H', s', \eta', \gamma_{H'})$ is equivalent to $(H, s, \eta, \gamma_H)$. Pick maximal tori $T_H \subset H$ and $T_{H'} \subset H'$ and $T \in G$ such that $\gamma_H \in T_H(F), \gamma_{H'} \in T_{H'}(F), \gamma_0 \in T(F)$. Then up to an isomorphism of endoscopic data, we can assume that $\eta(\widehat{T_H})=\widehat{T}=\eta'(\widehat{T_{H'}})$. Since $\eta(\widehat{H})$ and $\eta'(\widehat{H'})$ are determined by $\lambda$, we in fact have an isomorphism $\eta' \circ \eta^{-1}: \widehat{H'} \to \widehat{H}$. In particular if we pick $j: T_H \to G$ and $j': T_{H'} \to G$ in the canonical conjugacy classes of embeddings (for instance see \cite[Remark 2.6]{Shi3} for a description of this canonical conjugacy class), then we can choose $j, j'$ such that $j(T_H)=T=j'(T_{H'})$ and $j(\gamma_H)=\gamma=j'(\gamma_{H'})$. Since the roots and coroots of $H, H'$ are identified under $j^{-1} \circ j'$, we can extend it to an isomorphism over $\ov{F}$ of $H'$ and $H$ whose dual is given by  $\eta^{-1} \circ \eta'$. Our goal now is to show that the conjugacy class of $j^{-1} \circ j'$ is $\Gamma_F$-invariant, from which it will follow that since $H, H'$ are quasisplit, some $H(\ov{F})$-conjugate of $j^{-1} \circ j'$ is defined over $F$.  

Pick $\sigma \in \Gamma_F$. Since the conjugacy class of $j$ is $\Gamma_F$-invariant, there exists $g \in \G$ such that $\sigma(j)=\Int(g) \circ j$. Then we have $\sigma(j)(\gamma_H)=\gamma_0$ and so $g \in I^G_{\gamma_0}$. Now, since $\gamma_H$ is $(G, H)$-regular, we have an isomorphism $I^H_{\gamma_H}(\ov{F}) \cong  I^G_{\gamma_0}(\ov{F})$ and hence we can find an $h \in I^H_{\gamma_H}(\ov{F})$ such that $\sigma(j)= j \circ \Int(h)$. Repeating the argument for $j'$, we get that $j^{-1} \circ j'$ has $\Gamma_F$-invariant conjugacy class as desired. Hence we get an isomorphism $\alpha: H' \to H$ defined over $F$. By construction, $\alpha$ gives an isomorphism of refined endoscopic data. Furthermore, it is clear from the construction of $\alpha$ that $\gamma_H$ and $\alpha(\gamma_{H'})$ are $\ov{F}$-conjugate (hence also stably conjugate since we are assuming $\G_{\der}$ is simply connected). 
\end{proof}
We record the following compatibility.
\begin{corollary}{\label{endcomp}}
The following diagram commutes.
\begin{equation}
    \begin{tikzcd}
    \mcSS(G) \arrow[r, "\sim"] & \mcEQ(G)\\
    \mcSS^r(G) \arrow[u] \arrow[r, "\sim"] & \mcEQ^r(G) \arrow[u],
    \end{tikzcd}
\end{equation}
where $\mcSS(G), \mcEQ(G)$ are the analogues of the above definitions for non-refined data (as defined in \cite{Shi3}) and the vertical maps are the natural projections.

If we restrict to elliptic endoscopic data and $\gamma$ elliptic semisimple, then we get an analogous commutative diagram.
\end{corollary}

We now define an analogous version of the above for refined endoscopic data of a Levi subgroup $M \subset G$
\begin{definition}
We let $\mcEQ^r(M,G)$ denote equivalence classes of quadruples $(H_M, s_M, \eta_M, \gamma_{H_M})$ such that $(H_M, s_M, \eta_M)$ is a refined endoscopic datum for $M$ and $\gamma_{H_M}$ is an $(M, H_M)$-regular element of $H_M(F)$ that transfers to a $(G,M)$-regular element of $M(F)$. We say that two quadruples $(H_M, s_M, \eta_M, \gamma_{H_M})$ and $(H'_M, s'_M, \eta'_M, \gamma_{H'_M})$ are equivalent if there exists an isomorphism $\alpha: H_M \to H'_M$ inducing an isomorphism of refined endoscopic data and such that $\alpha(\gamma_{H_M}) \sim_{st} \gamma_{H'_M}$. Note in particular that $\mcEQ^r(M,G) \subset \mcEQ^r(M)$. We can make an analogue of the above definition for embedded endoscopic data which we denote $\mc{EQ}^e(M,G)$. Note that by proof of Proposition \ref{refemb}, we have  a natural identification $\mc{EQ}^r(M,G)=\mc{EQ}^e(M,G)$. 
\end{definition}
\begin{definition}
We let $\mcSS^r(M,G)$ denote the equivalence classes of pairs $(\gamma_0, \lambda)$ such that $\gamma_0 \in M(F)$ is a $(G,M)$-regular semisimple element and $\lambda \in Z(\widehat{I^M_{\gamma_0}})^{\Gamma_F}$. Two pairs $(\gamma_0, \lambda), (\gamma'_0, \lambda')$ are equivalent if $\gamma_0 \sim_{st} \gamma'_0$ and the canonical isomorphism of $I^M_{\gamma_0}$ and $I^M_{\gamma'_0}$ identifies $\lambda$ and $\lambda'$. In particular, $\mcSS^r(M,G) \subset \mcSS^r(M)$.
\end{definition}
The following corollary is clear.
\begin{corollary}
The bijection in Lemma \ref{SSeqEQ} restricts to give a bijection 
\begin{equation}
    \mcSS^r(M,G) \cong \mcEQ^r(M,G)=\mc{EQ}^e(M,G).
\end{equation}
\end{corollary}
Now, since the $\gamma_0$ for any $(\gamma_0, \lambda) \in \mcSS_F(M,G)$ is assumed to be $(G,M)$-regular, we have an equality $I^M_{\gamma_0}=I^G_{\gamma_0}$. Hence we get a map
\begin{equation}
    \mcSS^r(M,G) \to \mcSS^r(G)
\end{equation}
given by 
\begin{equation}
    (\gamma_0, \lambda) \mapsto (\gamma_0, \lambda).
\end{equation}
\begin{remark}
Note that this map need not be an injection because we might have $\gamma_0, \gamma'_0$ stably conjugate in $G$ but not $M$.
\end{remark}
The following lemma connects the various maps that we have defined.
\begin{lemma}{\label{SSEQMGGcomm}}
The following diagram is commutative
\begin{equation}
    \begin{tikzcd}
    \mcSS^r(G) \arrow[rr, "\sim"] & & \mcEQ^r(G)  \arrow[r] & \mc{E}^r(G)\\
    \mcSS^r(M,G) \arrow[u] \arrow[r, "\sim"] & \mcEQ^r(M,G) \arrow[r, equal] & \mcEQ^e(M,G) \arrow[r] \arrow[u] & \mc{E}^e(M) \arrow[u, swap, "Y^e"],\\
    \end{tikzcd}
\end{equation}
   where the left vertical map is the one we have just defined and the middle vertical map takes $(H, H_M, s, \eta, \gamma_{H_M})$ to $(H, s, \eta, \gamma_{H_M})$ where $\gamma_{H_M}$ is realized as an element of $H(F)$ via the inclusion $H_M \subset H$.
\end{lemma}
\begin{proof}

The commutativity of the right square in the above diagram is obvious, so it suffices to prove the commutativity of the left square. 

Fix $(\gamma_0, \lambda)$ projecting to an equivalence class in $\mcSS^r(M,G)$. The paths $SS^r(M,G) \to SS^r(G) \to \mc{EQ}^r(G)$ and $SS^r(M,G) \to \mc{EQ}^r(M,G) \to \mc{EQ}^e(M,G) \to \mc{EQ}^r(G)$ yield data $(H, H_M, s, \eta, \gamma_H)$ and $(H', H'_M, s', \eta', \gamma'_H)$ projecting to two potentially different classes in $\mc{EQ}^r(G)$. We need to show these classes are the same.

In the proof of Lemma \ref{SSeqEQ}, we chose a torus $T \subset G$ such that $\gamma_0 \in T(F)$. Since $\gamma_0 \in M(F)$, we can without loss of generality choose $T$ so that it is contained within $M$. Then the elements $s$ that we construct from $\lambda$ in the definition of the two horizontal maps labeled $\sim$ in the above diagram can be chosen to be the same. Hence we can assume that $s=s'$ and $\widehat{H} = \widehat{H'}$.  Since $\eta$ in both Construction \ref{Ymap} and Lemma \ref{SSeqEQ} agree, it follows that $\eta = \eta'$. 

We now show that the $\Gamma_F$-actions on $\widehat{H}$ are compatible, or equivalently that the two maps $\Gamma_F \to \Out(\widehat{H})$ are equal. The map $\Gamma_F \to \Out(\widehat{H})$ that we get from the maps $\mcSS^r(M,G) \to \mcSS^r(G) \to \mc{EQ}^r(G)$ is induced by the action of $\Gamma_F$ on $\widehat{T}$ and hence the roots of $\widehat{H}$. 

On the other hand, the lower horizontal map means that the action we get of $\Gamma_F$ on $\widehat{H_M}$ is also induced by the action of $\Gamma_F$ on $\widehat{T}$. Then in the proof of Construction \ref{Ymap}, we see that the resulting $\Gamma_F$ action on $\widehat{H}$ agrees with the action on $\widehat{H_M}$ up to conjugacy by $\widehat{H_M}$. In other words the two maps $\Gamma_F \to \Out(\widehat{H})$ that we get are the same. This shows that $H \cong H'$. Since $\gamma_H$ and $\gamma'_H$  both transfer to $\gamma_0$, we can choose an automorphism $\alpha$ of $(H,s, \eta)$ such that $\alpha(\gamma_H)$ and $\gamma_{H'}$ are stably conjugate. This completes the proof.
\end{proof}
\subsection{Some sets in local endoscopy}
We now study the refined analogue of the set $\mathcal{E}^{eff}(J_b, G; H)$ defined by Shin in \cite[\S 6]{Shi3}. First we need some notation.

Let $G$ be a quasisplit connected reductive group over $\mathbb{Q}_p$ and let $L=\widehat{\mathbb{Q}^{unr}_p}$. Consider the set $\mathbf{B}(G)$. In this section we consider this as the set of $\sigma$-conjugacy classes in $G(L)$ as in \cite[\S 3]{Shi3}.

Given $b \in \mathbf{B}(G)$, the slope $\overline{\nu}_b$ is an element of $(\mathrm{Int}(G) \setminus \mathrm{Hom}_{L}(\mathbb{D}, G))^{\langle \sigma \rangle}$. Given a lift $\tilde{b} \in G(L)$ of $b$, we have (for instance \cite[\S 4.3]{kot8}) a $\nu \in \mathrm{Hom}(\mathbb{D}, G)$ determined uniquely by the condition that there exists a positive integer $s$ and element $c \in G(L)$ such that 
    \begin{enumerate}
        \item $s\nu \in \mathrm{Hom}_L(\mathrm{G}_m, G)$\\
        \item $\mathrm{Int}(c) \circ s\nu$ is defined over the fixed field of $\sigma^s$ in $L$\\
        \item $cb \sigma(c)^{-1}\sigma(cb \sigma(c)^{-1}) \cdot ... \cdot \sigma^{s-1}(cb \sigma(c)^{-1})=c \cdot s \nu (p)c^{-1}$.
    \end{enumerate}
We say $\nu$ is decent if $c$ can be taken to be $1$ so that in particular, $s \nu$ is defined over the fixed field of $\sigma^s$. \begin{remark}{\label{decent}}
Note that for $\tilde{b} \in G(L)$, if we take a compatible element $c \in G(L)$ as above, we have that $cb\sigma(c)^{-1}$ is decent.
\end{remark}
\begin{lemma}{\label{decentreps}}
Let $S \subset G$ be $\Q_p$-split torus of maximal rank. Then given $b \in \mb{B}(G)$, we can pick a decent representative $\tilde{b}$ such that $\nu_{\tilde{b}}$ is decent and defined over $\Q_p$ and lies in the set of dominant rational cocharacters, $X_*(S)^+_{\Q}$.
\end{lemma}
\begin{proof}
By \cite[Proposition 6.2]{kot8}, since $G$ is quasisplit, we can pick $\tilde{b}$ so that $\nu$ is defined over $\mathbb{Q}_p$. Then by Remark \ref{decent}, $b$ can be picked to be decent.

Now the image of $s\nu$ is a split torus of $G$ and so lies in a maximal $\mathbb{Q}_p$-split torus of $G$. In particular if we fix a maximal $\mathbb{Q}_p$-split torus $S$ of $G$ and contained in a Borel subgroup $B$, we can pick a $c \in G(\mathbb{Q}_p)$ so that $\mathrm{Int}(c) \circ s\nu$ has image in $X_*(S)^+$. Let $b'=cb \sigma(c)^{-1}$. Then we have $\sigma(c)=c$ and so if we let $b'=cb \sigma(c)^{-1}$ then we have
\begin{equation}
    b'\sigma(b') \cdot ... \cdot \sigma^{s-1}(b')=s\nu_{b'}(p).
\end{equation}
\end{proof}

Now, we fix the standard Levi subgroup $M_b$ to be the centralizer of $\nu_{\tilde{b}}$ and consider the inner twist  $\psi: J_b \to M_b$ defined so that $\psi^{-1}\circ \sigma(\psi)= \Int(\tilde{b})$ (for $\sigma$ the lift of the Frobenius operator).

We recall the following definition of Shin (\cite[Definition 3.1]{Shi3}).
\begin{definition}
Let $G$ be a connected reductive group, $\nu: \mathbb{D} \to G$. Then we denote by $M$ the Levi subgroup which is the centralizer of $\nu$ in $G$. We can fix a maximal torus $T \subset M \subset G$ and view $\nu \in X_*(T)_{\mathbb{Q}}$. Choose a positive integer $s$ so that $s \nu \in X_*(T)$ and assume
\begin{equation}
    \forall \alpha \in R(G,T) \setminus R(M,T) \,\ \text{we have} \,\ v_p(\alpha(\nu(p))) \neq 0.
\end{equation}
Then we say $\gamma_0 \in M(\mathbb{Q}_p)$ is $\nu$-acceptable if for every $\alpha \in R(G,T) \setminus R(M,T)$, we have $\langle \alpha, \nu \rangle >0$ if and only if $\alpha(\gamma_0) \in \overline{\mathbb{Q}}^{\times}_p$ has positive $p$-adic valuation. If $J$ is an inner form of $M$, then we say that $\delta \in J(\mathbb{Q}_p)$ is $\nu$-acceptable if it transfers to a $\nu$-acceptable element of $M(\mathbb{Q}_p)$.
\end{definition}

Now we are ready to define the sets $\mc{SS}^r_{ef}(M_b,G)$  $\mcSS^r_{eff}(J_b,G)$.
\begin{definition}
Let $\mc{SS}^r_{ef}(M_b,G) \subset \mc{SS}^r(M_b,G)$ consist of those equivalence classes such that $\gamma_0$ can be chosen to be $\nu_b$-acceptable.

Let $\mcSS^r_{eff}(J_b, G)$ denote the equivalence classes of pairs $(\delta, \lambda)$ such that $\delta \in \J_b(\Q_p)$ and transfers to a $\nu_b$-acceptable, $(G,M_b)$-regular semisimple element of $M_b(\Q_p)$ and $\lambda \in Z(\widehat{I^{J_b}_{\delta}})^{\Gamma_{\Q_p}}$. We can identify $\mcSS^r_{eff}(J_b, G)$ with a subset of $\mcSS^r(M_b, G)$ via the map $(\delta, \lambda) \mapsto (\gamma_0, \lambda')$ where $\lambda'$ is the image of $\lambda$ under the canonical isomorphism $Z(\widehat{I^{J_b}_{\delta}}) \cong Z(\widehat{I^{M_b}_{\gamma}})$. 
\end{definition}
Then the bijection 
\begin{equation}
    \mcEQ^r(M_b, G) \cong \mcSS^r(M_b, G),
\end{equation}
identifies $\mc{SS}^r_{ef}(M_b, G)$ and $\mcSS^r_{eff}(J_b,G)$ with subsets of $\mcEQ^r(M_b, G)$ which we denote $\mc{EQ}^r_{ef}(M_b, G)$ and $\mcEQ^r_{eff}(J_b, G)$ respectively. We have a natural projection $\mcEQ^r(M_b, G) \to \mc{E}^r(M_b)$ and we denote the images of $\mc{EQ}^r_{ef}(M_b, G)$ and $\mcEQ^r_{eff}(J_b, G)$ by $\mc{E}^r_{ef}(M_b, G)$ and $\mc{E}^r_{eff}(J_b,G)$ respectively. 

Now, we define $\mc{E}^r_{eff}(J_b, G;H)$ for a fixed $\mc{H}^{\mf{e}}=[(H,s,\eta)] \in \mc{E}^r(G)$ to be
\begin{equation}{\label{jbghdef}}
    Y^{-1}(\mc{H}^{\mf{e}}) \cap \mc{E}^r_{eff}(J_b, G).
\end{equation}
Analogously, we can define $\mc{EQ}^e_{eff}(J_b,G)$ and $\mc{E}^e_{eff}(J_b, G)$ and we have $\mc{EQ}^e_{eff}(J_b,G) \cong \mc{EQ}^r_{eff}(J_b,G)$ and $\mc{E}^e_{eff}(J_b,G) \cong \mc{E}^r_{eff}(J_b,G)$. Then we define $\mc{E}^e_{eff}(J_b, G; H)$ to be $\mc{E}^e_{eff}(J_b, G) \cap \mc{E}^e(M_b, G, H)$. We also define $\mc{E}^i_{eff}(J_b,G;H)$ to be the pre-image of $\mc{E}^e_{eff}(J_b, G; H)$ under the projection $\mc{E}^i(M_b, G; H) \to \mc{E}^e(M_b, G ; H)$.

Finally, we can make the analogous definitions to those of the previous paragraph for $\mc{EQ}^r_{ef}(M_b,G)$.

We now reinterpret the set $\mathcal{E}^r_{eff}(J_b, G) $ in terms of transfer of maximal tori instead of transfer of semisimple conjugacy classes, as the former is in practice easier to work with. 
\begin{lemma}{\label{torendlem}}
For an element $(H_{M_b}, s_{M_b}, \eta_{M_b}) \in \mc{E}^r(M)$, we have $(H_{M_b}, s_{M_b}, \eta_{M_b}) \in \mathcal{E}^r_{eff}(J_b, G)$ if and only if there exist maximal tori $T_{H_{M_b}}, T_{M_b}, T_{J_b}$ defined over $\mathbb{Q}_p$ of $H_{M_b}, M_b, J_b$ respectively so that each torus transfers to the others.
\end{lemma}
\begin{proof}
Explicitly, $(H_{M_b}, s_{M_b}, \eta_{M_b}) \in \mathcal{E}^r_{eff}(J_b, G)$ means that there exists a semisimple $\gamma_{H_{M_b}} \in H_{M_b}(\mathbb{Q}_p)$ so that $\gamma_{H_{M_b}}$ is $(M_b, H_{M_b})$-regular and transfers to a $(G, M_b)$-regular and $\nu_b$-acceptable semisimple $\gamma_0 \in M_b(\mathbb{Q}_p)$ where also $\gamma_0$ transfers to some $\delta \in J_b(\mathbb{Q}_p)$. 

Let $I^{H_{M_b}}_{\gamma_{H_{M_b}}}, I^{M_b}_{\gamma_0}, I^{J_b}_{\delta}$ be the identity components of the centralizers of the above groups. By \cite[\S 3]{Kot6} we have that these groups are all inner forms of each other. Then by \cite[Lemma 3.2.1]{Kal3} we can pick an elliptic maximal torus of $I^{H_{M_b}}_{\gamma_{H_{M_b}}}$ and transfer it to the other groups.

Conversely, suppose we can find maximal tori $T_{H_{M_b}}, T_{M_b}, T_{J_b}$ that transfer to each other. Then we only need to find a $\gamma_0 \in T_{M_b}(\mathbb{Q}_p)$ which is $(G, H_{M_b})$-regular and $\nu_b$-acceptable. We first argue that for a positive natural number $s$ such that $s \nu_b \in X_*(T_{M_b})$ we have $s\nu_b(p) \in T_{M_b}(\mathbb{Q}_p)$ is $\nu_b$-acceptable. This follows because for each root $\alpha$, we have $v_p(\alpha (s\nu_b(p)))=\langle s\nu_p, \alpha \rangle$. 

Now let $t \in T_{M_b}(\mathbb{Q}_p)$ be a regular element. We claim that for $s$ sufficiently large (in particular, large enough  that $|\langle s \nu_b, \alpha \rangle| > |v_p(\alpha(t))|$ for each root $\alpha$ of $R(G, T_{M_b}) \setminus R(M_b, T_{M_b})$), the element $s\nu_b(p)\cdot t$ has the desired property. Pick $\alpha \in R(G, T_{M_b}) \setminus R(M_b, T_{M_b})$. Then we need to show that $\langle \nu_b, \alpha \rangle > 0$ if and only if $v_p(\alpha(s \nu_b(p) \cdot t))>0$. First we note that $\langle \nu_b, \alpha \rangle \neq 0$ since $M_b$ is precisely the centralizer of $\nu_b$, hence the roots that satisfy $\langle \nu_b, \alpha \rangle=0$ are exactly the roots of $M_b$.  Then $v_p(\alpha( s \nu_b(p) \cdot t))=\langle s \nu_b, \alpha \rangle + v_p(\alpha(t))$ which has the same sign as $\langle \nu_b, \alpha \rangle$ by our assumption on $s$. This finishes the proof.
\end{proof}
\subsection{A lemma on conjugacy classes and Levi subgroups}
In this subsection we prove a version of \cite[Lemma 6.2]{Shi3}, which is a key part of the stabilization of the cohomology of Igusa varieties. To begin, we pick a connected reductive group $G$ over $\Q_p$ and a refined endoscopic datum $(H,s,\eta)$ of $G$ and a set of representatives of $\mc{E}^i_{eff}(J_b, G;H)$ and $\mc{E}^i_{ef}(M_b,G,H)$ which we denote  $X^{\mf{e}}_{J_b}$ and $X^{\mf{e}}_{M_b}$ respectively. We can choose a pair  $(B_H, T_H)$ of a maximal torus of $H$ with maximal $\Q_p$ split rank and a rational Borel subgroup of $H$ such that $T_H \subset B_H$. We can pick each $(H, H_{M_b}, s, \eta_{M_b})$ in the above sets of representatives so that $H_M$ is a standard Levi subgroup. 

Associated to each element of $X^{\mf{e}}_{M_b}$ we have a map $\nu: \bb{D} \to H$ given by
\begin{equation}
\begin{tikzcd}
\bb{D} \arrow[r, "\nu_b"]  & A_{M_b} \arrow[r, hook] & T \arrow[r, "\sim"] & T_H \arrow[r, hook] & H.
\end{tikzcd}
\end{equation}

We now have the following lemma.

\begin{lemma}{\label{stabigusalem}}
\begin{enumerate}
    \item The map $\mc{EQ}^e(M_b, G) \to \mc{EQ}^r(G)$ induces maps 
    \begin{equation*}
        \mc{EQ}^e_{ef}(M_b, G) \to \mc{EQ}^r_{ef}(G),
    \end{equation*}
    and
    \begin{equation*}
        \mc{EQ}^e_{eff}(J_b, G) \to \mc{EQ}^r_{eff}(G).
    \end{equation*}
    \item Suppose that $[(H, H_{M_b}, s, \eta_1, \gamma_{H_{M_b}})] \in \mc{EQ}^e_{ef}(M_b, G; H)$ and $(H ,s ,\eta, \gamma_H)$ project to the same class in $\mc{EQ}^r_{ef}(G)$. Then there is a unique $(H, H'_{M_b}, s, \eta_2) \in X^{\mf{e}}_{M_b}$ and $(G, H_{M_b})$-regular, $\nu_b$-acceptable, semisimple $\gamma_{H'_{M_b}} \in H'_{M_b}(\Q_p) / \text{st} $ such that $(H, H'_{M_b}, s, \eta_2, \gamma_{H'_{M_b}})$ projects to $[(H, H_{M_b}, s, \eta_1, \gamma_{H_{M_b}})]$ and $\gamma_{H'_{M_b}}$ and $\gamma_H$ are stably conjugate in $H(\Q_p)$ under the natural inclusion $H'_{M_b}(\Q_p) \subset H(\Q_p)$. 
\end{enumerate}
\end{lemma}
\begin{proof}
Part (1) follows from Lemma \ref{SSEQMGGcomm}.

We now prove part (2).  We consider the following diagram:
\begin{equation}
    \begin{tikzcd}
    H(\Q_p)_{(G,H)-reg, ss} / \text{ st } \arrow[r, "T_1"] & \mc{EQ}^r_{ef}(G)\\
  \coprod\limits_{(H, H_{M_b}, s, \eta_1) \in X^{\mf{e}}_{M_b}} \; H_{M_b}(\Q_p)_{(G,H_{M_b})-reg, ss, \nu_b-acc.} / \text{ st } \arrow[r , "T_2"] \arrow[u, "S_1"] &  \mc{EQ}^e_{ef}(M_b, G; H) \arrow[u,"S_2"].
    \end{tikzcd}
\end{equation}
The map $T_1$ takes $\{\gamma'_H\}$ to the class of $(H,s,\eta, \gamma'_H)$ and $T_2$ is defined analogously. The map $S_1$ takes $\gamma_{H_{M_b}} \in H_{M_b}(\Q_p)$ to $
\gamma_{H_{M_b}} \in H(\Q_p)$ via the natural inclusion and $S_2$ is the natural forgetful map. We claim the diagram commutes. Indeed, we need to show that $(H, s, \eta, \gamma_{H_{M_b}})$ and $(H, s, \eta', \gamma_{H_{M_b}})$ induce the same element of $\mc{EQ}^r_{ef}(G)$. We may assume that $\eta'= \Int(g) \circ \eta$ for some $g \in \widehat{G}$. Then it is clear that $(H, s, \eta)$ and $(H,s,\eta')$ are isomorphic via the identity map $id: H \to H$ and therefore that this induces the desired isomorphism of endoscopic quadruples.

Now to prove (2), it suffices to show that if we have an element $x \in \mc{EQ}^r_{ef}(G)$ and a $y \in \mc{EQ}^{e}_{ef}(M_b, G; H)$ such that $S_2(y)=x$, then the map $S_1 |_{T^{-1}_2(y)}: T^{-1}_2(y) \to T^{-1}_1(x)$ is a bijection.

We claim that the fibers of $T_2$ have size $|\Out_r(H,s,\eta)|$. To compute the fiber $T^{-1}_2(y)$, we must compute the fiber in $H_{M_b}(\Q_p)$ for each $(H, H_{M_b}, s, \eta_1) \in X^{\mf{e}}_{M_b}$ whose class in $\mc{E}^{e}_{ef}(M_b, G; H)$ agrees with that of $y$. By Proposition \ref{fibercard}, the number of elements of $X^{\mf{e}}_{M_b}$ we must consider is $|\Out_r(H,s,\eta)/\Out_r(H_{M_b}, s, \eta_1)|$. Now, for a fixed $(H, H_{M_b}, s, \eta_1)$, the number of $\gamma_{H_{M_b}}$ such that $T_2((H, H_{M_b}, s, \eta_1, \gamma_{H_{M_b}}))=y$ is at most the size of the outer automorphism group of $(H, H_{M_b}, s, \eta_1)$ which is of cardinality $|\Out_r(H_{M_b}, s, \eta_1)|$. Hence, we will have proven the claim if we can show that if $\alpha \in \Aut_e(H, H_{M_b}, s, \eta_1))$ projects to a nontrivial element of $\Out_e(H, H_{M_b}, s, \eta_1) = \Out_r(H_{M_b}, s, \eta_1)$, then $\alpha(\gamma_{H_{M_b}})$ and $\gamma_{H_{M_b}}$ are not stably conjugate for any $(G, H_{M_b})$-regular, $\nu_b$-acceptable, semisimple $\gamma_{H_{M_b}} \in H_{M_b}(\Q_p)$. This follows from the argument given in the last paragraph of the proof of \cite[Lemma 9.7]{Kot6} and the fact that $M_{b, \der}$ is simply connected by Lemma \ref{levidersc}.

Then by a similar argument, we get that a fiber of $T_1$ has size $\Out_r(H,s,\eta)$. As a consequence, we have that $|T^{-1}_2(y)|=|T^{-1}_1(x)|$ and so to prove surjectivity of $S_1|_{T^{-1}_2(y)}$ we only need to prove injectivity. For this we cite \cite[Lemma 3.6]{Shi3}.


\end{proof}

\section{Representation theoretic preparations}
\subsection{Endoscopic transfer of representations}

Fix a connected reductive group $G$ over $\mathbb{Q}_p$. Following \cite[pg 1631]{Hir1}, we define $\mathbb{C}[\mathrm{Irr}(G(\mathbb{Q}_p))]$ to be the set of virtual characters of $G$ (i.e complex linear combinations of trace distributions of irreducible representations of $G(\mathbb{Q}_p)$). Then we define $\mathrm{Groth}(G(\mathbb{Q}_p))$ to be the Grothendieck group of admissible representations of $G(\mathbb{Q}_p)$ and observe that there is a natural embedding
\begin{equation*}
    \mathrm{Groth}(G(\mathbb{Q}_p)) \hookrightarrow \mathbb{C}[\mathrm{Irr}(G(\mathbb{Q}_p))],
\end{equation*}
so that the image consists of integer linear combinations of distribution characters.

We define $\mathbb{C}[\mathrm{Irr}(G(\mathbb{Q}_p))]^{st} \subset \mathbb{C}[\mathrm{Irr}(G(\mathbb{Q}_p))]$ to be the subset of stable distributions. Similarly, we define $\mathrm{Groth}(G(\mathbb{Q}_p))^{st}$ to be the preimage in $\mathrm{Groth}(G(\mathbb{Q}_p))$ of $\mathrm{Groth}(G(\mathbb{Q}_p)) \cap \mathbb{C}[\mathrm{Irr}(G(\mathbb{Q}_p))]^{st}$ under the above embedding.

Fix a refined endoscopic datum $(H,s, \eta) \in \mc{E}^r(G)$. We assume that $\eta$ can be lifted to a map $\Leta: \LH \to \LG$ and fix a lift. Then for a fixed choice of transfer factor $\Delta$, the transfer of distributions induces a map $\mathrm{Trans}$ which maps a stable distribution of $H(\mathbb{Q}_p)$ to an invariant distribution of $G(\mathbb{Q}_p)$. In fact, (for instance \cite[Proposition 4.6]{Hir1}), we get a map
\begin{equation}
    \mathrm{Trans} : \mathbb{C}[\mathrm{Irr}(H(\mathbb{Q}_p))]^{st} \to \mathbb{C}[\mathrm{Irr}(G(\mathbb{Q}_p))].
\end{equation}

We check that this map does not depend on our choice of representative of an isomorphism class in $\mathcal{E}^r(G)$ in the following sense. Suppose that $\alpha: H' \to H$ induces an isomorphism of refined endoscopic data between $(H, s, \eta)$ and  $(H', s' ,\eta')$. Choose a lift $\Leta'$ of $\eta'$ and a map $\Lalpha$ as in Lemma \ref{alphalem}.  We fix the transfer factor $\Delta \circ \alpha$ of $(H', s', \Leta')$. Then $\alpha$ induces an isomorphism 
\begin{equation*}
    \alpha^*: \mathbb{C}[\mathrm{Irr}(H(\mathbb{Q}_p)]^{st} \cong \mathbb{C}[\mathrm{Irr}(H'(\mathbb{Q}_p)]^{st}
\end{equation*} 
defined so that if $\phi= \sum\limits^k_{i=1} a_i \Theta_{\pi_i}$ then $\alpha^*(\phi)=\sum\limits^k_{i=1} a_i \Theta_{\pi_i \circ \alpha}$.

\begin{proposition}
Suppose $\phi \in \mathbb{C}[H(\mathbb{Q}_p)]^{st}$. Then we have the following equality of distributions in $
\mathbb{C}[\mathrm{Irr}(G(\mathbb{Q}_p))]$:
\begin{equation}
\mathrm{Trans}_{H,s, \eta}(\phi)=\mathrm{Trans}_{H', s', \eta'}(\alpha^*(\phi)).
\end{equation}
\end{proposition}
\begin{proof}
Fix Haar measures on $H$ and $H'$. Fix an $f \in C^{\infty}_c(G(\mathbb{Q}_p))$. Then we pick $f^H$ to be an $(H, s, \eta)$-matching function of $f$. We claim that $cf^{H'}$ where $f^{H'} := f^H \circ \alpha$ is an $(H', s', \eta')$ matching function of $f$, where $c$ is the ratio of the Haar measure on $H$ pulled back to $H'$ and the Haar measure on $H'$. Since the transfer factors for $(H,s,\eta)$ and $(H', s', \eta')$ are compatible, to check the claim we simply need to show that for any semisimple $\gamma_{H'} \in H'(\mathbb{Q}_p)$ we have 
\begin{equation*}
    SO^H_{\alpha(\gamma_{H'})}(f^H)=SO^{H'}_{\gamma_{H'}}(cf^{H'}).
\end{equation*}
But we have 
\begin{align*}
    SO^{H'}_{\gamma_{H'}}(cf^{H'}) &= \sum\limits_{\{\gamma'_{H'}\} \sim_{st} \{\gamma_{H'}\}} e({I^{H'}_{\gamma'_{H'}}})O_{\gamma'_{H'}}(cf^{H'})\\
    &=\sum\limits_{\{\alpha(\gamma'_{H'})\} \sim_{st} \{\alpha(\gamma_{H'})\}} e({I^H_{\alpha(\gamma'_{H'})}})O_{\alpha(\gamma'_{H'})}(f^{H})\\
    &=SO^{H}_{\alpha(\gamma_{H'})}(f^H),
\end{align*}
as desired.

Now take $\phi= \sum\limits_i a_i\Theta_{\pi_i} \in \C[\Irr(H(\Q_p))]^{st}$, where $\Theta_{\pi_i}$ is the distribution character of $\pi_i$.  Then we fix $f \in C^{\infty}_c(G(\mathbb{Q}_p))$ and thinking of it as an element of the Hecke algebra (since we've made a choice of Haar measure), we need to show that 
\begin{equation*}
    \mathrm{Trans}_{H', s', \eta'}(\alpha^* (\phi))(f)=\mathrm{Trans}_{H, s, \eta}(\phi)(f).
\end{equation*}
It suffices to show that
\begin{equation*}
    \int_{H(\Q_p)_{sr}} f^H(h)\phi(h)dh=\int_{H'(\Q_p)_{sr}} f^{H'}(h')\phi(\alpha(h'))dh' \left(=\int_{H'(\Q_p)_{sr}} cf^{H}(\alpha(h'))\phi(\alpha(h'))dh' \right).
\end{equation*}
This is clear from the definition of the constant $c$.
\end{proof}

\subsection{Recap of the Langlands correspondences}
In this section we review the formulation of the local Langlands correspondence including the endoscopic character identities. We also describe the conjectural global multiplicity formula for $L^2_{disc}(\mb{G}(\Q) \setminus \mb{G}(\A))$. In this paper we use the inverse normalization of the Langlands correspondence compatible with Deligne's normalization of the Artin map which agrees with \cite{HT1} and Shin. In particular, we are assuming that the Artin map takes a uniformizer to a geometric Frobenius element. However, we remark that this is the opposite convention to Kaletha and \cite{RV1}. We normalize our transfer factors for quasisplit groups using the `` Whittaker normalization'' denoted by $\Delta^{\lambda}_D$ in \cite[\S5.5]{KS}. Because of our normalization of the Langlands correspondence and because we simultaneously require that our transfer factors are compatible with twisted endoscopy, we are forced to use the $\Delta^{\lambda}_D$ normalization (cf. \cite{KS2}).

We let $\mc{L}_F$ denote the Langlands group that is globally conjectural and locally isomorphic to $W_F \rtimes SL_2(\C)$ when $F$ is non-archimedean and $W_F$ when $F$ is archimedean.  We use the notation of the (hypothetical) global Langlands group for simplicity. In particular cases, the contents of this subsection can be recast in the language of Arthur's version of global parameters as in \cite{KMSW}.

\begin{definition}
Let $F$ be a local or global field of characteristic $0$ and $\mb{G}^*$ a quasisplit connected reductive group over $F$. Then a map $\psi: \mathcal{L}_F \to \, ^L\mathbf{G}^*$ is a local (resp. global) Langlands parameter if $\psi$ maps semisimple elements to semisimple elements and commutes with the projections to $W_F$. We say that $\psi$ is tempered if its image in $^L\mb{G}^*$ projects to a relatively compact subset of $\widehat{\mb{G}^*}$.

A map $\psi: \mathcal{L}_F \times SL_2(\C) \to \, ^L\mathbf{G^*}$ is a local (resp global) Arthur parameter if the restriction to $\mc{L}_F$ is a tempered Langlands parameter. 

We often refer to Langlands and Arthur parameter as $L$- and $A$- parameters, respectively. We say that two $L$- or $A$- parameters $\psi_1, \psi_2$ are equivalent if there exists a (locally)-trivial $1$-cocycle $z: \mc{L}_F \times SL_2(\C) \to Z(\widehat{\mb{G}})$ (without the $SL_2(\C)$ factor in the $L$-parameter case) and a $g \in \widehat{\mb{G}}$ such that $\psi_1 = \Int(g)(\psi_2)z$.
\end{definition}
We have the following lemma:
\begin{lemma}{\label{equivconjlem}}
Suppose $\mb{G}^*$ satisfies the Hasse principle or $F$ is local. Then two $L$- or $A$- parameters $\psi_1, \psi_2$ are equivalent if and only if there exists $g \in \widehat{\mb{G}^*}$ such that $\psi_1 = \Int(g) \circ \psi_2$.
\end{lemma}
\begin{proof}
The second condition clearly implies the first.

Suppose now that $\psi_1, \psi_2$ are two equivalent $A$-parameters (the $L$-parameter case is analogous) so that we have $g \in \widehat{\mb{G}^*}$ and $1$-cocycle $z:\mc{L}_F \times SL_2(\C) \to Z(\widehat{\mb{G}^*}) $. Now, by definition in the local case and by \cite[Lemma 11.2.2]{Kot5} and \cite[Equation (4.2.2)]{Kot5} in the global case, we have that $z$ is trivial in $H^1(\mc{L}_F \times SL_2(\C), Z(\widehat{\mb{G}^*}))$. Hence there exists $x \in Z(\widehat{\mb{G
}^*})$ such that $z: w \mapsto x^{-1}\sigma_w(x)$. Hence if we let $g'=xg$, then we get
\begin{equation*}
    \psi_1= \Int(g') \circ \psi_2
\end{equation*}
as desired.
\end{proof}

Now, as in \cite[pg 641]{Kot5}, we define the group $S_{\psi}$ of self equivalences of $\psi$ to be the set of $g \in \widehat{\mathbf{G}^*}$ such that $x \mapsto g^{-1}\psi(x)g\psi(x)^{-1}$ is a locally trivial $1$-cocycle of $\mathcal{L}$, valued in $Z(\widehat{\mathbf{G}^*})$. If $\mathbf{G}^*$ satisfies the Hasse principle or $F$ is local, it follows from \cite[(4.2.2)]{Kot5} and \cite[(11.2.2)]{Kot5} (see also \cite[pg 48]{KalTai}) that 
\begin{equation}{\label{SCeq}}
S_{\psi}=C_{\psi}Z(\widehat{\mathbf{G}^*})
\end{equation}
where 
\begin{equation*}
    C_{\psi} :=\{g \in \widehat{\mb{G}^*}: \forall x \in \mathcal{L}_F \times SL_2(\mathbb{C}), g^{-1} \psi(x)g\psi(x)^{-1}=1\}.
\end{equation*}

We define the group $\mathcal{S}_{\psi} := \pi_0(S_{\psi}/Z(\widehat{\mathbf{G}^*}))$. We say an Arthur parameter is discrete if $S^{\circ}_{\psi} \subset Z(\widehat{\mathbf{G}^*})$. We see that if $\psi$ is discrete and $\mb{G}^*$ satisfies the Hasse principle, then $\mathcal{S}_{\psi}=C_{\psi}/Z(\widehat{\mathbf{G}^*})^{\Gamma}$.

We briefly discuss how the group $S_{\psi}$ depends on $\psi$ up to equivalence.
\begin{lemma}{\label{canoncent}}
If two $L$- or $A$- parameters $\psi_1$ and $\psi_2$  are equivalent, then there exists a canonical (up to conjugacy) isomorphism $S_{\psi_1} \cong S_{\psi_2}$.
\end{lemma}
\begin{proof}
We prove the $A$-parameter case. Let $g \in \widehat{\mb{G}^*}, z \in Z^1(\mc{L}_F, Z(\widehat{\mb{G}^*}))$ be such that $\psi_2= \Int(g)(\psi_1)z$. Then we get a map
\begin{equation*}
    \Int(g): S_{\psi_1} \to S_{\psi_2},
\end{equation*}
given by $s \mapsto \Int(g)(s)$. If we make some different choice, $g' \in \widehat{\mb{G}^*}, z' \in Z^1(\mc{L}_F, Z(\widehat{\mb{G}^*}))$, then we get a map $\Int(g^{-1}g'): S_{\psi_1} \to S_{\psi_1}$. But $g^{-1}g' \in S_{\psi}$ and so this map is just an inner automorphism of $S_{\psi_1}$. This shows the desired result.
\end{proof}

Given a place $v$ of $\mathbb{Q}$, we expect to have a localization map $\mathcal{L}_v \times SL_2(\mathbb{C}) \hookrightarrow \mathcal{L} \times SL_2(\mathbb{C})$ defined up to conjugation. Thus for each place $v$, we get a local Arthur parameter $\psi_v: \mathcal{L}_v \times SL_2(\mathbb{C}) \to \,\ ^L\mathbf{G}^*_{\mathbb{Q}_v}$. We define the local centralizer group $S^{\natural}_{\psi_v}$ by
\begin{equation}
S^{\natural}_{\psi_v} := C_{\psi_v}/[C_{\psi_v} \cap \widehat{\mathbf{G}^*}^{\der}_{\mathbb{Q}_v}]^{\circ}.
\end{equation}
\begin{remark}
When $Z(\mathbf{G}^*)$ is connected, this is the same as $C_{\psi_v}/[C_{\psi_v} \cap \widehat{\mathbf{G}^*}^{\Sc}_{\mathbb{Q}_v}]^{\circ}$ since $Z(\mb{G}^*)$ connected implies that $\widehat{\mb{G}^*}_{\der}$ is simply connected. 
\end{remark}

We now review the conjectural local and global Langlands correspondences for a connected reductive group $\mb{G}$ over $\Q$ with a fixed quasisplit inner form $\mb{G}^*$ and extended pure inner twist $(\varrho, z^{\iso})$ such that $\varrho: \mb{G}^* \to \mb{G}$ and $z^{\iso} \in Z^1_{\bas}(\mc{E}^{\iso}, \mb{G}^*)$. The notion of an extended pure inner twist is discussed for instance in \cite[pg 19]{Kal1}. 

We begin by describing the local Langlands correspondence for our group $\mb{G}^*$ at a fixed place $v$ of $\Q$ following \cite[Conj. F]{Kal1}. Define $G := \mathbf{G}_{\mathbb{Q}_v}$ and $G^* := \mb{G}^*_{\Q_v}$. We get a localization $z^{\iso}_v \in Z^1_{\bas}(\mc{E}^{\iso}_v, \mb{G}^*)$ and hence a localization $(\varrho_v, z^{\iso}_v)$ of our extended pure inner twist.

Then to a local $A$-parameter $\psi_v: \mathcal{L}_v \times SL_2(\mathbb{C}) \to \,\ ^LG^*$, there is a local $A$-packet $\Pi_{\psi_v}(G^*, \varrho_v)$. This is a priori multiset
whose elements lie in $\mathrm{Irr}_{unit}(G)$, the set of unitary irreducible admissible representations of $G(\mathbb{Q}_v)$.

For a fixed Whittaker datum, $\mathfrak{w}_v$ of $G^*$,  we expect a map
\begin{equation}
   \begin{tikzcd}
    \Pi_{\psi_v}(G, \varrho_v) \arrow[r, "\iota_{\mathfrak{w}_v}"] & \mathrm{Irr}(S^{\natural}_{\psi_v}, \chi_{z^{iso}_v}),
\end{tikzcd} 
\end{equation}

where $\mathrm{Irr}(S^{\natural}_{\psi_v}, \chi_{z^{iso}_v})$ denotes the set of irreducible representations of $S^{\natural}_{\psi_v}$ whose restriction to $Z(\widehat{G})^{\Gamma_v}$ acts by the character attained from $z^{\iso}_v$ under the Kottwitz map $\kappa_G: \mathbf{B}(\Q_v, G)_{bas} \to X^*(Z(\widehat{G})^{\Gamma_v})$. The map $\iota_{\mathfrak{w}_v}$ allows us to define a pairing $\langle \cdot, \cdot \rangle : \mathrm{Irr}(S^{\natural}_{\psi_v}, \chi_{z^{iso}_v}) \times S^{\natural}_{\psi_v} \to \mathbb{C}$ given by 
\begin{equation}
    \langle \pi, s \rangle= \mathrm{tr}(\iota_{\mathfrak{w}_v}(\pi) \mid s).
\end{equation}

We are now in a position to describe the endoscopic character identities satisfied by $\Pi_{\psi_v}(G, \varrho_v)$. Let $(H, s, \eta)$ be a representative of an equivalence class of refined endoscopic data for $G^*$ and suppose there exists an extension of $\eta$ to a map $\Leta: \LH \to \LG^*$. Then suppose that $\psi^{\mathfrak{e}}_v$ is a local Arthur parameter for $H$ such that $\Leta \circ \psi^{\mathfrak{e}}_v=\psi_v$. Suppose $f \in \mc{H}(G(\Q_v)), f^{\mathfrak{e}} \in \mc{H}(H(
\Q_v))$ are $\Delta[\mathfrak{w}_v, z^{iso}_v]$-matching functions where the transfer factor is constructed as in \cite[\S 4.3]{KalTai}. Let $\Theta_{\pi}$ denote the distribution character for a representation $\pi$ and denote by $s_{\psi_v}$ the image under $\psi_v$ of the nontrivial central element in the Arthur $SL_2(\mathbb{C})$ factor. Then we expect
\begin{equation}{\label{endocharid}}
    \sum\limits_{\pi^{\mathfrak{e}} \in \Pi_{\psi^{\mathfrak{e}}}(H, 1)} \langle \pi^{\mathfrak{e}}, s_{\psi^{\mathfrak{e}}_v} \rangle \Theta_{\pi^{\mathfrak{e}}}(f^{\mathfrak{e}}) = e(G) \sum\limits_{\pi \in \Pi_{\psi_v(G, \varrho_v)}} \langle \pi, \eta(s)s_{\psi_v} \rangle\Theta_{\pi}(f),
\end{equation}
where $e(G)$ is the Kottwitz sign.

We now turn to describing the isocrystal normalization of the global multiplicity formula following \cite{KalTai} Let $\psi$ be a discrete global Arthur parameter of $\mathbf{G}^*$. For each place $v$, we expect to have canonical (up to conjugacy) embeddings $\mc{L}_{F_v} \hookrightarrow \mc{L}_F$ and  $\, ^L\mathbf{G}^*_{\Q_v} \hookrightarrow \,\ ^L\mathbf{G}^*$ where the latter induces an isomorphism between $\widehat{\mathbf{G}^*}$ and $\widehat{\mathbf{G}^*_{\Q_v}}$. Then for each place $v$ of $\mathbb{Q}$, we get a local Arthur parameter $\psi_v$ such that the following diagram commutes
\begin{equation}
\begin{tikzcd}
\mc{L}_{F_v} \arrow[r, hook] \arrow[d, "\psi_v"] & \mc{L}_{F} \arrow[d, "\psi"] \\
^L\mb{G}^*_{\Q_v} \arrow[r, hook] & ^L\mb{G}^*. 
\end{tikzcd}
\end{equation}
These maps induce an embedding $C_{\psi} \hookrightarrow C_{\psi_v}$.

Fix a reductive model $\ms{G}$ of $\mb{G}$ over $\Z[\frac{1}{N}]$ for some integer $N>0$.  Then we have a global packet $\Pi_{\psi}(\mathbf{G}, \varrho)$ consisting of admissible $\mathbf{G}(\mathbb{A})$-representations. In particular, the set $\Pi_{\psi}(\mathbf{G}, \varrho)$ consists of all global representations $\pi$ such that for each $v$ we have $\pi_v \in \Pi_{\psi}(\mb{G}_{\Q_v}, \varrho_v)$ and for all but finitely $v$, the representation $\pi_v$ is unramified with respect to $\ms{G}(\Z_v)$. 

Fix a global Whittaker datum $\mf{w}$ of $\mb{G}^*$. We are now going to define a map $\iota_{\mathfrak{w}}: \Pi_{\psi}(\mathbf{G}, \varrho) \to \mathrm{Rep}(\mathcal{S}_{\psi})$. Pick $\pi \in \Pi_{\psi}(\mathbf{G}, \varrho)$ and write $\pi= \otimes'_v \pi_v$. Then for each place $v$, we lift the representation $\iota_{\mathfrak{w}_v}(\pi_v)$ to a representation $C_{\psi_v}$ and then restrict to $C_{\psi}$ to get a representation $\chi_{\pi_v}$. 
At almost all $v$ we have that $\chi_{\pi_v}=1$ (see \cite[\S 4.6]{KalTai}). Hence the tensor product $\chi_{\pi}=\otimes_v \chi_{\pi_v}$ is well defined. 

\begin{proposition}
The representation $\chi_{\pi}$ of $C_{\psi}$ factors through $\mathcal{S}_{\psi}$.
\end{proposition}
\begin{proof}
It suffices to show that $\chi_{\pi}|_{Z(\widehat{\mathbf{G}})^{\Gamma}}$ acts trivially. 

Now recall Kottwitz (\cite[ Prop. 15.6]{Kot9}) constructs a localization map
\begin{equation*}
    \mathbf{B}(\mathbb{Q}, \mathbf{G}) \to \oplus_v \mathbf{B}(\mathbb{Q}_v, \mathbf{G}_v),
\end{equation*}
as well as a map
\begin{equation*}
    \oplus_v \mathbf{B}(\mathbb{Q}_v, \mathbf{G}) \to \oplus_v X^*(Z(\widehat{\mathbf{G}})^{\Gamma_v}) \to X^*(Z(\widehat{\mathbf{G}})^{\Gamma}),
\end{equation*}
where the first map is the Kottwitz map $\kappa_{\mathbf{G}_v}$ at each place and the second map is restriction and sum.

Then Kottwitz proves that an element in $\oplus_v \mathbf{B}(\mathbb{Q}_v, \mathbf{G}_v)$ lies in the image of the localization map precisely when its image in $X^*(Z(\widehat{\mathbf{G}})^{\Gamma})$ is trivial.

Now, since $\iota_{\mathfrak{w}_v}(\pi_v) \in \mathrm{Irr}(S^{\natural}_{\psi_v}, \chi_{z^{iso}_v})$, it follows that $\chi_{\pi_v}|_{Z(\widehat{\mathbf{G}}_v)^{\Gamma_v}}$ acts by $\chi_{z^{iso}_v}$ and that $\chi_{\pi}|_{Z(\widehat{\mathbf{G}})^{\Gamma}}$ acts by $\sum\limits_v \chi_{z^{iso}_v}$. But since $(\chi_{z^{iso}_v})$ is indeed in the image of the localization map, by what we said earlier, $\chi_{\pi}|_{Z(\widehat{\mathbf{G}})^{\Gamma}}$ acts trivially.
\end{proof}
We now define the global pairing 
\begin{equation*}
    \langle \cdot, \cdot \rangle : \Pi_{\psi}(\mb{G}, \varrho) \times \mathcal{S}_{\psi} \to \mathbb{C},
\end{equation*}
 by 
 \begin{equation*}
      \langle \pi, s \rangle= \mathrm{tr}\chi_{\pi}(s).
 \end{equation*}

We can now state Arthur's conjectural multiplicity formula for the discrete spectrum. Arthur constructs a character $\epsilon_{\psi}: \mathcal{S}_{\psi} \to \{ \pm 1\}$ (see equation \cite[(8.4)]{Art1}).

\begin{conjecture}
Fix a character $\chi$ of $A_{\mathbf{G}}(\mathbb{R})^{\circ}$. Then $L^2_{disc, \chi}(\mathbf{G}(\mathbb{Q})\setminus \mathbf{G}(\mathbb{A}))$ is isomorphic to 
\begin{equation}
    \bigoplus\limits_{[\psi]}\bigoplus\limits_{\pi \in \Pi_{\psi}(\mb{G}, \varrho)} m(\psi, \pi) \pi,
\end{equation}
where 
\begin{equation}
    m(\psi, \pi) := \frac{1}{|\mathcal{S}_{\psi}|} \sum\limits_{x \in \mathcal{S}_{\psi}} \epsilon_{\psi}(x) \langle \pi, x \rangle.
\end{equation}
The first sum is taken over equivalence classes of discrete Arthur parameters with associated character on $A_{\mathbf{G}}(\mathbb{R})^{\circ}$ equal to $\chi$.
\end{conjecture}

\subsection{Representation theory and endoscopy}

To begin, we prove spectral analogues of the maps in Corollary \ref{endcomp}. Such constructions appear for instance in works of Kottwitz (see the proof of \cite[Prop 11.3.2]{Kot5}) and Shelstad (\cite[\S 4.2]{She1}). We choose to provide the details in this work. For the discussion of these spectral analogues only, $F$ is a number field or $p$-adic field and $G$ is a reductive group over $F$ with simply connected derived subgroup. We assume that we can lift each endoscopic datum $(H,s,\eta)$ of $G$ to a datum $(H, s, \Leta)$. 

\begin{definition}
We define $\mc{EP}_F(G)$ to be equivalences classes of tuples $(H, s, \, ^L\eta, \psi^H)$ where $(H, s, \eta)$ is a standard endoscopic datum and $\psi^H$ is an Arthur parameter. We say that $(H_1, s_1, \Leta_1, \psi^H_1)$ and $(H_2, s_2, \Leta_2, \psi^H_2)$ are equivalent if there exists an isomorphism $\alpha: H_2 \to H_1$ of endoscopic data and such that $\psi^H_2 \circ \, ^L\alpha$ is $Z(G)$-equivalent to $\psi^H_1$. Note that by Lemma \ref{alphalem}, the choice of $^L\alpha$ is unique up to $\widehat{H_1}$-conjugacy and that therefore the notion of $Z(G)$-equivalence does not depend on this choice.

We define $\mc{EP}^r_F(G)$ to consist of tuples $(H, s, \, ^L\eta, \psi^H)$ where now $(H, s, \eta)$ is assumed to be a refined endoscopic datum and $(H_1, s_1, \Leta_1, \psi^H_1)$ and $(H_2, s_2, \Leta_2, \psi^H_2)$ are equivalent if there exists an isomorphism $\alpha: H_2 \to H_1$ giving an isomorphism of refined endoscopic data and such that ${}^L\alpha \circ \psi^H_1$ is $Z(G)$-equivalent to $\psi^H_2$.
\end{definition}

\begin{definition}
We also define the set $\mc{SP}_F(G)$ of equivalence classes of pairs $(\psi, \ov{s})$ such that $\psi$ is an Arthur parameter of $G$ and $\ov{s} \in S_{\psi}/Z(\widehat{G})$. Two pairs $(\psi_1, \ov{s}_1)$ and $(\psi_2, \ov{s}_2)$ are equivalent if $\psi_1$ and $\psi_2$ are equivalent by some $g \in \widehat{G}$ such that $\Int(g)(\ov{s}_1)$ and $\ov{s}_2$ are conjugate in $\mc{S}_{\psi_2}$.

We define the set $\mc{SP}^r_F(G)$ to be equivalence classes of pairs $(\psi, s)$ such that $\psi$ is an Arthur parameter of $G$ and $s \in C_{\psi}$. We say that $(\psi_1, s_1)$ is equivalent to $(\psi_2, s_2)$ if $\psi_1$ is equivalent to $\psi_2$ by some $g \in \widehat{G}$ such that $\Int(g)(s_1)$ is conjugate to $s_2$ in $C_{\psi_2}$.
\end{definition}
We now have
\begin{proposition}{\label{EPeqSP}}
There exists natural bijection
\begin{equation}
    \mc{EP}^r_F(G) \cong \mc{SP}^r_F(G)
\end{equation}
induced by
\begin{equation*}
    (H, s, \Leta, \psi^H) \mapsto (\Leta \circ \psi^H, \eta(s)).
\end{equation*}
\end{proposition}
\begin{proof}
Given $(H, s, \, ^L\eta, \psi^H) \in \mc{EP}^r_F(G)$, we get a parameter $\psi$ of $G$ given by ${}^L\eta \circ \psi^H$. Now, $s \in Z(H)^{\Gamma_F}$ and so commutes with the image of $\psi^H$ in $^LH$. In particular, this implies that  $\eta(s)$ commutes with the image of $\psi$ in $^LG$ and hence that $\eta(s) \in C_{\psi}$. It is easy to check that this induces a map $\mc{EP}^r_F(G) \to \mc{SP}^r_F(G)$. To prove the proposition, we construct an inverse map.

Pick an equivalence class $[\psi, s] \in \mc{SP}^r_F(G)$ and pick a representative $(\psi, s)$. Define $\widehat{H} := Z_{\widehat{G}}(s)^0$ and define $\eta$ to be the natural embedding $\widehat{H} \hookrightarrow \widehat{G}$. Then, for any $g \in \im(\psi) \subset \,  ^LG$, the map $\Int(g): \widehat{G} \to \widehat{G}$ stabilizes $\widehat{H}$ and hence gives a map 
\begin{equation*}
   \ov{\psi}: \mc{L}_F \times SL_2(\C) \to \Aut(\widehat{H}) \to \Out(\widehat{H}),
\end{equation*}
which has finite image by Lemma \ref{finoutim}.
Now, pick a splitting 
\begin{equation*}
    \iota: \Out(\widehat{H}) \to \Aut(\widehat{H}),
\end{equation*}
compatible with the splitting used to define $^LG$ and consider the action of $\mc{L}_F \times SL_2(\C)$ on $\widehat{H}$ given by $\iota \circ \ov{\psi}$. Since the image of $\ov{\psi}$ is finite, this new action has the benefit that it factors through a finite quotient of $\mc{L}_F \times SL_2(\C)$. Now note that any continuous finite quotient of $\mc{L}_F$ is of the form $\Gal(F'/F)$ for some finite extension $F'/F$. Indeed, evidently $\SL_2(\C)$ has no non-trivial finite continuous quotients. Thus, it suffices to prove the claim for $\mc{L}_F$. Now, if $K$ denotes the kernel of $\mc{L}_F\to W_F$ then $K$ is a connected pro-reductive complex group. Thus, $K$ also has no non-trivial finite continuous quotients. Thus, we've reduced the claim to $W_F$ for which the claim is obvious.

In particular, we can find a finite extension $E/F$ so that the action factors through $\Gal(E/F)$. Hence we get an action of $\Gamma_F$ on $\widehat{H}$. In particular, our action of $\Gamma_F$ allows us to define a quasisplit group $H$ over $F$ whose dual group is $\widehat{H}$. 

We need to check that $(H,s, \eta)$ is a refined endoscopic datum. It remains to check that the conjugacy class of $\eta$ is $\Gamma_F$-invariant and that $s \in Z(\widehat{H})^{\Gamma_F}$. For the first check, we pick $w \in \Gamma_F$ and need to show that the constructed action of $w$ on $\widehat{H}$ differs from the action of $w$ on $\widehat{G}$ by an inner automorphism of $\widehat{G}$. This is true by construction. For the second check, we pick a $w \in \Gal(E/F)$ and observe that the action of $w$ on $\widehat{H}$ differs from that induced by $\im(\psi)$ by an inner automorphism of $\widehat{H}$. In particular, since $s$ commutes with $\im(\psi)$, it is invariant under the action on $\widehat{H}$ as desired.

By our assumption that $G_{\der}$ is simply connected, we can extend $\eta$ to a map $\eta:  \, ^LH \to \, ^LG$. Then we need to check that the parameter $\psi$ factors through $^LH$. In other words, we need to show that $\im(\psi) \subset \mc{H}$ using the notation of \ref{triptodat}. First pick $w \in \mc{L}_F \times SL_2(\C)$ and consider $\psi(w)$. Then we check that there exists an element $y \in {}^LH$ such that $\Int(\psi(w)) \circ \eta= \eta \circ \Int(y)$. But indeed this follows immediately from the fact that the action of $w \in W_F$ on $\widehat{H} \subset \, ^LH$ differs from that of $\Int(\psi(w))$ by an element of $\Inn(\widehat{H})$. Finally we define $\psi^H$ by $\psi= \, ^L\eta \circ \psi^H$

We now show the map we have constructed is well-defined. First, it is clear that our construction does not depend on our choice of splitting. Next, suppose that $(\psi_1, s_1)$ is equivalent to $(\psi_2, s_2)$ by some $g \in \widehat{G}$ satisfying $w \mapsto g\psi(w)g^{-1}\psi(w)^{-1}$ is a (locally) trivial cocycle of $\mc{L}_F \times SL_2(\C)$ valued in $Z(\widehat{G})$. Then by assumption $gs_1g^{-1}$ is conjugate by some $s \in C_{\psi_2}$ to $s_2$ and so the groups $\widehat{H}_1$ and $\widehat{H}_2$ are conjugate in $\widehat{G}$ by $sg$. Moreover, it is easy to check that the map $\Int(sg) : \widehat{H}_1 \to \widehat{H}_2$ will preserve the actions of $\Gamma_F$ up to an inner automorphism of $\widehat{H}_2$ and hence descends to an isomorphism $\alpha: H_2 \to H_1$ defined over $F$. The map $\alpha$ then gives an isomorphism of the endoscopic data $(H_1, s_1, \eta_1)$ and $(H_2, s_2, \eta_2)$ and $\Int(sg) \circ \psi^H_1$ is $Z(G)$-equivalent to $\psi^H_2$. This shows the map is well-defined.

To conclude the proof, we must show that the maps $\mc{EP}^r_F(G) \to \mc{SP}^r_F(G)$ and $\mc{SP}^r_F(G) \to \mc{EP}^r_F(G)$ that we have constructed are inverses of each other. It is clear that the composition $\mc{SP}^r_F(G) \to \mc{EP}^r_F(G) \to \mc{SP}^r_F(G)$ is the identity. Now take a representative $(H, s, \, ^L\eta, \psi^H)$ of $[H,s, \, ^L\eta, \psi^H] \in \mc{EP}^r_F(G)$. Then we want to show that this is equivalent to the tuple $(H', s', \, ^L\eta', \psi^{H'})$ that we get from applying the composition $\mc{EP}^r_F(G) \to \mc{SP}^r_F(G) \to \mc{EP}^r_F(G)$ to $(H, s, \, ^L\eta, \psi^H)$. We have a map ${\eta'}^{-1} \circ \eta: \widehat{H} \to \widehat{H'}$. We claim this map is equivariant for each $w \in \Gamma_F$ up to conjugation by some $h \in \widehat{H}$. There exists some finite extension $E/F$ such that the actions of $\Gamma_F$ on both groups factor through $\Gal(E/F)$ hence we need only prove the claim for $w \in \Gal(E/F)$. Then if we pick a lift $w' \in \mc{L}_F \times SL_2(\C)$ of $w$, the action of $w$ on each group differs by an inner automorphism from the action of conjugation by $\psi^H(w')$ or $\psi^{H'}(w')$ respectively. So then we have (up to conjugation which we denote by $\sim$) for $h \in \widehat{H}$:
\begin{align*}
    (w \cdot ({\eta'}^{-1} \circ \eta))(h) &=w({\eta'}^{-1} \eta(w^{-1}(h))\\
    & \sim \Int(\psi^{H'}(w'))({\eta'}^{-1}\eta (\Int(\psi^H(w')^{-1})(h))\\
    &= ({\eta'}^{-1} \circ \Int(\psi(w')) \circ \Int(\psi(w')^{-1}) \circ  \eta)(h)\\
    &=({\eta'} \circ \eta)(h).
\end{align*}
This proves the claim and implies that the isomorphism descends to an isomorphism $\alpha: H' \to H$ defined over $F$. This satisfies $\widehat{\alpha}(s)=s'$ and hence gives the desired isomorphism of endoscopic data. Moreover, it is clear that we have an equivalence $(H,s, ^L\eta, \psi^H), (H', s' , \, ^L\eta', \psi^{H'})$.
\end{proof}
There is also a non-refined version of the above proposition.
\begin{proposition}
We have a natural bijection
\begin{equation}
    \mc{EP}_F(G) \cong \mc{SP}_F(G)
\end{equation}
\end{proposition}
\begin{proof}
The argument is similar to that of the previous lemma, though we comment on some of the key differences. 

We can still define a map $\mc{EP}_F(G) \to \mc{SP}_F(G)$. We note that because the endoscopic datum is no longer assumed to be refined, $s$ will not be an element of $Z(\widehat{H})^{\Gamma_F}$ and hence, $\eta(s)$ will not be an element of $C_{\psi}$. Instead, we can check that $w \mapsto s^{-1}\psi^H(w)s \psi^H(w)^{-1}$ gives a $1$-cocycle of $\mc{L}_F$ valued in $Z(\widehat{G})$ so that $\eta(s) \in S_{\psi}$. 

To define the map $\mc{SP}_F(G) \to \mc{EP}_F(G)$, we start with $(\psi, \ov{s})$ choose a lift $s \in S_{\psi}$. Now we can proceed as in the previous proposition and can check that the isomorphism class of the endoscopic datum $(H, s, \eta)$ that we construct will not depend on this choice of lift.
\end{proof}
\begin{corollary}{\label{paramcompat}}
The above maps fit into a natural commutative diagram
\begin{equation}
\begin{tikzcd}
\mc{EP}^r_F(G) \arrow[d] \arrow[r, "\sim"] & \mc{SP}^r_F(G) \arrow[d]  \\
\mc{EP}_F(G) \arrow[r, "\sim"] & \mc{SP}_F(G).
\end{tikzcd}
\end{equation}
\end{corollary}
\begin{corollary}
We define subsets $\mc{EP}^r_F(G)^{\el} \subset \mc{EP}^r_F(G), \mc{EP}_F(G)^{\el} \subset \mc{EP}_F(G)$ where $(H, s, \eta)$ is assumed to be an elliptic endoscopic datum and $\psi^H$ is such that $\eta \circ \psi^H$ is  discrete. Similarly, we define subsets $\mc{SP}^r_F(G)^{\el} \subset \mc{SP}^r_F(G), \mc{SP}_F(G)^{\el} \subset \mc{SP}_F(G)$ where $\psi$ is assumed be discrete. Then the above bijections and commutative diagram restrict to these sets.
\end{corollary}
\begin{proof}
We need to check that if $[\psi, s] \in \mc{SP}^r_F(G)^{\el}$, then the tuple $(H, s, \, ^L\eta, \psi^H)$ we construct from $(\psi, s)$ satisfies that $(H, s, \eta)$ is elliptic as well as the analogous statement for $\mc{SP}_F(G)^{\el}$. But we have have $\eta((Z(\widehat{H})^{\Gamma_F})^{\circ}) \subset \eta(C^{\circ}_{\psi^H}) \subset C^{\circ}_{\psi} \subset Z(\widehat{G})$ as desired. 
\end{proof}
\begin{remark}
In the case that $\psi$ is discrete, $S_{\psi}/Z(\widehat{G})$ equals $\mc{S}_{\psi}$ and so $\mc{SP}_F(G)^{\el}$ parametrizes $(\psi, \ov{s})$ such that $\psi$ is discrete and  $\ov{s} \in \mc{S}_{\psi}$.
\end{remark}

\section{Cohomology of Shimura varieties}
In this section we review the conjectural desciption of the cohomology of Shimura varieties as in \cite{kot7}. For the most part, we can follow the stabilization and de-stabilization of the trace formula for the cohomology of Shimura varieties as in \cite{kot7}. However, complications arise coming from the normalization of the local Langlands correspondence at non-quasisplit places. We will need a description of the local Langlands correspondence at non-quasisplit places for two reasons. The first is that we will not assume $\mb{G}$ is quasisplit and will need a description of the local Langlands correspondence for $\mb{G}(\Q_v)$ for each finite place $v$. More seriously, when we study Igusa varieties and Rapoport--Zink spaces, we will need to describe the local Langlands correspondence for $J_b(\Q_p)$ where $J_b$ is an inner form of a Levi subgroup $M_b$ of $\mb{G}_{\Q_p}$. 

We choose to normalize the local Langlands correspondence using the isocrystal description as in \cite[Conjecture F]{Kal1} and statements from the conjectural global Langlands correspondence as in \cite{KalTai}. In order to make the rest of our discussion compatible with this choice, we need to make a number of other modifications. Most prominently we need the notion of refined endoscopic datum in the local case as this is the version of endoscopy needed to state the local endoscopic character identity (Equation \eqref{endocharid}) using the isocrystal normalization of the local Langlands correspondence.  Via Lemma \ref{SSeqEQ}, this means that locally, the group $Z(\widehat{I}^G_{\gamma})^{\Gamma_v}$ must appear in the stabilization of the trace formula as opposed to the group $\mf{K}(I^G_{\gamma})$ of Kottwitz. We describe this modified description of the cohomology of Shimura varieties in detail in this section.

\subsection{Normalization of transfer factors}{\label{transfactsect}}
In this subsection we discuss the normalization of transfer factors using the Kottwitz set $\mb{B}(F, \mb{G})$. Our exposition follows the theory developed in \cite{Kal4}, \cite{KalTai}, \cite{Kot9}, and \cite{BM3}. 

\begin{assumption}
For the remainder of this section, we work with a connected reductive group $\mb{G}$ over $\Q$ and a fixed quasisplit inner form $\mb{G}^*$ satisfying the following properties.
\begin{itemize}
    \item The group $\mb{G}$ satisfies the Hasse principle,
    \item The group $\mb{G}_{\der}$ is simply connected,
    \item There exists an extended pure inner twist $(\varrho, z^{iso})$ where $z^{iso} \in Z^1_{\bas}(\mc{E}^{\iso}, \mb{G}^*)$ and $\varrho: \mb{G}^* \to \mb{G}$.
\end{itemize}
\end{assumption}
We fix an extended pure inner twist $(\varrho, z^{iso})$ as in the above assumption. We also fix a $\Q$-splitting $(\mb{B}, \mb{T}, \{X_{\alpha}\})$ of $\mb{G}$ as well as a character $\chi: \Q \to \C^{\times}$. By \cite[\S 5.3]{KS}, this data induces a Whittaker datum $\mf{w}$ of $\mb{G}^*$ and at each place $v$ of $\Q$, the above data induces a local Whittaker datum $\mf{w}_v$ of $\mb{G}^*_{\Q_v}$.

From the theory of standard endoscopy (\cite{LS2}, see also \cite[\S5.3]{KS}, \cite{KS2}), there exists for each refined endosopic datum $(\mb{H}, s, \eta)$ of $\mb{G}^*$ and each place $v$ of $\Q$, a canonical ``Whittaker normalized'' transfer factor 
\begin{equation*}
    \Delta[\mf{w}_v]: \mb{G}^*(\Q_v)_{ss} \times \mb{H}(\Q_v)_{ss, (\mb{G}^*_{\Q_v}, \mb{H}_{\Q_v})-reg} \to \C.
\end{equation*}
We recall that we are using the normalization $\Delta^{\lambda}_D$ in the notation of \cite[\S5.5]{KS2}.

Now, let $v$ be a place of $\Q$ and consider a pair of semisimple $\gamma^*_v \in \mb{G}^*(\Q_v)$ and $\gamma_v \in \mb{G}(\Q_v)$ such that $\varrho_v(\gamma^*_v)$ and $\gamma_v$ are stably conjugate in $\mb{G}(\Q_v)$. Then \cite[\S4.3]{KalTai} in the strongly regular case and \cite[\S3.2]{BM3} generally, construct an element $\inv[z^{iso}_v](\gamma^*_v, \gamma_v) \in \mb{B}(\Q_v, I^{\mb{G}^*_{\Q_v}}_{\gamma^*_v})_{bas}$. The Kottwitz map induces a pairing
\begin{equation*}
    \mb{B}(\Q_v, I^{\mb{G}^*_{\Q_v}}_{\gamma^*_v})_{bas} \times Z(\widehat{I^{\mb{G}^*_{\Q_v}}_{\gamma^*_v}})^{\Gamma_v} \to \C^{\times}, 
\end{equation*}
and by \cite[Proposition 4.3.1]{KalTai} and \cite[Theorem 3.4]{BM3},
\begin{equation*}
    \Delta[\mf{w}_v, z^{iso}_v](\gamma^{\mb{H}}_v, \gamma) := \Delta[\mf{w}_v](\gamma^{\mb{H}}_v, \gamma^*_v)\langle  \inv[z^{iso}_v](\gamma^*_v, \gamma_v), \eta(s) \rangle^{-1},
\end{equation*}
is a local transfer factor between $(\mb{H}_{\Q_v}, s, \eta_v)$ and $\mb{G}_{\Q_v}$. We note that the inverse sign appears in the above formula because we use a different normalization of transfer factors than \cite{KalTai}. We remark that the trace formula for Shimura varieties involves non-strongly regular semisimple elements and that it will be crucial to have the above explicit description of the local transfer factor for such elements.

By \cite[Proposition 4.3.2]{KalTai} and \cite[Corollary 3.9]{BM3}, the product 
\begin{equation*}
    \Delta[\mf{w}, z^{iso}] := \prod\limits_v \Delta[\mf{w}_v, z^{iso}_v],
\end{equation*}
is equal to the canonical adelic transfer factor and in particular satisfies
\begin{equation*}
    \Delta[\mf{w}, z^{iso}_v](\gamma^{\mb{H}}, \gamma) = 1,
\end{equation*}
for $\gamma^{\mb{H}} \in \mb{H}(\A)$ that transfers to $\gamma \in \mb{G}(\Q)$.

With the above normalizations, we have the following fundamental theorem. 
\begin{theorem}
For each place $v$ of $\Q$ and each $f \in C^{\infty}_c(\mb{G}(\Q_v))$, there exists a function $f^{\mb{H}} \in C^{\infty}_c(\mb{H}(\Q_v))$ so that for any $(\mb{G}, \mb{H})$-regular semisimple $\gamma^{\mb{H}} \in \mb{H}(\Q_v)$ transferring to $\gamma \in \mb{G}(\Q_v)$, we have
\begin{equation}
    SO_{\gamma^{\mb{H}}}(f^{\mb{H}})=\sum\limits_{\gamma' \sim_{st} \gamma} e(I^{\mb{G}}_{\gamma'})\Delta[\mf{w}_v, z^{iso}_v](\gamma^{H}, \gamma')O_{\gamma'}(f).
\end{equation}
\end{theorem}

We will also need to discuss twisted transfer factors at a finite place $p$ of $\Q$. In particular, we have to consider the group $R_G=\mathrm{Res}_{\mathbb{Q}_{p^r}/\mathbb{Q}_p} \mathbf{G}_{\mathbb{Q}_{p^r}}$. 

If $(H, s, \eta)$ is a refined endoscopic datum of a reductive group $G$ over $\Q_p$, then $(H,s,\eta)$ induces a twisted endoscopic datum of $R_G$ (see \cite[\S A.1.3]{Mor1}). Then, given the transfer factor at $p$ of $G$ defined previously and denoted by $\Delta[\mf{w}_p, z^{\iso}_p]$, we get, following \cite[\S5.5, \S5.6]{KS2} (see also the appendix by Kottwitz of \cite{Mor1}), a twisted transfer factor $\Delta_R[\mf{w}_p, z^{\iso}_p]$. Moreover, $\Delta_R[\mf{w}_p, z^{\iso}_p]$ satisfies the following equality:
\begin{equation}{\label{TOSO}}
    \Delta_R[\mathfrak{w}_p, z^{iso}_p](\gamma_H, \delta)=\langle \alpha(\gamma, \delta), \eta(s) \rangle^{-1}\Delta[\mathfrak{w}_p, z^{iso}_p](\gamma_H, \gamma),
\end{equation}
where $\gamma_H$ is a $(G,H)$-regular semisimple element of $H(\mathbb{Q}_p)$, the element $\gamma \in G(\mathbb{Q}_p)$ is a transfer of $\gamma_H$, and $\delta \in R_G(\Q_p)$ has norm $\gamma$. The term $\alpha(\gamma, \delta)$ is defined in \cite[\S A.3.5]{Mor1}.

\subsection{Point counting formula for Shimura varieties}
We now recall the trace formula for the cohomology of compact Shimura varieties as in \cite{kot7}.

\begin{assumption}
We now make the following further assumptions on $\mb{G}$:
\begin{itemize}
\item The maximal $\Q$-split torus in $Z(\mb{G})$  coincides with the maximal $\R$-split torus in $Z(\mb{G})$.
\item The group $\mb{G}$ is anisotropic modulo center.
\item There exists a Shimura datum $(\mb{G}, X)$ whose trace formula for compactly supported cohomology is known and coincides with Equation \eqref{pointcountingformula}.
\end{itemize}
\end{assumption}

In particular, the second condition will guarantee that the Shimura varieties we work with are compact so that intersection cohomology and compactly supported cohomology coincide.

We now fix a prime $p$ such that $\mb{G}_{\Q_p}$ is unramified and a Shimura datum $(\mb{G}, X)$. Recall that $X$ defines a conjugacy class of cocharacters $\{\mu\}$ of $\mathbf{G}_{\C}$ given by taking the basechange of any $h \in X$ to $\C$ and restricting to the first component. Let $\ov{\Q}$ be an algebraic closure of $\Q$ and suppose $\{ \mu \}$ has reflex field $E \subset \overline{\mathbb{Q}}$. Let $\mf{p}$ be a place of $E$ over $p$. We fix a positive integer $j$ and define $F$ to be the degree $j$ unramified extension of $E_{\mf{p}}$.

Following \cite{kot7}, we let  $\Phi^j_{\mf{p}}$ be the $j$th power of a geometric Frobenius element of $\Gal(\ov{\Q_p}/E_{\mf{p}})$. Since $\mb{G}_{\Q_p}$ is assumed to be unramified, we can pick a model $\mc{G}$ of $\mb{G}_{\Q_p}$ defined over $\Z_p$. Then we define $K_F :=   \mc{G}(\mc{O}_F)$ where $\mc{O}_F$ is the ring of integers of $F$.

Now, we let $\phi_j \in C^{\infty}_c(\mb{G}_{\Q_p}, K_F)$ be the indicator function of $K_F \mu(\pi^{-1}_F)K_F$ where $\pi_F$ is a  uniformizer of $F$.

Let $K^p$ be a compact open subgroup of $\mb{G}(\A^p_f)$ and define $K_p$ to be $\mc{G}(\Z_p)$. Let $K := K^pK_p$. Choose any $f^p \in C^{\infty}_c(\mb{G}(\A^p_f))$ and let $f_p$ be the indicator function of $K_p$. Let $f:= f^pf_p$.

Now, let $L$ be a number field and fix a finite dimensional representation $\xi$ of $\mb{G}$ with values in $L$. For any finite place $\lambda$ of $L$ we get a smooth $L_{\lambda}$-sheaf $\mc{F}_{\lambda}$ in the standard way (see \cite[pg163]{kot7}). Now define $H^*_{\xi} := \sum\limits_i (-1)^i H^i_c(\Sh_K, \mc{F}_{\lambda})$. Then we expect (see \cite[Equation (3.1)]{kot7}) the following formula.

\begin{equation}{\label{pointcountingformula}}
    \mathrm{tr}(f \times \Phi^j_{\mf{p}}| H^*_{\xi})= \sum_{\substack{(\gamma_0 ; \gamma, \delta) \in KT_j \\ \alpha(\gamma_0 ; \gamma, \delta)=1}} c(\gamma_0; \gamma, \delta)O_{\gamma}(f^p)TO_{\delta}(\phi_j)\mathrm{tr}\xi(\gamma_0).
\end{equation}

We refer the reader to \cite{kot7} for the somewhat involved definitions of $\alpha(\gamma_0 ; \gamma, \delta)$ and $c(\gamma_0, \gamma, \delta)$. 

We recall that the set $KT_j$ is defined to consist of triples $(\gamma_0; \gamma, \delta)$ such that $\gamma_0 \in \mb{G}(\Q)$ and is elliptic in $\mb{G}(\R)$, the elements $\gamma \in \mb{G}(\A^p_f)$ and $\delta \in \mb{G}(F)$ are such that for each place $v$ of $\Q$ other than $\infty, p$ we have $\gamma_v$ and $\gamma_0$ are conjugate in $\mb{G}(\ov{\Q_v})$, the norm of $\delta$ is conjugate to $\gamma_0$ in $\mb{G}(\ov{\Q_p})$, and the $\sigma$-conjugacy class of $\delta$ in $\mb{B}(G_{\Q_p})$ satisfies that its image in $X^*(Z(\mb{G})^{\Gamma_p})$ equals the restriction of $-\mu_1 \in X^*(Z(\widehat{G}))$ to $Z(\widehat{G})^{\Gamma_p}$. We recall that $\mu_1$ is defined so that if $\mu \in \{\mu\}$ and $T$ is a maximal torus of $\mb{G}$, then $\mu$ gives us an element of $X^*(\widehat{T})$ which we restrict to $Z(\widehat{G})$ to get  a character $\mu_1$ that is evidently independent of our choice of $T$ and $\mu$.

Then we have that (\cite[Equation (4.1)]{kot7})
\begin{equation}
    1=\frac{1}{|\mf{K}(I^{\mb{G}}_{\gamma_0}/\Q)|} \sum\limits_{\kappa \in \mathfrak{K}(I^{\mb{G}}_{\gamma_0}/\mathbb{Q})} \langle \alpha(\gamma_0; \gamma, \delta), \kappa \rangle^{-1}.
\end{equation}

Thus following \cite{kot7} we can rewrite Equation \eqref{pointcountingformula} as
\begin{equation}{\label{fouriereqn}}
    \tau(\mathbf{G})\sum\limits_{(\gamma_0; \gamma, \delta) \in KT_j}\sum\limits_{\kappa \in \mf{K}(I^{\mb{G}}_{\gamma_0}/\Q)} \langle \alpha(\gamma_0; \gamma, \delta), \kappa \rangle^{-1} e(\gamma, \delta)O_{\gamma}(f^p)TO_{\delta}(\phi_j)\mathrm{tr}\xi(\gamma_0)\mathrm{vol}(A_{\mathbf{G}}(\mathbb{R})^0\setminus I(\infty)(\mathbb{R}))^{-1},
\end{equation}
which is Equation $(4.2)$ of that work. The sign $e(\gamma, \delta)$ is defined to be $\prod\limits_v e(I(v))$ (product over finite and infinite $v$) and $I(\infty)/\mathbb{R}$ is the compact modulo center inner form of $I_0$ and $I(v)$ is defined in \cite[\S2]{kot7}.

Our goal now is to manipulate the above equation so that it is compatible with refined endoscopic data via the isomorphism of Lemma \ref{SSeqEQ}. We first record the following lemma.
\begin{lemma}
For all semisimple $\gamma_0 \in \mb{G}(\Q)$, we have a surjection 
\begin{equation}
Z(\widehat{I^{\mb{G}}_{\gamma_0}})^{\Gamma_{\Q}} \twoheadrightarrow \mf{K}(I^{\mb{G}}_{\gamma_0}/ \Q).
\end{equation}
\end{lemma}
\begin{proof}
By definition, $\mf{K}(I^{\mb{G}}_{\gamma_0}/ \Q)$ is the subgroup of $(Z(\widehat{I^{\mb{G}}_{\gamma_0}})/Z(\widehat{\mb{G}}))^{\Gamma_{\Q}}$ whose image in $H^1( \Q, \mb{G})$ is locally trivial. Since we are assuming $\mb{G}$ satisfies the Hasse principle, it follows that the image of $\mf{K}(I^{\mb{G}}_{\gamma_0}/ \Q)$ under the natural map to $H^1(\Q, \mb{G})$ is trivial and hence that $\mf{K}(I^{\mb{G}}_{\gamma_0}/ \Q)$ lies in the image of $Z(\widehat{I^{\mb{G}}_{\gamma_0}})^{\Gamma_{\Q}} / Z(\widehat{\mb{G}})^{\Gamma_{\Q}}$. This implies the desired result.
\end{proof}
The above lemma implies that both vertical arrows in the diagram in the statement of Corollary \ref{endcomp} are surjections. By the Hasse principle for $\mb{G}$, we also have a surjection 
\begin{equation}
    \mc{E}^r(\mb{G}) \twoheadrightarrow \mc{E}(\mb{G}).
\end{equation}

\begin{construction}{\label{endoreps}}
We now pick a section $\mc{S}: \mc{EQ}(\mb{G}) \to \mc{EQ}^r(\mb{G})$ of the projection $\mc{EQ}^r(\mb{G}) \twoheadrightarrow \mc{EQ}(\mb{G})$. Moreover, we can and do choose this section so that if $[\mb{H},s,\eta, \gamma_{\mb{H}}], [\mb{H}', s',\eta',\gamma_{\mb{H}'}] \in \mc{EQ}(\mb{G})$ map to the same isomorphism class of endoscopic datum under the map $\mc{EQ}(\mb{G}) \to \mc{E}(\mb{G})$ then $\mc{S}([\mb{H},s,\eta, \gamma_H])$ and $\mc{S}([\mb{H}', s',\eta', \gamma_{\mb{H}'}])$ project to the same refined endoscopic isomorphism class under the projection $\mc{EQ}^r(\mb{G}) \to \mc{E}^r(\mb{G})$. This induces a partial section of the map $\mc{E}^r(\mb{G}) \to \mc{E}(\mb{G})$ (defined on the images of the $\mc{EQ}$ maps). We extend this to a full section of $\mc{E}^r(\mb{G}) \to \mc{E}(\mb{G})$ and pick a representative of each equivalence class of $\mc{E}^r(\mb{G})$ appearing in the image of this section. We denote the set of these representatives by $X^{\mf{e}}$.
\end{construction}
\begin{proof}
To see that we can pick $\mc{S}$ to be compatible with the projection $\mc{E}^r(\mb{G}) \to \mc{E}(\mb{G})$, note that for any $(\mb{H}, s, \eta, \gamma_{\mb{H}})$ and $(\mb{H}', s', \eta', \gamma_{\mb{H}'})$ that project to the same class of $\mc{E}(\mb{G})$, the endoscopic data $(\mb{H}, s, \eta)$ and $(\mb{H}', s', \eta')$ are isomorphic and their isomorphism class lifts to $\mc{E}^r(\mb{G})$. Hence we may as well assume that $(\mb{H},s,\eta) = (\mb{H}', s', \eta')$ and that $(\mb{H}, s, \eta)$ is a refined endoscopic datum. Then we can define $\mc{S}$ on the fiber of $[\mb{H}, s, \eta] \in \mc{E}(\mb{G})$ by choosing representatives of the isomorphism classes in $\mc{EQ}(\mb{G})$  and defining $(\mb{H}, s, \eta, \gamma_{\mb{H}}) \mapsto (\mb{H}, s, \eta, \gamma_{\mb{H}})$.
\end{proof}

We record the following lemma.
\begin{lemma}
The above construction gives, for each  $\gamma_0 \in \mb{G}(\Q)$ and $\kappa \in \mf{K}(I^{\mb{G}}_{\gamma_0}/\Q)$, a natural lift $\lambda \in Z(\widehat{I^{\mb{G}}_{\gamma_0}})^{\Gamma_{\Q}}$ of $\kappa$.
\end{lemma}
\begin{proof}
This is immediate from our choice of $\mc{S}$ above and Corollary \ref{endcomp}.
\end{proof}
Denote by $X_{\gamma_0} \subset Z(\widehat{I^{\mb{G}}_{\gamma_0}})^{\Gamma_{\Q}}$ the corresponding collection of lifts that we get by the previous lemma. Finally, for each representative $(\mb{H}, s, \eta) \in X^{\mf{e}}$, since $\mb{G}_{\der}$ is simply connected, we may extend $\eta$ to a map of $L$-groups
\begin{equation*}
\Leta: \, ^L\mb{H} \to \, ^L\mb{G}.    
\end{equation*}

Now, we can rewrite our equation for $\mathrm{tr}(f \times \Phi^j_{\mf{p}}| H^*_{\xi})$ as \begin{equation}{\label{prestab}}
        \tau(\mathbf{G})\sum\limits_{(\gamma_0; \gamma, \delta) \in KT_j}\sum\limits_{\lambda \in X_{\gamma_0}} \langle \alpha(\gamma_0; \gamma, \delta), \kappa \rangle^{-1} e(\gamma, \delta)O_{\gamma}(f^p)TO_{\delta}(\phi_j)\mathrm{tr}\xi(\gamma_0)\mathrm{vol}(A_{\mathbf{G}}(\mathbb{R})^0\setminus I(\infty)(\mathbb{R}))^{-1},
\end{equation}
where $\kappa$ is the projection of $\lambda$ to $\mf{K}(I^{\mb{G}}_{\gamma_0}/\Q)$.

We are now ready to begin the discussion of the stabilization of Equation \eqref{prestab}. Let $(\mb{H}, s, \Leta)$ be a refined endoscopic datum in the set of lifts chosen previously and $\gamma^{\mb{H}} \in \mb{H}(\Q)$ a semisimple element. The contribution to \eqref{prestab} of $(\mb{H},s, \Leta, \gamma^{\mb{H}})$ is defined to be $0$ unless
\begin{enumerate}
    \item $(\mb{H}, s, \eta)$ is elliptic,
    \item $\mb{H}$ is unramified at $p$,
    \item Each elliptic maximal torus of $\mb{G}_{\R}$ transfers from $\mb{H}_{\R}$,
    \item The element $\gamma^{\mb{H}}$ is $(\mb{G}, \mb{H})$-regular,
    \item For each place $v$ of $\Q$, the element $\gamma^{\mb{H}}$ transfers to an element of $\mb{G}(\Q_v)$,
    \item The element $\gamma^{\mb{H}}$ is elliptic in $\mb{H}(\R)$ (and so in particular elliptic in $\mb{H}(\Q)$).
    \item The class $[\mb{H}, s^{\mb{H}}, \eta, \gamma^{\mb{H}}]$ appears in the image of the section $\mc{S}$ defined previously.
\end{enumerate}
If these conditions are satisfied, then we define the contribution of $(\mb{H}, s, \Leta, \gamma^{\mb{H}})$ in the following sections.

\subsection{Stabilization at infinity}
We review Kottwitz's description starting on page 182 of \cite{kot7} of the stabilization at infinity of Equation \eqref{prestab}.

Fix an elliptic maximal torus $T$ of $\mathbf{G}_{\mathbb{R}}$ and an elliptic maximal torus $T_{\mathbf{H}_{\mathbb{R}}}$ of $\mathbf{H}_{\mathbb{R}}$ such that $T_{\mb{H}_{\R}}$ transfers to $T$. This is possible by our assumptions on $(\mb{H}, s, \eta)$. Fix a pair $(j, B)$ so that $B$ is a Borel subgroup of $\mathbf{G}_{\mathbb{C}}$ containing $T_{\mathbb{C}}$ and $j: T_H \to T$ is an isomorphism in the canonical conjugacy class. The pair $(j, B)$ gives a natural choice of Borel subgroup $B_H$ of $\mathbf{H}_{\mathbb{C}}$.

We let $\phi: W_{\mathbb{R}} \to ^L\mathbf{G}$ be an elliptic Langlands parameter whose packet consists of discrete series representations with the same central and infinitesimal character as the contragredient of $\xi$. 

Pick $\phi_H$ so that $\Leta \circ \phi_H$ is equivalent to $\phi$ and denote the set of such parameters by $\Phi_H$. Now (using the notation of \cite{kot7}) there is a unique $w_* \in \Omega_*$ so that $(w^{-1}_* \circ j, B, B_H)$ is aligned with $\phi_H$.

Then define
\begin{equation*}
    f(\phi_H)= \frac{1}{|\Omega_H/\Omega_{H(\mathbb{R})}|}\sum\limits_{\pi_H} f(\pi_H),
\end{equation*}
where $f(\pi_H)$ is a pseudocoefficient for $\pi_H$ and the sum ranges over the $L$-packet $\Pi(\phi_H)$. Then we define
\begin{equation}
    f^{\mathbf{H}}_{\xi}=\langle \mu_h, s_{\mb{H}} \rangle \sum\limits_{\phi_H \in \Phi_H} \det(w_*(\phi_H))(-1)^{q(G)}f(\phi_H),
\end{equation}
where $\mu_h$ is the cocharacter associated to some $h \in X$.

Now, we let $\Delta_{j,B}$ be the Shelstad normalization of the local transfer factor at $\infty$. Then Kottwitz shows that
\begin{equation}
    SO_{\gamma^{\mb{H}}}(f^{\mathbf{H}}_{\xi})=\langle \beta(\gamma), \lambda \rangle^{-1} \Delta_{j,B}(\gamma^{\mb{H}}, \gamma) \mathrm{tr} \xi_{\mathbb{C}}(\gamma)\mathrm{vol}(A_{\mathbf{G}(\mathbb{R})^0 \setminus I(\mathbb{R})})^{-1},
\end{equation}
where $\gamma \in \mb{G}(\R)$ is some transfer of $\gamma^{\mb{H}}$ and $\beta(\gamma)$ is defined by restricting $\mu_h$ to an element of $X^*(Z(\widehat{I^{\mb{G}}_{\gamma}})^{\Gamma_{\R}})$. Kottwitz shows this does not depend on our choice of $\mu_h \in \{ \mu \}$. The element $\lambda$ is the image of $s$ in $Z(\widehat{I^{\mb{G}}_{\gamma}})^{\Gamma_{\R}}$.

Since local transfer factors differ up to a scalar, we have 
\begin{equation*}
\Delta[\mathfrak{w}_{\infty}, z^{iso}_{\infty}]=b_H\Delta_{j,B}
\end{equation*}
for some complex number  $b_H$. We could try to determine the value of $b_H$ explicitly (as in Yihang Zhu's thesis \cite{Zhu}) but we choose instead to carry around the constant $b_H$ as it will eventually cancel.

We define $f^{'\mathbf{H}}_{\xi} := b_H f^{\mathbf{H}}_{\xi}$ so that
\begin{equation}
    SO_{\gamma^{\mb{H}}}(f^{'\mathbf{H}}_{\xi}) =b_HSO_{\gamma^{\mb{H}}}(f^{\mathbf{H}}_{\xi})=\langle \beta(\gamma), \lambda \rangle^{-1} \Delta[\mathfrak{w}_{\infty}, z^{iso}_{\infty}](\gamma^{\mb{H}}, \gamma) \mathrm{tr} \xi_{\mathbb{C}}(\gamma)\mathrm{vol}(A_{\mathbf{G}(\mathbb{R})^0 \setminus I(\mathbb{R})})^{-1}.
\end{equation}
\subsection{Stabilization for \texorpdfstring{$v \neq p, \infty $}{v not p or infty}}
Let $f^p \in C^{\infty}_c(G(\mathbb{A}^{p, \infty}))$. Then by the fundamental lemma, we get $f^{\mathbf{H}, p} \in C^{\infty}_c(\mathbf{H}(\mathbb{A}^{p, \infty}))$ so that
\begin{equation}{\label{stabawayfromp}}
    SO_{\gamma^{\mb{H}}}(f^{\mathbf{H}, p})= \sum\limits_{\gamma} e^p(I^{\mb{G}}_{\gamma}) \Delta[\mathfrak{w}^{p, \infty}, z^{iso, p, \infty}](\gamma^{\mb{H}}, \gamma)O_{\gamma}(f^{p, \infty}),
\end{equation}
where the sum is over $\gamma \in \mb{G}(\A^p_f)$ that are transfers of $\gamma^{\mb{H}}$ and $e^p(I^{\mb{G}}_{\gamma})= \prod\limits_v e(I^{\mb{G}_{\Q_v}}_{\gamma_v})$.
\subsection{Stabilization at \texorpdfstring{$p$}{p}}
following \cite[pg180]{kot7}, we construct from $\phi_j$ a function $f^{\mb{H}}_p$ that is an element of $\mc{H}(\mb{H}(\Q_p), K_H)$ times a character on $
\mb{H}(\Q_p)$, where $K_H$ is any hyperspecial subgroup of $\mb{H}(\Q_p)$. Kottwitz shows we get the following equality which is a case of the twisted fundamental lemma

\begin{theorem}
We have the following equality
\begin{equation}
    SO_{\gamma_H}(f^{\mb{H}}_p)=\sum\limits_{\delta} e(J_\delta)\Delta_R[\mathfrak{w}_p, z^{iso}_p](\gamma^{\mb{H}}, \delta)TO_{\delta}(\phi_j),
\end{equation}
where $J_{\delta}$ is the twisted centralizer of $\delta$ and the sum is over twisted conjugacy classes of $\delta$ whose stable norm is equal to a fixed choice of $\gamma \in \mb{G}(\Q_p)$ which is a transfer of $\gamma^{\mb{H}}$.
\end{theorem}

Using Equation \eqref{TOSO}, we can rewrite the above as
\begin{equation}
    SO_{\gamma_H}(f^{\mb{H}}_p)= \sum\limits_{\delta}e(J_{\delta}) \langle \alpha(\gamma;\delta), \lambda \rangle^{-1} \Delta[\mathfrak{w}_p, z^{iso}_p](\gamma^{\mb{H}}, \gamma)TO_{\delta}(\phi_j).
\end{equation}
\subsection{Stabilized formula}
We now combine the contributions described in the previous three sections.

We first recall that as proven on page 188 of \cite{kot7}, by our assumptions on $\gamma^{\mb{H}}$ there exists a semisimple $\gamma_0 \in \mb{G}(\Q)$ such that $\gamma^{\mb{H}}$ transfers to $\gamma_0$.

Consider the function $f^{\mathbf{H}}:=f^{\mathbf{H},p}f^{'\mathbf{H}}f^{\mb{H}}_p \in C^{\infty}_c(\mathbf{H}(\mathbb{A}))$. We now combine the equations of the previous three sections to give a formula for $SO_{\gamma^{\mb{H}}}(f^{\mb{H}})$. We keep in mind the following facts.
\begin{itemize}
    \item At $p$ and $\infty$, we chose elements $\gamma$ in $\mb{G}(\Q_p)$ and $\mb{G}(\R)$ respectively that were transfers of $\gamma^{\mb{H}}$. We are free to (and do) choose $\gamma$ to be the image of $\gamma_0$ in the respective groups,
    \item At each place $v \neq p, \infty$ the fundamental lemma gave us a sum over conjugacy classes of $\gamma \in \mb{G}(\Q_v)$ such that $\gamma$ is a transfer of $\gamma^{\mb{H}}$. Note that (see \cite[\S5.6]{Kot6}) since $\gamma, \gamma_0$ are stably conjugate in $\mb{G}(\ov{\Q_v})$, we have 
\begin{equation*}
    \Delta[\mathfrak{w}_v, z^{\iso}_v](\gamma^{\mb{H}}, \gamma)=\langle \inv[z^{\iso}_v](\gamma_0, \gamma), \lambda \rangle^{-1} \Delta[\mathfrak{w}_v, z^{\iso}_v](\gamma^{\mb{H}}, \gamma_0),
\end{equation*}
and that by definition, $\alpha_v(\gamma_0 ; \gamma, \delta)$ as defined in \cite{kot7} equals $ \inv(\gamma, \gamma_0)$ as we have defined it here. Hence we could rewrite Equation \eqref{stabawayfromp} as
\begin{equation}
    SO_{\gamma_H}(f^{\mathbf{H}, p})= \sum\limits_{\gamma} e^p(I^{\mb{G}}_{\gamma}) \prod\limits_{v \neq p, \infty} \langle \alpha_v(\gamma_0 ; \gamma, \delta), \lambda \rangle^{-1} \Delta[\mathfrak{w}^{p, \infty}, z^{iso, p, \infty}](\gamma^{\mb{H}}, \gamma_0)O_{\gamma}(f^{p, \infty}).
\end{equation},
\item By construction, (see \cite{kot7}) we have the following
\begin{equation*}
    \langle \alpha(\gamma_0; \gamma, \delta), \kappa \rangle=\langle\alpha(\gamma_0; \delta), \lambda \rangle\langle \beta(\gamma_0), \lambda \rangle\prod\limits_{v \neq p, \infty}\langle \alpha_v(\gamma_0, \gamma, \delta), \lambda \rangle.
\end{equation*}
\end{itemize}

Hence, we get the following formula
\begin{align*}
    &SO_{\gamma^{\mb{H}}}(f^{\mathbf{H}})=\\
    &\sum\limits_{(\gamma, \delta)} e^p(I^{\mb{G}}_{\gamma})e(J_{\delta})\langle \alpha(\gamma_0; \gamma, \delta), \kappa \rangle^{-1} O_{\gamma}(f^{p, \infty})TO_{\delta}(\phi_j)\mathrm{tr} \xi_{\mathbb{C}}(\gamma_0)\mathrm{vol}(A_{\mathbf{G}(\mathbb{R})^0} \setminus I(\mathbb{R}))^{-1},
\end{align*}
where the sum is over conjugacy clases of $\gamma \in \mb{G}(\A^p_f), \delta \in \mb{G}(F)$ such that $(\gamma_0 ; \gamma, \delta) \in \KT_j$.

We now observe that we have a map
\begin{equation*}
    \coprod\limits_{\mathcal{H}^{\mathfrak{e}} \in \mathcal{E}_r(\mathbf{G})} \mathbf{H}(\mathbb{Q})_{(\mathbf{G}, \mathbf{H})-reg, ss}/ \sim_{st} \to \mcSS^r(\mb{G}) \cup \{\emptyset\},
\end{equation*}
given by mapping a $(\mb{G},\mb{H})$-regular $\gamma^{\mb{H}}$ to its equivalence class $(\mb{H}, s, \eta, \gamma^{\mb{H}}) \in \mcEQ^r(\mb{G})$ and then applying Lemma \ref{SSeqEQ} (we define the image of $\gamma^{\mb{H}}$ to be $\emptyset$ if it does not transfer to $\mb{G}$. In particular, for each $(\gamma_0, \lambda) \in SS^r(\mb{G})$ there is a unique isomorphism class $\mc{H}^{\mf{e}}$ of refined endoscopic datum containing the pre-image of $(\gamma_0, \lambda)$ under this map and by the proof of Lemma \ref{stabigusalem}, the fiber has cardinality $|\mathrm{Out}_r(\mc{H}^{\mf{e}})|$. 

We now define the set $X^{ss}_{\mc{H}^{\mf{e}}}$ to consist of those stable conjugacy classes  $\gamma_{\mb{H}} \in \mb{H}(\Q)_{(\mb{G}, \mb{H})-reg, ss}/ \sim st$ such that $[\mb{H}, s^{\mb{H}}, \eta, \gamma_{\mb{H}}]$ is in the image of the section $\mc{S}$.

In particular, we have
\begin{align*}
    &\sum\limits_{\mathcal{H}^{\mathfrak{e}} \in X^{\mf{e}}} \frac{\tau(\mathbf{G})}{|\mathrm{Out}_r(\mathcal{H}^{\mathfrak{e}})|} \sum\limits_{\gamma_{\mathbf{H}} \in X^{ss}_{\mc{H}^{\mf{e}}}} SO_{\gamma_{\mathbf{H}}}(f^{\mathbf{H}})\\
    &=\tau(\mathbf{G})\sum\limits_{(\gamma_0; \gamma, \delta) \in \KT_j}\sum\limits_{\lambda \in X_{\gamma_0}} e^p(I^{\mb{G}}_{\gamma})e(J_{\delta})\langle \alpha(\gamma_0; \gamma, \delta), \kappa \rangle^{-1} O_{\gamma}(f^{p, \infty})TO_{\delta}(\phi_j)\mathrm{tr} \xi_{\mathbb{C}}(\gamma_0)\mathrm{vol}(A_{\mathbf{G}(\mathbb{R})^0} \setminus I(\mathbb{R}))^{-1},
\end{align*}
where, as discussed on page 188 of \cite{kot7}, $SO_{\gamma^{\mb{H}}}(f^{\mb{H}})$ is $0$ unless $\gamma^{\mb{H}}$ is actually elliptic over $\R$.

If we define
\begin{equation}{\label{STelldef}}
    ST^{\mathbf{H}}_{r, ell}(f^H) := \tau(\mathbf{H})\sum\limits_{\gamma_{\mathbf{H}} \in X^{ss}_{\mc{H}^{\mf{e}}}} SO_{\gamma_{\mathbf{H}}}(f^{\mathbf{H}}),
\end{equation}
and
\begin{equation*}
    \iota_r(\mathbf{G}, \mathbf{H}) := \frac{\tau(\mathbf{G})}{\tau(\mathbf{H})|\mathrm{Out}_r(\mathcal{H}^{\mathfrak{e}})|},
\end{equation*}
then finally we have the formula
\begin{equation}
     \mathrm{tr}(f \times \Phi^j_{\mf{p}}| H^*_{\xi})=\sum\limits_{\mathcal{H}^{\mathfrak{e}}\in X^{\mf{e}}} \iota_r(\mathbf{G},\mathbf{H}) ST^{\mathbf{H}}_{r, ell}(f^{\mathbf{H}}).
\end{equation}

We note that this agrees with the formula arrived at in \cite{kot7}. Indeed, if $\gamma_{\mb{H}}$ and $\gamma'_{\mb{H}}$ transfer to the same element of $\mb{G}(\Q)$, then their stable orbital integrals agree. Then counting the relative number of stable conjugacy classes in each term, we get
\begin{equation*}
    ST^{\mb{H}}_{ell, r}(f^{\mb{H}})=\frac{\Out_r(\mc{H}^{\mf{e}})}{\Out(\mc{H}^{\mf{e}})}ST^{\mb{H}}_{ell}(f^{\mb{H}}),
\end{equation*}
as desired.

\subsection{Destabilization of the cohomology of Shimura varieties}

We are now ready to derive an expression for the cohomology of Shimura varieties following \cite{kot7}.

For each elliptic $\mathcal{H}^{\mathfrak{e}} \in X^{\mf{e}}$, we define $ST^{\mathbf{H}}_{r, ell}(f^{\mathbf{H}})$ as in Equation \eqref{STelldef}. We define $ST^{\mathbf{H}}_{disc}(f^{\mathbf{H}})$ by
\begin{equation}
    ST^{\mb{H}}_{disc}(f^{\mathbf{H}})=\sum\limits_{[\psi_{\mathbf{H}}]} \frac{1}{|\mathcal{S}_{\psi^{\mathbf{H}}}|}  \sum\limits_{\pi^{\mathbf{H}} \in \Pi_{\psi^{\mathbf{H}}}} \epsilon_{\psi^{\mathbf{H}}}(s_{\psi^{\mathbf{H}}})\langle \pi^{\mathbf{H}}, s_{\psi^{\mathbf{H}}} \rangle \mathrm{tr}(\pi \mid f^{\mathbf{H}}),
\end{equation}
where the sum is over equivalence classes of discrete Arthur parameters of $\mb{H}$ such that the associated character on $A_{\mb{G}}(\R)^0=A_{\mb{H}}(\R)^0$ equals $\chi$.

Then we make the following assumption.
\begin{assumption}{\label{STELLA}}
Let $f^{\mb{H}} \in \mc{H}(\mb{H})$. We assume that 
\begin{equation}
    ST^{\mathbf{H}}_{ell}(f^{\mathbf{H}})=ST^{\mathbf{H}}_{disc}(f^{\mathbf{H}}).
\end{equation}
\end{assumption}

Our temporary goal is to rewrite this expression for $ST^{\mb{H}}_{disc}(f^{\mathbf{H}})$ as a sum of traces of representations of $\mathbf{G}$.

First we consider the situation away from $p, \infty$. Then by Equation \eqref{endocharid}, we have the following.
\begin{align*}
    &\sum\limits_{\pi^{\mathbf{H}, p, \infty} \in \Pi_{\psi^{\mathbf{H}, p, \infty}}(\mathbf{H}^{p, \infty}, 1)} \langle \pi^{\mathbf{H},p, \infty}, s_{\psi^{\mathbf{H},p, \infty}} \rangle \mathrm{tr}(\pi^{\mathbf{H},p, \infty} \mid f^{\mathbf{H},p, \infty})\\
    &=e(\mathbf{G}^{p, \infty}) \sum\limits_{\pi^{p, \infty} \in \Pi_{\psi^{p, \infty}}(\mathbf{G}^{p, \infty}, \varrho^{p, \infty})} \langle \pi^{p, \infty}, \overline{\eta(s)}s_{\psi^{p, \infty}} \rangle \mathrm{tr}(\pi^{p, \infty} \mid f^p),
\end{align*}
where $\overline{\eta(s)}$ is the projection of $\eta(s)$ to $S^{\natural}_{\psi_v}$ and $e(\mb{G}^{p, \infty})= \prod\limits_{v \neq p, \infty} e(\mb{G}_v)$.

We review the destabilization at $p$ carried out in \cite[pg 193]{kot7}. We first have the assumption:
\begin{assumption}{\label{unramassump}}
If $\psi_p$ is an unramified Arthur parameter for $\mb{G}_p$ and $K_p \subset \mb{G}(\Q_p)$ is a hyperspecial  subgroup, then $\Pi(\psi_p, \varrho_p)$ contains a unique representation unramified with respect to $K_p$. 

On the other hand, if $\psi_p$ is a ramified Arthur parameter, the $\Pi(\psi_p, \varrho_p)$ contains no unramified representation with respect to $K_p$.
\end{assumption}

Now let $\phi_{\psi}$ be the $L$-parameter defined by $\phi_{\psi}(w)=\psi(w, \begin{pmatrix} |w|^{\frac{1}{2}} & 0 \\ 0 & |w|^{-\frac{1}{2}} \end{pmatrix})$. Then following the argument of \cite[pg 193]{kot7}, we get the formula:
\begin{align*}
       &\sum\limits_{\pi^{\mathbf{H}}_p \in \Pi_{\psi^{\mathbf{H}}_p}(\mb{H}_p,1)} \langle \pi^{\mathbf{H}}_p, s_{\psi^{\mathbf{H}}_p} \rangle \mathrm{tr} (\pi^{\mathbf{H}}_p \mid f^{\mathbf{H}}_p)\\
    &=\mathrm{tr}(r_{- \mu} \circ \eta \circ \phi_{\psi_H}|_{W_{E_p}}|\cdot|^{-\mathrm{dim}Sh/2} \mid \eta(s)\Phi^j_{\mf{p}})e(\mathbf{G}_p) \langle \pi_p, \overline{\eta(s)}s_{\psi_p} \rangle.
\end{align*}
We denote by $A(\psi, s^{\mb{H}}_p, \Phi^j_{\mf{p}})$ the term $\mathrm{tr}(r_{- \mu} \circ \eta \circ \phi_{\psi_H}|_{W_{E_p}}|\cdot|^{-\mathrm{dim}Sh/2} \mid \eta(s)\Phi^j_{\mf{p}})$.

Finally we study the situation at $\infty$. We recall that our normalization of the transfer factor at the Archimedean place differs by a constant $b_{\mathbf{H}}$ and that we have defined our function $f^{' \mathbf{H}}_{\xi}$ by
\begin{equation*}
    f^{' \mathbf{H}}_{\xi}=b_{\mathbf{H}}f^{\mathbf{H}}_{\xi}.
\end{equation*}
From Kottwitz's computation in \cite{kot7} we get
\begin{equation}
    \sum\limits_{\pi^{\mathbf{H}}_{\infty} \in \Pi_{\psi^{\mathbf{H}}_{\infty}}(\mb{H}_{\infty}, 1)} \langle \pi^{\mathbf{H}}_{\infty}, s_{\psi^{\mathbf{H}}_{\infty}} \rangle \mathrm{tr}( \pi^{\mathbf{H}}_{\infty} \mid f^{\mathbf{H}}_{\xi})=e(\mathbf{G}_{\infty})(-1)^{q(\mathbf{G})}\langle \lambda_{\pi_{\infty}}, \ov{\eta(s)}s_{\psi} \rangle\langle \pi_{\infty}, \overline{\eta(s)} s_{\psi} \rangle'.
\end{equation}
In particular, Kottwitz shows (\cite[Lemma 9.2]{kot7}) that $\langle \lambda_{\pi_{\infty}}, \eta(s)s_{\psi} \rangle\langle \pi_{\infty}, \overline{\eta(s)s_{\psi}} \rangle'$ is independent of the choice of $\pi_{\infty} \in \Pi_{\psi_{\infty}}(\mathbf{G}_{\infty}, \varrho_{\infty})$. We recall that $\lambda_{\pi_{\infty}}$ is a character of $S_{\psi_{\infty}}$ defined by Kottwitz in \cite[pg 195]{kot7} (in particular we remark that $\lambda_{\pi_{\infty}}$ is not a character of $ \mathcal{S}_{\psi}$). In the above formula, $\langle  \cdot, \cdot \rangle'$ refers to the pairing determined by Shelstad's normalization of transfer factors. By Equation \cite[(5.16)]{Kal2}, we get
\begin{equation*}
   b_{\mathbf{H}}\langle \pi_{\infty}, \overline{\eta(s)} \rangle'=\langle \pi_{\infty}, \overline{\eta(s)} \rangle.
\end{equation*}
Thus, we have
\begin{align*}
    &\sum\limits_{\pi^{\mathbf{H}}_{\infty} \in \Pi_{\psi^{\mathbf{H}}_{\infty}}(\mb{H}_{\infty}, 1)} \langle \pi^{\mathbf{H}}_{\infty}, s_{\psi^{\mathbf{H}}_{\infty}} \rangle \mathrm{tr}( \pi^{\mathbf{H}}_{\infty} \mid f^{',\mathbf{H}}_{\xi})\\
    &=b_{\mathbf{H}}\sum\limits_{\pi^{\mathbf{H}}_{\infty} \in \Pi_{\psi^{\mathbf{H}}_{\infty}}(\mb{H}_{\infty}, 1)} \langle \pi^{\mathbf{H}}_{\infty}, s_{\psi^{\mathbf{H}}_{\infty}} \rangle \mathrm{tr}( \pi^{\mathbf{H}}_{\infty} \mid f^{\mathbf{H}}_{\xi})\\
    &=b_{\mathbf{H}}e(\mathbf{G}_{\infty})(-1)^{q(\mathbf{G})}\langle \lambda_{\pi_{\infty}}, \eta(s)s_{\psi} \rangle\langle \pi_{\infty}, \overline{\eta(s)}s_{\psi} \rangle'\\
    &=e(\mathbf{G}_{\infty})(-1)^{q(\mathbf{G})}\langle \lambda_{\pi_{\infty}}, \eta(s)s_{\psi} \rangle\langle \pi_{\infty}, \overline{\eta(s)}s_{\psi} \rangle.
\end{align*}
Finally we recall Arthur's identity \cite[Lemma 7.1]{Art2}
\begin{equation*}
    \epsilon_{\psi^{\mathbf{H}}}(s_{\psi^{\mathbf{H}}})=\epsilon_{\psi}(\eta(s)s_{\psi}).
\end{equation*}

Now that we have discussed the destabilizations at each place, we show how we can combine them to get a formula for the cohomology of Shimura varieties. 

Recall that in the previous section we chose a section $\mc{S}: \mc{EQ}(\mb{G}) \to \mc{EQ}^r(\mb{G})$. Analogously, we now pick a section $\mc{R}: \mc{EP}(\mb{G}) \to \mc{EP}^r(\mb{G})$. We choose it so that the following diagram commutes:
\begin{equation}
\begin{tikzcd}
\mc{EP}^r(\mb{G}) \arrow[r] & \mc{E}^r(\mb{G}) \\
\mc{EP}(\mb{G}) \arrow[u, "\mc{R}"] \arrow[r] & \mc{E}(\mb{G}) \arrow[u]
\end{tikzcd}    
\end{equation}
where the map $\mc{E}(\mb{G}) \to \mc{E}^r(\mb{G})$ is the section we chose previously. By Corollary \ref{paramcompat} we get a set of lifts of the elements of $\mc{SP}(\mb{G})$ to  $\mc{SP}^r(\mb{G})$. In particular, an element in the image of this lift consists of a pair $([\psi], [x])$ where $[\psi]$ is an equivalence class of parameters and $[x]$ is a conjugacy class in $C_{\psi}$. We choose a representative $(\psi, x)$ of each class, chosen such that $x= \eta(s)$ for the corresponding element of $X^{\mf{e}}$ and refer to the set of such representatives as $X_{\Psi}$. We denote by $X^{sp}_{\mc{H}^{\mf{e}}}$ the set of equivalence classes of parameters $[\psi^{\mb{H}}]$ such that $[\mb{H}, s, \Leta, \psi^{\mb{H}}] \in \mc{R}(\mc{EP}(\mb{G}))$.

We now define a refined version of $ST^{\mb{H}}_{disc}$ given by
\begin{equation}
ST^{\mb{H}}_{r, disc}(f^{\mb{H}}) := \sum\limits_{[\psi_{\mathbf{H}}] \in X^{sp}_{\mc{H}^{\mf{e}}}}  \frac{1}{|\mc{S}_{\psi^{\mb{H}}}|}  \sum\limits_{\pi^{\mathbf{H}} \in \Pi_{\psi^{\mathbf{H}}}} \epsilon_{\psi^{\mathbf{H}}}(s_{\psi^{\mathbf{H}}})\langle \pi^{\mathbf{H}}, s_{\psi^{\mathbf{H}}} \rangle \mathrm{tr}(\pi \mid f^{\mathbf{H}}).
\end{equation}

The following lemma (and its analogue in the non-refined case) proves that 

\begin{equation}
    ST^{\mb{H}}_{r, disc}(f^{\mb{H}}) = \frac{\Out_r(\mc{H}^{\mf{e}})}{\Out(\mc{H}^{\mf{e}})} ST^{\mb{H}}_{disc}(f^{\mb{H}}).
\end{equation}

To introduce the lemma, we first recall (\cite[\S 5]{Kot5}) that Kottwitz defines the notion of $Z(\widehat{\mathbf{G}})$-equivalence for parameters of $\mathcal{H}^{\mathfrak{e}}$. We recall that $\psi^{\mathbf{H}}, \psi^{\mathbf{H} '}$ are $Z(\widehat{\mathbf{G}})$-equivalent if they are conjugate up to a locally trivial continuous $\mathcal{L}_{\mathbb{Q}_p}$-cocycle valued in $Z(\widehat{\mathbf{G}})$. Clearly $Z(\widehat{\mathbf{G}})$-equivalent parameters are also equivalent and following Kottwitz we denote the  $Z(\widehat{\mathbf{G}})$-equivalence class of $\psi^{\mathbf{H}}$ by $[[\psi^{\mathbf{H}}]]$.

Then Kottwitz defines a group $_\mathbf{G}S_{\psi^{\mathbf{H}}}$ consisting of the $g \in \widehat{\mb{H}}$ such that  $g \psi^{\mb{H}} g^{-1}(\psi^{\mb{H}})^{-1}$ is a locally trivial $1$-cocycle of $\mc{L}_{\Q} \times \SL_2(\C)$  valued in $ Z(\widehat{\mb{G}})$. 

Then $_\mathbf{G}\mc{S}_{\psi^{\mathbf{H}}}$ is defined to be $_\mathbf{G}S_{\psi^{\mathbf{H}}}/{_\mathbf{G}S^{\circ}_{\psi^{\mathbf{H}}}}Z(\widehat{\mb{G}})$ and Kottwitz shows that the number of $Z(\widehat{\mathbf{G}})$-equivalence classes in the equivalence class of $\psi^{\mathbf{H}}$ is
\begin{equation*}
    \frac{\tau(\mathbf{G})|  _{\mathbf{G}}\mathcal{S}_{\psi^{\mathbf{H}}}|}{\tau(\mathbf{H})|\mathcal{S}_{\psi^{\mathbf{H}}}|}.
\end{equation*}
\begin{lemma}{\label{paramcountinglem}}
From Lemma \ref{EPeqSP}, we have a map given by the composition
\begin{equation}
    Y: \coprod\limits_{\mc{H}^{\mf{e}} \in X^{\mf{e}}} [[\psi^{\mb{H}}]] \to  \mc{EP}^r(\mb{G}) \to \mc{SP}^r(\mb{G}).
\end{equation}
The pre-image of $(\psi, x) \in X_{\Psi}$ under this map has contributions from a single $\mc{H}^{\mf{e}} \in X^{\mf{e}}$ with $\mc{H}^{\mf{e}}=(\mb{H}, s, \, ^L\eta)$ (where $x= \eta(s)$) and is a set of cardinality 
\begin{equation}
    | \Out_r(\mb{H}, s ,\eta)| \cdot |\frac{\eta(C_{\psi^\mb{H}})}{Z_{C_{\psi}(x)}}|.
\end{equation}
\end{lemma}
\begin{proof}
The class of $(\psi, x)$ in $\mc{SP}^r(\mb{G})$ corresponds to a class in $\mc{EP}^r(\mb{G})$ which projects to some element $\mc{H}^{\mf{e}}$ of $X^{\mf{e}}$. Any pre-image of $(\psi, x)$ under $Y$ will lie in the part of $\coprod\limits_{ X^{\mf{e}}} [[\psi^{\mb{H}}]]$ indexed by $\mc{H}^{\mf{e}}$. This proves the first assertion.

To determine the cardinality of the pre-image under $Y$, we look at the set of $\alpha \in \Aut_r(\mb{H}, s, \eta)$. By Lemma \ref{alphalem}, we can choose $^L\alpha: \, ^L\mb{H} \to \, ^L\mb{H}$ and $g \in \widehat{\mb{G}}$ such that the diagram
\begin{equation*}
\begin{tikzcd}
^L\mb{H} \arrow[r, "^L\eta"] \arrow[d, swap,  "^L\alpha"] & ^L\mb{G} \arrow[d, "\Int(g)"] \\
^L\mb{H} \arrow[r, "^L\eta"]& ^L\mb{G},
\end{tikzcd}
\end{equation*}
commutes. Our choices of $g$ and ${}^L\alpha$ are unique up to an element of $\eta(\widehat{\mb{H}})$ and hence define a map 
\begin{equation*}
    \Aut_r(\mb{H},s,\eta) \to \widehat{\mb{G}}/ \eta(\widehat{\mb{H}}),
\end{equation*}
given by
\begin{equation*}
    \alpha \mapsto g.
\end{equation*}
Clearly, this induces an embedding
\begin{equation}{\label{outmap}}
    \Out_r(\mb{H},s,\eta) \hookrightarrow \widehat{\mb{G}}/\eta(\widehat{\mb{H}}).
\end{equation}

We now investigate which elements $\alpha \in \Aut_r(\mb{H},s,\eta)$ give a $Z(\widehat{\mb{G}})$- equivalence $\psi^{\mb{H}} \sim \, ^L \alpha \circ \psi^{\mb{H}}$. Since this is the case for all $\alpha \in \Inn_r(\mb{H},s,\eta)$, we need only investigate which elements of $\Out_r(\mb{H},s,\eta)$ have a representative with this property. 

Now suppose that $\alpha$ is such that for some choice of $^L\alpha$, we have that $\psi^{\mb{H}}$ and $^L\alpha \circ \psi^{\mb{H}}$ are $Z(\widehat{\mb{G}})$-equivalent. Then by the proof of Lemma \ref{equivconjlem} and since $\mb{G}$ satisfies the Hasse principle, these parameters are in fact equivalent. Moreover, since $\Lalpha$ is defined up to $\widehat{\mb{H}}$ conjugacy, we can choose it such that we have the following commutative diagram.
\begin{equation}
\begin{tikzcd}
& ^L\mb{H} \arrow[r, "\Leta"]  \arrow[dd,swap, "^L\alpha"]& ^L \mb{G} \arrow[dd, "\Int(g)"] \\
\mc{L}_F  \times SL_2(\C) \arrow[ur, "\psi^{\mb{H}}"] \arrow[dr,swap, "\psi^{\mb{H}}"] &&\\
& ^L\mb{H} \arrow[r, "\Leta"] & ^L\mb{G}.
\end{tikzcd}
\end{equation}
In particular, we see that $g \in Z_{C_{\psi}}(x)$. 

The set of elements in $\widehat{\mb{G}}/\eta(\widehat{\mb{H}})$ with some representative in $Z_{C_{\psi}}(x)$ is given by
\begin{equation*}
    Z_{C_{\psi}}(x)/[Z_{C_{\psi}}(x) \cap \eta(\widehat{\mb{H}})].
\end{equation*}
Since we have
\begin{equation*}
    Z_{C_{\psi}}(x) \cap \eta(\widehat{\mb{H}})=\eta(C_{\psi^{\mb{H}}}),
\end{equation*}
the above set has cardinality
\begin{equation*}
    | Z_{C_{\psi}}(x) / \eta(C_{\psi^{\mb{H}}})|.
\end{equation*}
In particular, the lemma will follow from the fact that each element of $Z_{C_{\psi}}(x) / \eta(C_{\psi^{\mb{H}}})$ lies in the image of some element $\Out_r(\mb{H}, s, \eta)$ under Equation \eqref{outmap}. Pick $g \in Z_{C_{\psi}}(x)$. Then clearly $g$ induces an automorphism 
\begin{equation*}
    \widehat{\alpha}:= \eta^{-1} \circ \Int(g) \circ \eta: \widehat{\mb{H}} \to \widehat{\mb{H}},
\end{equation*}
that fixes $s$ and such that $\eta \circ \widehat{\alpha} = \Int(g) \circ \eta$. To check this descends to an automorphism of $\mb{H}$, we need to show it is $\Gamma_{\Q}$-invariant with respect to the action of $\Gamma_{\Q}$ on $\widehat{\mb{H}}$ up to $\Inn(\widehat{\mb{H}})$. We can pick $K$ such that the action of $\Gamma_{\Q}$ factors through $\Gal(K/\Q)$. Then pick $w \in \Gal(K/\Q)$ and $w' \in \mc{L}_{\Q} \times SL_2(\C)$ such that $w'$ projects to $w$. Then pick $h \in \widehat{\mb{H}}$. We have up to $\widehat{\mb{H}}$-conjugacy
\begin{align*}
    (\eta^{-1} \circ \Int(g) \circ \eta)(w(h))  & \sim (\eta^{-1} \circ \Int(g) \circ \eta \circ \Int(\psi^{\mb{H}}(w')))(h)\\
    & =(\eta^{-1} \circ \Int(\eta( \psi^{\mb{H}}(w'))) \circ \Int(g) \circ \eta)(h)\\
    & \sim w(\eta^{-1} \circ \Int(g) \circ \eta)(h)),
\end{align*}
as desired. This completes the proof.
\end{proof}
Finally, we record a simple lemma that will be useful to us.
\begin{lemma}{\label{centralizerlem}}
Suppose that $(\mb{H}, s, \, ^L\eta,  \psi^{\mb{H}})$ is a representative of an element of $\mc{EP}^r(\mb{G})^{\el}$. Then we have an equality
\begin{equation}
    \frac{1}{| {_\mb{G}\mc{S}_{\psi^{\mb{H}}}}|} = |\frac{Z(\widehat{\mb{G}})^{\Gamma_{\Q}}}{\eta(C_{\psi^{\mb{H}}})}|.
\end{equation}
\end{lemma}
\begin{proof}
As in Equation \ref{SCeq}, we have an equality
\begin{equation*}
    {_\mb{G}S_{\psi^{\mb{H}}}}=C_{\psi^{\mb{H}}}Z(\widehat{\mb{G}}),
\end{equation*}
and hence
\begin{equation*}
    {_\mb{G}S_{\psi^{\mb{H}}}}/ Z(\widehat{\mb{G}}) = C_{\psi^{\mb{H}}}Z(\widehat{\mb{G}})/Z(\widehat{\mb{G}})=C_{\psi^{\mb{H}}}/Z(\widehat{\mb{G}})^{\Gamma_{\Q}}.
\end{equation*}

Finally, since $\psi := ^L\eta \circ \psi^{\mb{H}}$ is discrete, we have that $\eta( \, _\mb{G}S^{\circ}_{\psi^{\mb{H}}}) \subset S^{\circ}_{\psi} \subset Z(\widehat{\mb{G}})$ and hence 
\begin{equation*}
    _\mb{G}\mc{S}_{\psi^{\mb{H}}}= \, _\mb{G}S_{\psi^{\mb{H}}} / Z(\widehat{\mb{G}}).
\end{equation*}
The desired equality follows by combining the above equations.
\end{proof}

We now return to the cohomology of Shimura varieties. Combining everything together, we have
\begin{align*}
    &\mathrm{tr}(f \times \Phi^j_{\mf{p}})|H^*_{\xi})\\
    &=\sum\limits_{\mathcal{H}^{\mathfrak{e}} \in X^{\mf{e}}} \iota_r(\mathbf{G}, \mathbf{H}) ST^{\mathbf{H}}_{r, ell}(f^{\mathbf{H}})\\
    &=\sum\limits_{\mathcal{H}^{\mathfrak{e}} \in X^{\mf{e}}} \iota_r(\mathbf{G}, \mathbf{H}) ST^{\mathbf{H}}_{r, disc}(f^{\mathbf{H}})\\
    &=\sum\limits_{\mathcal{H}^{\mathfrak{e}} \in X^{\mf{e}}} \iota_r(\mathbf{G}, \mathbf{H})\sum\limits_{[\psi_{\mathbf{H}}] \in X^{sp}_{\mc{H}^{\mf{e}}}} \frac{1}{|\mathcal{S}_{\psi^{\mathbf{H}}}|}  \sum\limits_{\pi^{\mathbf{H}} \in \Pi_{\psi^{\mathbf{H}}}} \epsilon_{\psi^{\mathbf{H}}}(s_{\psi^{\mathbf{H}}})\langle \pi^{\mathbf{H}}, s_{\psi^{\mathbf{H}}} \rangle \mathrm{tr}(\pi \mid f^{\mathbf{H}}),
\end{align*}
where  we recall that $[\psi]$ denotes an equivalence class of Arthur parameters.

We continue: 
\begin{align*}
    & \mathrm{tr}(f \times \Phi^j_{\mf{p}}|H^*_{\xi})\\
    &=\sum\limits_{\mathcal{H}^{\mathfrak{e}} \in X^{\mf{e}}} \iota_r(\mathbf{G}, \mathbf{H})\sum\limits_{[\psi_{\mathbf{H}}] \in X^{sp}_{\mc{H}^{\mf{e}}}} \frac{1}{|\mathcal{S}_{\psi^{\mathbf{H}}}|}  \sum\limits_{\pi^{\mathbf{H}} \in \Pi_{\psi^{\mathbf{H}}}} \epsilon_{\psi^{\mathbf{H}}}(s_{\psi^{\mathbf{H}}})\langle \pi^{\mathbf{H}}, s_{\psi^{\mathbf{H}}} \rangle \mathrm{tr}(\pi \mid f^{\mathbf{H}})\\
    &=\sum\limits_{\mathcal{H}^{\mathfrak{e}} \in X^{\mf{e}}} \frac{\tau(\mathbf{G})}{\tau(\mathbf{H})|\mathrm{Out}_r(\mathcal{H}^{\mathfrak{e}})|}\sum\limits_{[\psi_{\mathbf{H}}] \in X^{sp}_{\mc{H}^{\mf{e}}}} \frac{1}{|\mathcal{S}_{\psi^{\mathbf{H}}}|}  \sum\limits_{\pi^{\mathbf{H}} \in \Pi_{\psi^{\mathbf{H}}}} \epsilon_{\psi^{\mathbf{H}}}(s_{\psi^{\mathbf{H}}})\langle \pi^{\mathbf{H}}, s_{\psi^{\mathbf{H}}} \rangle \mathrm{tr}(\pi \mid f^{\mathbf{H}})\\
    &=\sum\limits_{\mathcal{H}^{\mathfrak{e}} \in X^{\mf{e}}} \,\ \sum\limits_{[[\psi_{\mathbf{H}}]] : [\psi_{\mathbf{H}}] \in X^{sp}_{\mc{H}^{\mf{e}}}} \frac{1}{| _{\mathbf{G}}\mathcal{S}_{\psi^{\mathbf{H}}}| \cdot |\mathrm{Out}_r(\mathcal{H}^{\mathfrak{e}})|}  \sum\limits_{\pi^{\mathbf{H}} \in \Pi_{\psi^{\mathbf{H}}}} \epsilon_{\psi^{\mathbf{H}}}(s_{\psi^{\mathbf{H}}})\langle \pi^{\mathbf{H}}, s_{\psi^{\mathbf{H}}} \rangle \mathrm{tr}(\pi \mid f^{\mathbf{H}})\\
    &=\sum\limits_{\mathcal{H}^{\mathfrak{e}} \in X^{\mf{e}}} \,\ \sum\limits_{[[\psi_{\mathbf{H}}]] : [\psi_{\mathbf{H}}] \in X^{sp}_{\mc{H}^{\mf{e}}}} |\frac{Z(\widehat{\mb{G}})^{\Gamma_{\Q}}}{\eta(C_{\psi^{\mb{H}}})}| \cdot \frac{1}{ |\mathrm{Out}_r(\mathcal{H}^{\mathfrak{e}})|}  \sum\limits_{\pi^{\mathbf{H}} \in \Pi_{\psi^{\mathbf{H}}}} \epsilon_{\psi^{\mathbf{H}}}(s_{\psi^{\mathbf{H}}})\langle \pi^{\mathbf{H}}, s_{\psi^{\mathbf{H}}} \rangle \mathrm{tr}(\pi \mid f^{\mathbf{H}})\\
    &=\sum\limits_{(\psi, x) \in X_{\Psi}} \,\ \sum\limits_{(\mathcal{H}^{\mathfrak{e}}, [[\psi_{\mb{H}}]]) \in Y^{-1}(\psi, x)}|\frac{Z(\widehat{\mb{G}})^{\Gamma_{\Q}}}{\eta(C_{\psi^{\mb{H}}})}| \cdot \frac{1}{ |\mathrm{Out}_r(\mathcal{H}^{\mathfrak{e}})|}   \sum\limits_{\pi^{\mathbf{H}} \in \Pi_{\psi^{\mathbf{H}}}} \epsilon_{\psi^{\mathbf{H}}}(s_{\psi^{\mathbf{H}}})\langle \pi^{\mathbf{H}}, s_{\psi^{\mathbf{H}}} \rangle \mathrm{tr}(\pi \mid f^{\mathbf{H}})\\
    &=\sum\limits_{(\psi, x) \in X_{\Psi}} \,\ \sum\limits_{(\mathcal{H}^{\mathfrak{e}}, [[\psi_{\mb{H}}]]) \in Y^{-1}(\psi, x)} |\frac{Z(\widehat{\mb{G}})^{\Gamma_{\Q}}}{\eta(C_{\psi^{\mb{H}}})}| \cdot \frac{1}{ |\mathrm{Out}_r(\mathcal{H}^{\mathfrak{e}})|}  \sum\limits_{\pi^{\infty} \in \Pi_{\psi^{\infty}}(\mathbf{G}, \varrho^{\infty})} \epsilon_{\psi}(\overline{x}s_{\psi})\langle \pi, \overline{x}s_{\psi}\rangle \\
    & \cdot \mathrm{tr}(\pi^{\infty} \mid f)A(\psi, x, \Phi^j_{\mf{p}})(-1)^{q(\mathbf{G})}\langle\lambda_{\pi_{\infty}}, xs_{\psi} \rangle\\
    &=\sum\limits_{(\psi, x) \in X_{\Psi}}  \,     | \Out_r(\mc{H}^{\mf{e}})| \cdot |\frac{\eta(C_{\psi^\mb{H}})}{Z_{C_{\psi}(x)}}| \cdot|\frac{Z(\widehat{\mb{G}})^{\Gamma_{\Q}}}{\eta(C_{\psi^{\mb{H}}})}| \cdot \frac{1}{ |\mathrm{Out}_r(\mathcal{H}^{\mathfrak{e}})|}\\
    & \cdot \sum\limits_{\pi^{\infty} \in \Pi_{\psi^{\infty}}(\mathbf{G}, \varrho^{p, \infty})} \epsilon_{\psi}(\overline{x}s_{\psi})\langle \pi, \overline{x}s_{\psi}\rangle \mathrm{tr}(\pi^{\infty} \mid f)A(\psi, x, \Phi^j_{\mf{p}})(-1)^{q(\mathbf{G})}\langle\lambda_{\pi_{\infty}}, xs_{\psi} \rangle,
    \end{align*}
    where the fourth equality follows from Lemma \ref{centralizerlem} and the seventh follows from Lemma \ref{paramcountinglem}. We then get that the above equals
    \begin{align*}    
    &=\sum\limits_{(\psi, x) \in X_{\Psi}} \, |\frac{Z(\widehat{\mb{G}})^{\Gamma_{\Q}}}{Z_{C_{\psi}(x)}}| \sum\limits_{\pi^{\infty} \in \Pi_{\psi^{\infty}}(\mathbf{G}, z^{iso, \infty})} \epsilon_{\psi}(\overline{x}s_{\psi})\langle \pi, \overline{x}s_{\psi}\rangle \mathrm{tr}(\pi^{\infty} \mid f)A(\psi, x, \Phi^j_{\mf{p}})(-1)^{q(\mathbf{G})}\langle\lambda_{\pi_{\infty}}, xs_{\psi} \rangle\\
   &=\sum\limits_{([\psi], [x]) \in X_{\Psi}}\,  \sum\limits_{x \in [x]}  \frac{1}{|\mathcal{S}_{\psi}|}\sum\limits_{\pi^{\infty} \in \Pi_{\psi^{\infty}}(\mathbf{G}, z^{iso, \infty})} \epsilon_{\psi}(\overline{x}s_{\psi})\langle \pi, \overline{x}s_{\psi} \rangle \mathrm{tr}(\pi^{\infty} \mid f)A(\psi, x, \Phi^j_{\mf{p}})(-1)^{q(\mathbf{G})}\langle \lambda_{\pi_{\infty}}, xs_{\psi} \rangle.
    \end{align*}

\begin{remark}
Any term corresponding to a parameter $\psi$ in the above sum is zero unless it is a discrete Arthur parameter such that $\psi_{\infty}$ is cohomological for $\xi_{\mathbb{C}}$. 

Note that following Kottwitz, for such $\psi$, the groups $\mathcal{S}_{\psi}$ and $S_{\psi}$ and $C_{\psi}$ are all abelian since $S_{\psi}$ embeds into $S_{\psi_{\infty}}$ which is abelian and given by the formula on pg 195 of \cite{kot7}.
\end{remark}

In the above, we have chosen a $\pi_{\infty} \in \Pi_{\psi_{\infty}}(\mathbf{G}, \varrho_{\infty})$ such that $\pi \in \Pi_{\psi}(\mathbf{G}, \varrho))$. By our earlier remarks, the above expression does not depend on this choice.

It will also be convenient to check that $A(\psi, x, \Phi^j_{\mf{p}})\langle \lambda_{\pi_{\infty}}, xs_{\psi} \rangle$ depends only on $\overline{x} \in \mathcal{S}_{\psi}$. Since $\psi$ is assumed to be discrete, we need only check that $A(\psi, x, \Phi^j_{\mf{p}})\langle \lambda_{\pi_{\infty}}, xs_{\psi} \rangle$ is trivial restricted to $Z(\widehat{\mb{G}})$. The character $\lambda_{\pi_{\infty}}$ restricts to $\mu$ on $Z(\widehat{\mb{G}})$ and $r_{-\mu}$ acts by $- \mu$ so these actions cancel as desired.

In \cite{kot7}, Kottwitz works with the group $S_{\psi}$ of self-equivalences of $\psi$. Since $\mathbf{G}$ satisfies the Hasse principle, we can work instead with the centralizer group $C_{\psi}$ of $\psi$ instead. Recall that $\mathcal{S}_{\psi}=C_{\psi}/Z(\widehat{\mathbf{G}})^{\Gamma}$. 

Consider the set $\nu \in X^*(C_{\psi})$ such that $\nu|_{Z(\widehat{\mathbf{G}})^{\Gamma}}= \mu|_{Z(\widehat{\mathbf{G}})^{\Gamma}}$ and let $V_{\nu}$ be the largest subspace of $V$ (the target space of the $r_{- \mu} \circ \psi$-action) where $C_{\psi}$ acts by $-\nu$.  Then (following Kottwitz), let $\phi_E$ be the restriction of $\phi_{\psi}$ to $W_E$ and let $V(\psi, \nu)$ be the action of $\psi_E$ on $V_{\nu}$ twisted by $|\cdot |^{\frac{-\mathrm{dim}Sh}{2}}$.

Now for a given $\pi \in \Pi_{\psi}(\mathbf{G}, \varrho)$, we expect that $x \mapsto \epsilon_{\psi}(x)\langle \pi, x \rangle$ is (assuming Arthur's conjectures) the character of an irreducible representation of $\mathcal{S}_{\psi}$. The character $\nu -\lambda_{\pi_{\infty}}$ of $C_{\psi}$ is trivial on $Z(\widehat{\mathbf{G}})^{\Gamma}$ and therefore descends to a character of $\mathcal{S}_{\psi}$. We define $m(\pi, \nu)$ to be the multiplicity of this character in the above irreducible representation. Explicitly, we have

\begin{equation}
    m(\pi, \nu)= \frac{1}{|\mathcal{S}_{\psi}|}\sum\limits_{\overline{x} \in \mathcal{S}_{\psi}} \epsilon_{\psi}(\overline{x})\langle \pi, \overline{x} \rangle \nu(x)^{-1}\langle \lambda_{\pi_{\infty}}, x \rangle.  
\end{equation}

By \cite[Lemma 9.2]{kot7}, this multiplicity does not depend on $\pi_{\infty}$ so we use the notation $m(\pi^{\infty}, \nu)$.

Now, set $A(\psi_p, \nu)= (-1)^{q(\mathbf{G})}\nu(s_{\psi}) \mathrm{tr}\phi_{E_p}(\Phi^j_{\mf{p}})( V_{\nu} | \cdot |_p^{\frac{- \mathrm{dim}Sh}{2}})$ where  $q(\mathbf{G})$ is half the dimension of the symmetric space associated to $\mathbf{G}$.

Then we have
\begin{align*}
    &\sum\limits_{\nu} m(\pi^{\infty}, \nu)A(\psi_p, \nu)\\
    &= \frac{1}{|\mathcal{S}_{\psi}|} \sum\limits_{\overline{x} \in \mathcal{S}_{\psi}} \sum\limits_{\nu} \epsilon(\overline{x})\langle \pi, \overline{x} \rangle \nu(x)^{-1} \langle \lambda_{\pi_{\infty}}, x \rangle (-1)^{q(\mathbf{G})} \nu(s_{\psi})\mathrm{tr}\phi_{E,p}(\Phi^j_{\mf{p}})(V_{\nu} |\cdot|^{-\frac{\mathrm{dim} Sh}{2}})\\
    &=\frac{1}{|\mathcal{S}_{\psi}|}\sum\limits_{\overline{x} \in \mathcal{S}_{\psi}} \epsilon_{\psi}(s_{\psi}\overline{x})\langle \pi, s_{\psi}\overline{x} \rangle \langle \lambda_{\pi_{\infty}}, s_{\psi}x \rangle (-1)^{q(\mathbf{G})}\sum\limits_{\nu} \nu(x)^{-1} \mathrm{tr}\phi_{E,p}(\Phi^j_{\mf{p}})(V_{\nu} | \cdot |^{\frac{-\mathrm{dim} Sh}{2}})\\
    &=\frac{1}{|\mathcal{S}_{\psi}|}\sum\limits_{\overline{x} \in \mathcal{S}_{\psi}} \epsilon_{\psi}(s_{\psi}\overline{x})\langle \pi, s_{\psi}\overline{x} \rangle \langle \lambda_{\pi_{\infty}}, s_{\psi}x \rangle (-1)^{q(\mathbf{G})}\sum\limits_{\nu}  \mathrm{tr}\phi_{E,p}(x\Phi^j_{\mf{p}})(V_{\nu} | \cdot |^{\frac{-\mathrm{dim} Sh}{2}})\\
    &=\frac{1}{|\mathcal{S}_{\psi}|}\sum\limits_{\overline{x} \in \mathcal{S}_{\psi}} \epsilon_{\psi}(s_{\psi}\overline{x})\langle \pi, s_{\psi}\overline{x} \rangle \langle \lambda_{\pi_{\infty}}, s_{\psi}x \rangle (-1)^{q(\mathbf{G})}\mathrm{tr}\phi_{E,p}(x\Phi^j_{\mf{p}})(V | \cdot |^{\frac{-\mathrm{dim} Sh}{2}})\\
    &=\frac{1}{|\mathcal{S}_{\psi}|}\sum\limits_{\overline{x} \in \mathcal{S}_{\psi}} \epsilon_{\psi}(s_{\psi}\overline{x})\langle \pi, s_{\psi}\overline{x} \rangle \langle \lambda_{\pi_{\infty}}, s_{\psi}x \rangle (-1)^{q(\mathbf{G})}A(\psi, x, \Phi^j_{\mf{p}})
    \end{align*}
    \begin{align*}
    &=\frac{1}{|\mathcal{S}_{\psi}|}\sum\limits_{x : ([x], [\psi]) \in X_{\Psi}}  \epsilon_{\psi}(s_{\psi}\overline{x})\langle \pi, s_{\psi}\overline{x} \rangle \langle \lambda_{\pi_{\infty}}, s_{\psi}x \rangle (-1)^{q(\mathbf{G})}A(\psi, x, \Phi^j_{\mf{p}}).
\end{align*}
In particular, we have
\begin{equation}
    \mathrm{tr}(f \times \Phi^j_{\mf{p}}|H^*_{\xi})=\sum\limits_{[\psi]} \, \sum\limits_{\pi^{\infty} \in \Pi_{\psi^{\infty}}(\mathbf{G}, \varrho^{\infty})} \sum\limits_{\nu} m(\pi^{\infty}, \nu)A(\psi_p,\nu) \mathrm{tr}(\pi^{\infty} \mid f).
\end{equation}

We now consider what this formula implies of $H^*(Sh, \mc{L_{\xi}}) := \lim\limits_K H^*_c(Sh_K, \mc{L}_{\xi})$ as an element in the Grothendieck group of admissible $\mb{G}(\A_f) \times \Gamma_{\mb{E}}$ representations, where $\mb{E}$ is the global reflex field of $\mu$. 

Take some irreducible $\mb{G}(\A_f)$-representation $\pi_f$. We show that the $\pi_f$-isotypic parts of $H^*(Sh, \mc{L}_{\xi})$ and 

\begin{equation}
\sum\limits_{[\psi]} \, \sum\limits_{\nu}\sum\limits_{\pi^{\infty} \in \Pi_{\psi^{\infty}}(\mathbf{G}, \varrho^{\infty})} m(\pi^{\infty}, \nu)\nu(s_{\psi})(-1)^{q(\mathbf{G})}\pi^{\infty} \boxtimes V(\psi, \nu)_{\lambda},
\end{equation}
are equal. We note that by Matsushima's formula, $H^*(Sh, \mc{L}_{\xi})$ is a semisimple $\mb{G}(\A_f)$ representation and can be written in the form $\bigoplus\limits_{\pi'} \pi' \boxtimes \sigma(\pi')$ where $\sigma(\pi')$ is a $\Gamma_{\mb{E}}$-representation.

Fix  compact open $K \subset \mb{G}(\A_f)$ such that $\pi^K_f \neq \emptyset$. Choose $f \in \mc{H}(K)$ so that for each irreducible admissible $\mb{G}(\A_f)$ representation $\pi'$ appearing in either of the above two expressions, we have $\tr(\pi' \mid f) = 0$ unless $\pi'= \pi$ in which case $\tr(\pi' \mid f)=1$. See \cite[ \S II.3.2]{BMY1} for details.

Now, for all but finitely many $p$, we have factorizations $K=K^pK_p$ and $f=f^pf_p$ where $K_p$ is hyperspecial and $f_p$ is the indicator function on $K_p$. We can now apply the result of the above extensive trace computations to conclude that for all but finitely many $p$ and  each $j$ and each place $\mf{p}$ of $\mb{E}$ over $p$, we have that the trace of the $\pi_f$-isotypic parts of the above two expressions at $\Phi^j_{\mf{p}}$ as $\Gamma_{\mb{E}}$-representations are equal. Hence by the Cebotarev density theorem, the semi-simplifications of these $\pi$-isotypic pieces are isomorphic.

We have therefore shown that 
\begin{equation}
H^*(Sh, \mc{L}_{\xi}) = \sum\limits_{[\psi]} \, \sum\limits_{\nu}\sum\limits_{\pi^{\infty} \in \Pi_{\psi^{\infty}}(\mathbf{G}, \varrho^{\infty})} m(\pi^{\infty}, \nu)\nu(s_{\psi})(-1)^{q(\mathbf{G})}\pi^{\infty} \boxtimes V(\psi, \nu)_{\lambda},
\end{equation}
in the Grothendieck group of $\mb{G}(\A_f) \times \Gamma_{\mb{E}}$-modules.

\section{Cohomology of Igusa varieties}
We now repeat the analysis of the previous section for Igusa varieties. For the stabilization parts we draw from \cite{Shi3}. The destabilization is carried out in Shin's thesis under some simplifying assumptions (for instance he only works with PEL data of A-type and assumes his groups are products of inner forms of $GL_n$ at $p$). We largely emulate his method while also proceeding in analogy with the previous section.

\subsection{Igusa varieties}

We will use the description of Igusa varieties and their cohomology as in \cite{Shi1}.  
\begin{assumption}
We fix a quasisplit connected reductive group $\mb{G}^*$ over $\Q$ and an inner form $\mb{G}$. We make the following assumptions on $\mb{G}$:
\begin{itemize}
    \item We assume exists an extended pure inner twist $(\varrho, z^{iso})$ where $z^{iso} \in Z^1_{\bas}(\mc{E}^{\iso}, \mb{G}^*)$ and $\varrho: \mb{G}^* \to \mb{G}$ and fix one such.
    \item The group $\mb{G}_{\der}$ is simply connected.
    \item The group $\mb{G}$ satisfies the Hasse principle.
    \item We assume $\mb{G}_{\Q_p}$ is unramified.
    \item We assume there exists a PEL type Shimura datum $(\mb{G}, X)$ and fix one such.
    \item The maximal $\Q$-split torus in $Z(\mb{G})$  coincides with the maximal $\R$-split torus in $Z(\mb{G})$.
\end{itemize}
\end{assumption}

We let $G= \mathbf{G}_{\mathbb{Q}_p}$ and let $G^* = \mb{G}^*_{\Q_p}$. Note that even though $G$ and $G^*$ are isomorphic as groups, $G$ is an extended pure inner twist of $G^*$ via $(\varrho_p, z^{\iso}_p)$ and this extended pure inner twist may well be nontrivial. For instance, in the case that $G$ is an odd unitary group, there are two equivalence classes of extended pure inner twists, both of which are quasisplit.

We fix $S \subset T \subset B \subset G$ such that $S$ is a maximal $\mathbb{Q}_p$-split torus, $T$ is a maximal torus defined over $\mathbb{Q}_p$, and $B$ is a Borel subgroup defined over $\mathbb{Q}_p$. 

We fix $b \in \mathbf{B}(\Q_p, G)$ and $\tilde{b} \in G(L)$ a decent representative of $b$ such that $\nu_{\tilde{b}} \in X_*(S)^+_{\mathbb{Q}}$ as in Lemma \ref{decentreps}. Then $M_{\tilde{b}}$ is a standard Levi subgroup and we can fix an inner twist $\psi_b: M_{\tilde{b}} \to J_{\tilde{b}}$ so that $\psi^{-1}_b \circ \psi^{\sigma}_b=\mathrm{Int}(\tilde{b})$. We denote $M_{\tilde{b}}, J_{\tilde{b}}$ by $M_b, J_b$ respectively.

Now we fix a $p$-divisible group $\Sigma_b$ of isogeny type $b$ over $\overline{F}_p$ satisfying the extra conditions (i), ..., (iv) in \cite[pg15]{Shi4}. Let $K^pK_p \subset \mb{G}(\A^{\infty})$ such that $K^p$ is a compact open subgroup of $\mathbf{G}(\mathbb{A}^{p, \infty})$ and $K_p \subset \mathbf{G}(\mathbb{Q}_p)$ is a hyperspecial subgroup. We get a projective system $\mathrm{Ig}_{\Sigma_b}$ indexed by $K^p$ and integers $m$. Then let $\Xi$ be a fixed finite dimensional representation of $\mathbf{G}$ over $\overline{\mathbb{Q}}_l$ which determines a local system $\mathcal{L}_{\xi}$ on $\mathrm{Ig}_{\Sigma_b}$. Define $H^*_c(Ig_{\Sigma_b}, \mathcal{L}_{\xi})$ as in \cite[ph13]{Shi1}. This is a virtual representation of $\mathbf{G}(\mathbb{A}^{\infty, p}) \times J_b(\mathbb{Q}_p)$ in $\mathrm{Groth}(\mathbf{G}(\mathbb{A}^{p, \infty}) \times J_b(\mathbb{Q}_p))$.

The stable trace formula for Igusa varieties uses Kottwitz triples of type $b$ as in \cite[Def 4.2]{Shi1}. In particular, a Kottwitz triple of type $b$ is a triple $(\gamma_0 ; \gamma, \delta)$ so that $\gamma_0 \in \mathbf{G}(\mathbb{Q})$ is semisimple and elliptic in $\mathbf{G}(\mathbb{R})$, $\gamma \in \mathbf{G}(\mathbb{A}^{p, \infty})$ and $\gamma \sim \gamma_0$ in $\mathbf{G}(\overline{\mathbb{A}}^{p, \infty}), \delta \in J_b(\mathbb{Q}_p)$ is $\nu_b$-acceptable and $\psi^{-1}_b(\delta) \sim \gamma_0 \in \mathbf{G}(\overline{\mathbb{Q}}_p)$. Triples $(\gamma_0 ; \gamma, \delta)$ and $(\gamma'_0; \gamma', \delta')$ are equivalent if $\gamma_0$ and $\gamma'_0$ are stably conjugate in $\mb{G}(\Q)$, and if $\gamma$ and $\gamma'$ are conjugate in $\mb{G}(\A^{p, \infty})$, and if $\delta$ and $\delta'$ are conjugate in $J_b(\Q_p)$. The set of equivalence classes of Kottwitz triples of type $b$ is denoted $KT_b$.

Similarly to the case of Shimura varieties, Shin defines the invariant $\alpha(\gamma_0; \gamma, \delta) \in \mathfrak{K}(I^{\mb{G}}_{\gamma_0}/ \mathbb{Q})^D$ and considers the set $KT^{eff}_b$ of equivalence classes of Kottwitz triples such that $\alpha(\gamma_0; \gamma, \delta)=1$.

Now, Shin's formula for the cohomology of Igusa varieties is stated for acceptable $\phi \in C_c(\mathbf{G}(\mathbb{A}^{p, \infty} \times J_b(\mathbb{Q}_p))$. A function $\phi$ is acceptable if it can be written $\phi=\phi^p \times \phi_p$ such that the support of $\phi_p$ is contained within the $\nu_b$-acceptable locus of $J_b(\mathbb{Q}_p)$ plus a few other conditions (see \cite[Def 6.2]{Shi4}).

\begin{theorem}[Shin]
For acceptable $\phi \in C_c(\mathbf{G}(\mathbb{A}^{p, \infty} \times J_b(\mathbb{Q}_p))$, we have
\begin{equation*}
\mathrm{tr}(\phi| H^*_c(Ig_{\Sigma_b}, \mathcal{L}_{\xi}))=\sum_{\substack{(\gamma_0; \gamma ,\delta) \in KT_b  \\ \alpha(\gamma_0 ; \gamma, \delta)=1}} \frac{1}{|\mathrm{Vol}(I_{\infty}(\mathbb{R})^1)|}|A_{\mathbb{Q}}(I^{\mb{G}}_{\gamma_0})|\mathrm{tr}\xi(\gamma_0)O_{\gamma}(\phi^p)O^{J_b(\mathbb{Q}_p)}_{\delta}(\phi_p).
\end{equation*}
We recall that $A_{\mathbb{Q}}(I_0)$ is defined to be $\pi_0(Z(\widehat{I^{\mb{G}}_{\gamma_0}})^{\Gamma})^D$ and $I_{\infty}$ is a compact modulo center inner form of $I_0$ over $\mathbb{R}$.
\end{theorem}
We now choose a set of representatives of lifts of elements of $\mc{E}(\mb{G})$ to $\mc{E}^r(\mb{G})$ as in Construction \ref{endoreps}.  We normalize transfer factors $\Delta[\mathfrak{w}_v, z^{\iso}_v]$ as in \S \ref{transfactsect}. In particular, this involves fixing a $\Q$-splitting $(\mb{B}, \mb{T}, \{X_{\alpha}\})$ of $\mb{G}^*$ and a character $\chi: \Q \to \C^*$.

Recall that we have also fixed a section $\mc{S}$ of the map $\mc{EQ}^r(\mb{G}) \to \mc{EQ}(\mb{G})$ and that this gives a set of lifts, $X_{\gamma_0}$, of each $\kappa \in \mathfrak{K}(I^{\mb{G}}_{\gamma_0}/ \mathbb{Q})$ to a $\lambda \in Z(\widehat{I^{\mb{G}}_{\gamma_0}})^{\Gamma_{\Q}}$.

Following what we did for Shimura varieties in Equation \eqref{fouriereqn}, we get the formula
\begin{align*}
    &\mathrm{tr}( \phi \mid H^*_c(Ig_{\Sigma_b}, \mathcal{L}_{\xi}))\\
    &=\tau(\mb{G})\sum\limits_{(\gamma_0; \gamma ,\delta) \in KT_b } \frac{1}{|\mathrm{Vol}(I_{\infty}(\mathbb{R})^1)|} \sum\limits_{\kappa \in\mathfrak{K}(I^{\mb{G}}_{\gamma_0}/ \mathbb{Q})} \langle \alpha(\gamma_0 ; \gamma , \delta), \kappa \rangle^{-1} \mathrm{tr}\xi(\gamma_0)O_{\gamma}(\phi^p)O^{J_b(\mathbb{Q}_p)}_{\delta}(\phi_p),
\end{align*}
where we note that we omit the term $|\ker^1(\Q, \mb{G})|$ that appears in \cite[(5.2)]{Shi1} because $\mb{G}$ is assumed to satisfy the Hasse principle.

We now split the above expression into three parts: a $p$-part, an $\infty$-part, and a part covering all the places of $\Q$ not equal to $p$ or $\infty$. Following \cite{Shi1}, we construct extensions $\tilde{\alpha}_v(\gamma_0; \gamma ,\delta) \in X^*(Z(\widehat{I^{\mb{G}}_{\gamma_0}})^{\Gamma_{\Q_v}}Z(\widehat{\mb{G}}))$ of $\alpha_v(\gamma_0; \gamma, \delta)$ such that $\tilde{\alpha}_v(\gamma_0; \gamma ,\delta)|_{Z(\widehat{\mb{G}})}$ equals $1, - \mu_1, \mu_1$ respectively for $v \neq p, \infty, v=p, v=\infty$. We use the notations $\tilde{\alpha}_p(\gamma_0, \delta), \tilde{\alpha}^p(\gamma_0, \gamma)$, and $\tilde{\alpha}_{\infty}(\gamma_0)$ to reflect the various dependencies.

We note that the pair $(\gamma_0, \kappa)$ gives a class in $\mc{SS}(\mb{G})$ and hence a quadruple $(\mb{H}, s, \eta, \gamma_{\mb{H}}) \in \mc{EQ}(\mb{G})$. Similarly, the lift $(\gamma_0, \lambda) \in \mc{SS}^r(\mb{G})$ gives a quadruple $(\mb{H}, s, \eta, \gamma_{\mb{H}}) \in \mc{EQ}^r(\mb{G})$.

Finally, we have the following equality (\cite[Lemma 5.1]{Shi1})for $(\gamma_0; \gamma, \delta)$ such that $\alpha(\gamma_0; \gamma, \delta)=1$: 
\begin{equation*}
    e^p(I^{\mb{G}}_{\gamma})e_p(I^{J_b}_{\delta})e_{\infty}(I_{\infty}) =1.
\end{equation*}

Using the above constructions, we get
\begin{equation*}
    \mathrm{tr}(\phi| H^*_c(Ig_{\Sigma_b}, \mathcal{L}_{\xi}))=\tau(\mb{G}) \sum\limits_{\gamma_0: \exists (\gamma_0, \gamma, \delta) \in KT_b} \sum\limits_{\lambda \in X_{\gamma_0}} O^p(\gamma_0, \lambda, \phi^p)O_p(\gamma_0, \lambda, \phi_p)O_{\infty}(\gamma_0, \lambda),
\end{equation*}
where we have
\begin{align*}
    O^p(\gamma_0, \lambda, \phi^p) &:= \Delta[\mf{w}^{p, \infty}, z^{\iso, p, \infty}](\gamma_{\mb{H}}, \gamma_0) \prod\limits_{v \neq p, \infty} \sum\limits_{\gamma_v \sim_{st} \gamma_0} e^p(I^{\mb{G}}_{\gamma})\langle \tilde{\alpha}^p(\gamma_0; \gamma, \delta), \lambda \rangle^{-1} O^{\mb{G}(\A^{p, \infty})}_{\gamma_v}(\phi^{p, \infty}),\\
    O_p(\gamma_0, \lambda, \phi_p) &:= \Delta[\mf{w}_p, z^{\iso}_p](\gamma_{\mb{H}}, \gamma_0)\sum\limits_{\delta \sim_{st} \gamma_0}\langle\tilde{\alpha}_p(\gamma_0; \gamma, \delta), \lambda \rangle^{-1} e(I^{J_b}_{\delta})O^{J_b(\Q_p)}_{\delta}(\phi_p),\\
    O_{\infty}(\gamma_0, \lambda) &:= \Delta[\mf{w}_{\infty}, z^{\iso}_{\infty}](\gamma_{\mb{H}}, \gamma_0) \sum\limits_{\delta}\mathrm{vol}(I_{\infty}(\R)^1)^{-1}) \langle \tilde{\alpha}_{\infty}(\gamma_0 ; \gamma, \delta), \lambda \rangle^{-1} e_{\infty}(I_{\infty})\tr\xi(\gamma_0).
\end{align*}

Now we define $f^{\mathbf{H},p} \in C_c(\mathbf{G}(\mathbb{A}^{p, \infty}))$ exactly as in \cite[lemma 5.2]{Shi1} which is the same as what we did for Shimura varieties and the treatment in \cite{kot7}.

We define the function $f^{\mathbf{H}}_{\infty}$ again following Shin (and therefore Kottwitz). However, as in the case of Shimura varieties, since we are using a Whittaker normalization of transfer factors, we will have some constant $b_{\mathbf{H}}$ such that
\begin{equation*}
    \Delta[\mathfrak{w}_{\infty}, z^{iso}_{\infty}]=b_{\mathbf{H}} \Delta_{j, B}.
\end{equation*}
In particular, we define our function $f^{' \mathbf{H}}_{\infty}$ so that
\begin{equation*}
    f^{' \mathbf{H}}_{\infty}=b_{\mathbf{H}}f^{\mathbf{H}}_{\infty},
\end{equation*}
so that we have
\begin{equation*}
    SO_{\gamma^{\mb{H}}}({f'}^{\mb{H}}_{\infty})= O_{\infty}(\gamma_0, \lambda).
\end{equation*}

\subsection{Stabilization at \texorpdfstring{$p$}{p}}
We now discuss the function $f^{\mathbf{H}}_p$ as its definition is substantially different to the Shimura variety case. We follow \cite[\S 6]{Shi1}. Note that in this section we use embedded endoscopic data as in Definition \ref{embdat}.

Let $H=\mathbf{H}_{\mathbb{Q}_p}$ and $G=\mathbf{G}_{\mathbb{Q}_p}$. We now consider the set $\mathcal{E}^i_{eff}(J_b, G; H)$ as defined in the paragraph containing Equation \eqref{jbghdef}.

We can pick a pair $(B_H, T_H)$ of a Borel subgroup and maximal torus of $H$ and a maximal torus of $G$ such that $T \subset M_b$ and $(H,s, \eta)$ transfers $T_H$ to $T$. We fix a set $X^{\mf{e}}_{J_b}$ of representatives of $\mathcal{E}^i_{eff}(J_b, G; H)$ such that each $H_{M_b}$ is a standard Levi subgroup of $H$ relative to $(B_H, T_H)$. We denote by $M^*_b$ the standard Levi subgroup of $G^*$ corresponding to $M_b$. Associated to each element of $X^{\mf{e}}_{J_b}$ we have a triple $(H, H_{M_b}, \nu)$ where $\nu: \bb{D} \to H$ is given by
\begin{equation*}
\begin{tikzcd}
\bb{D} \arrow[r, "\nu_b"]  & A_{M_b} \arrow[r, hook] & T \arrow[r, "\sim"] & T_H \arrow[r, hook] & H.
\end{tikzcd}
\end{equation*}

Now fix a $(H, H_{M_b}, s, \eta) \in X^{\mf{e}}_{J_b}$. We discuss transfer factors for this endoscopic datum.

Recall that we have fixed a splitting $(\mb{B}, \mb{T}, \{X_{\alpha}\})$ of $\mb{G}^*$ and character $\chi: \Q \to \C^*$. Localizing at $p$ gives a splitting and Whittaker datum $\mf{w}_p$ of $G^*$.  We can restrict the splitting to $M^*_b$ and hence get a Whittaker datum $\mf{w}^{M^*_b}$ of $M^*_b$. Let $\Delta[\mf{w}^{M^*_b}]$ be the corresponding Whittaker-normalized transfer factor between $H_M$ and $M^*_b$. Then we have
\begin{proposition}
For all $\gamma_{H_{M_b}} \in H_{M_b}(\Q_p)$ that is $(G^*,H)$-regular and $\gamma^*_0 \in M^*_b(\Q_p)$, we have the following equality of transfer factors,
\begin{equation}
    \Delta[\mf{w}^{M^*_b}](\gamma_{H_{M_b}}, \gamma^*_0)=|D^{G^*}_{M^*_b}(\gamma^*_0)|^{- \frac{1}{2}}_p|D^H_{H_{M_b}}(\gamma_{H_{M_b}})|^{\frac{1}{2}}_p\Delta[\mathfrak{w}_p](\gamma^*_0, \gamma_{H_{M_b}}),
\end{equation}
where we recall that  $D^G_{M_b}(\gamma)$ is defined to equal $\mathrm{det}(1- \mathrm{ad}(m))|_{\mathrm{Lie}(G) \setminus \mathrm{Lie}(M_b)}$.
\end{proposition}
\begin{proof}
See also \cite[Lemma 6.5]{Wal2} and \cite[Lemma 5.2]{Hir1}. We recall the Whittaker normalization of transfer factors (for instance see \cite[pg 6]{Kal1}) and compare the terms for $M^*_b$ and $G^*$. The term $\epsilon_L(V, \psi_F)$ is the same for $M^*_b$ and $G^*$. The $\Delta_{IV}$ terms differ by $|D^{G^*}_{M^*_b}(\gamma^*_0)|^{- \frac{1}{2}}_p|D^H_{H_{M_b}}(\gamma_{H_{M_b}})|^{\frac{1}{2}}_p$. The remaining terms depend on choices of $\chi$-data and $a$-data.  We can choose $a$-data of $G^*$ and $M^*_b$ so that the restriction of our $a$-datum for $G^*$ to the roots of $M^*_b$ is our $a$-datum for $M^*_b$. We can do the same for $\chi$-data. Moreover, we can choose our $\chi$-datum such that $\chi_{\alpha}$ is trivial for all roots $\alpha$ except for those such there exists a $\sigma \in \Gamma_{\Q_p}$ such that $\sigma(\alpha)=- \alpha$. Note that if $M$ is a rationally defined Levi subgroup with associated parabolic subgroup $P$, then the unipotent radical $U$ of $P$ is stable under the action of $\Gamma_{\Q_p}$. In particular, there is no root outside of $M$ satisfying the above property (cf \cite[pg27]{Art3}). Hence we can pick a $\chi$-datum of $G^*$ that is trivial on the roots not contained in $M^*_b$. With these choices, our $\Delta_{II}$ terms for $M^*_b$ and $G^*$ agree. 

For $\Delta_{I}$, we argue that the splitting invariants are the same for $G^*$ and $M^*_b$. The terms $n(\sigma)$ and $g^{-1}\sigma(g)$ are the same for $M^*_b$ and $G^*$ since we can choose $\Int(g): T \to S$ such that $g \in M^*_b$. Finally, the product is over roots $\alpha$ such that $\alpha>0, \sigma^{-1}(\alpha)<0$ but none of these can lie outside of $M^*_b$ by a similar argument to the above.

It remains to check that the $\Delta_{III_2}$ terms agree. It suffices to show that the terms $r_p(\sigma)$ as in \cite[pg32]{LS1} agree for $G^*$ and $M^*_b$. We have $r_p=r_{p_0} s_{p/p_0}$ so we need only check that $r_{p_0}$ and $s_{p/p_0}$ agree. The $r_{p_0}$ agree (defined \cite[pg27]{LS1}) since the $\chi$-datum for $G^*$ is trivial outside of $M^*_b$. The term $s_{p/p_0}$ (\cite[pg24]{LS1}) agrees since $p=p_0$ on roots outside of $M^*_b$ since as above, the $\Gamma_{\Q_p}$-action preserves positivity for such roots.
\end{proof}



Then the transfer factor from $H_{M_b}$ to $M_b$ is given by
\begin{equation*}
    \Delta[\mf{w}^{M^*_b}, z^{\iso}_p](\gamma_{H_{M_b}}, \gamma_0)=\Delta[\mf{w}^{M^*_b}](\gamma_{H_{M_b}}, \gamma^*_0)\langle \inv[z^{iso}_p](\gamma_0, \gamma^*_0), \eta(s)  \rangle^{-1},
\end{equation*}
where $\gamma_{H_{M_b}}$ transfers to $\gamma_0 \in M_b(\Q_p)$.

Finally, we define transfer factors to $J_b$. Associated to $\tilde{b}$ we have (for instance as in \cite[pg 9]{KW}) a cocycle $z^b \in Z^1_{\alg}(\mc{E}^{\iso}, M_b(\ov{\Q}))$ and satisfying $\psi^{-1}_b \circ \psi^{\sigma}_b= \Int(z^b_{\sigma}) $. Then one can check that $\varrho_{M^*_b, p}^{-1}(z^b)z^{\iso}_p \in Z^1(\mc{E}^{\iso}, M^*_b(\ov{\Q_p}))$ where we take $\varrho_{M^*_b, p}$ to be the restriction of $\varrho_p$ to $M^*_b$. We define a transfer factor from $H_{M_b}$ to $J_b$ by
\begin{equation*}
    \Delta[\mf{w}^{M^*_b}, \varrho_{M^*_b, p}^{-1}(z^b)z^{\iso}_p](\gamma_{H_{M_b}}, \delta)=\Delta[\mf{w}^{M^*_b}](\gamma_{H_{M_b}}, \gamma^*_0)\langle \inv[\varrho_{M^*_b, p}^{-1}(z^b)z^{\iso}_p](\delta, \gamma^*_0), \eta(s)  \rangle^{-1},
\end{equation*}
where $\gamma_{H_{M_b}}$ transfers to $\delta$.

As in \cite[Lemma 6.3]{Shi3}, we need to relate the transfer factors from $H_{M_b}$ to $J_b$ and $H_{M_b}$ to $M_b$. Recall that $I^{M^*_b}_{\gamma^*_0}, I^{M_b}_{\gamma_0}$, and $ I^{J_b}_{\delta}$ are all inner forms. The diagram after \cite[Equation (4.13.2)]{Kot1} shows that 
\begin{equation*}
    \inv[\varrho_{M^*_b, p}^{-1}(z^b)z^{\iso}_p]=\inv[z^b]+\inv[z^{\iso}_p],
\end{equation*}
and hence that
\begin{equation*}
 \Delta[\mf{w}^{M^*_b}, \varrho_{M^*_b, p}^{-1}(z^b)z^{\iso}_p](\gamma_{H_{M_b}}, \delta) =  \Delta[\mf{w}^{M^*_b}, z^{\iso}_p](\gamma_{H_{M_b}}, \gamma_0) \langle \inv[z^b](\delta, \gamma_0), \eta(s)  \rangle^{-1}.
\end{equation*}
By definition, we have 
\begin{equation*}
    \langle \inv[z^b](\delta, \gamma_0), \eta(s) \rangle = \langle \tilde{\alpha}_p(\gamma_0, \delta), \lambda \rangle
\end{equation*}
so that
\begin{equation}
     \Delta[\mf{w}^{M^*_b}, \varrho_{M^*_b, p}^{-1}(z^b)z^{\iso}_p](\gamma_{H_{M_b}}, \delta) =  \Delta[\mf{w}^{M^*_b}, z^{\iso}_p](\gamma_{H_{M_b}}, \gamma_0) \langle \tilde{\alpha}_p(\gamma_0, \delta), \lambda \rangle^{-1}.
\end{equation}

Next, we define a character $\overline{\delta}^{\frac{1}{2}}_{P(\nu_b)}$ of $J_b(\mathbb{Q}_p)$ so that $\overline{\delta}^{\frac{1}{2}}_{P(\nu_b)}(\delta)=\delta^{\frac{1}{2}}_{P(\nu_b)}(\gamma_0)$ where $\gamma_0$ is a transfer of $\delta$ and $\delta_{P(\nu_b)}$ is the modulus character. Note that this doesn't depend on our choice of $\gamma_0$ as conjugation by some $m \in M_b$ preserves the unipotent radical of $P(\nu_b)$. Then we have (since we assume $\delta$ is $\nu_b$-acceptable, see \cite[Lemma 3.4]{Shi1})
\begin{equation*}
    \overline{\delta}^{\frac{1}{2}}_{P(\nu_b)}(\delta)=|D^G_{M_b}(\gamma_0)|^{\frac{1}{2}}_p.
\end{equation*}

We now return to the stabilization. We define the function $\phi^0_p := \phi_p \cdot \overline{\delta}^{\frac{1}{2}}_{P(\nu_b)}  \in C^{\infty}_c(J_b(\mathbb{Q}_p))$ and pick $\phi^{H_{M_b}}_p \in C^{\infty}_c(H_{M_b}(\mathbb{Q}_p))$ to be a matching function of $\phi^0_p$. 

Suppose that the pair $(\gamma_0, \lambda) \in SS^r(G)$ associated to a term $O_p(\gamma_0, \lambda, \phi_p)$ comes from an element of $SS^r_{eff}(J_b, G)$. Since by \cite[Lemma 3.6]{Shi1} we have $SS^r_{eff}(J_b, G)$ injects into $SS^r(G)$, we denote the pair in $SS^r_{eff}(J_b, G)$ also by $(\gamma_0, \lambda)$.  Then $(\gamma_0, \lambda)$ is associated to some element $(H, H_{M_b}, s, \eta, \gamma_{H_{M_b}}) \in \mc{EQ}^e_{eff}(J_b, G)$ such that $\gamma_{H_{M_b}}$ transfers to $\gamma_0$ and is stably conjugate to $\gamma_{H}$ in $H(\ov{\Q_p})$. 

Following \cite[pg22]{Shi3}, we now claim that in this situation we have
\begin{equation}
    O_p(\gamma_0, \lambda, \phi_p)=|D^H_{H_{M_b}}(\gamma_{H})|^{-\frac{1}{2}}_p SO^{H_{M_b}(\mathbb{Q}_p)}_{\gamma_{H_{M_b}}}(\phi^{H_{M_b}}_p).
\end{equation}
Indeed, we have
\begin{align*}
     O_p(\gamma_0, \lambda, \phi_p) &= \Delta[\mf{w}_p, z^{\iso}_p](\gamma_{\mb{H}}, \gamma_0)\sum\limits_{\delta \sim_{st} \gamma_0}\langle\tilde{\alpha}_p(\gamma_0; \gamma, \delta), \lambda \rangle^{-1} e(I^{J_b}_{\delta})O^{J_b(\Q_p)}_{\delta}(\phi_p)\\
     &= \Delta[\mf{w}^{M^*_b}, z^{\iso}_p](\gamma_{H_{M_b}}, \gamma_0) |D^G_{M_b}(\gamma_0)|^{ \frac{1}{2}}_p|D^H_{H_{M_b}}(\gamma_H)|^{ -\frac{1}{2}}_p \sum\limits_{\delta \sim_{st} \gamma_0}\langle\tilde{\alpha}_p(\gamma_0; \gamma, \delta), \lambda \rangle^{-1} e(I^{J_b}_{\delta})O^{J_b(\Q_p)}_{\delta}(\phi_p)\\
     &= |D^H_{H_{M_b}}(\gamma_H)|^{ -\frac{1}{2}}_p\sum\limits_{\delta \sim_{st} \gamma_0}\Delta[\mf{w}^{M^*_b}, \varrho_{M^*_b, p}^{-1}(z^b)z^{\iso}_p](\gamma_{H_{M_b}}, \delta)  \ov{\delta}^{\frac{1}{2}}_{P_{(\nu_b)}}(\delta)e(I^{J_b}_{\delta})O^{J_b(\Q_p)}_{\delta}(\phi_p)\\
     &= |D^H_{H_M}(\gamma_H)|^{ -\frac{1}{2}}_p\sum\limits_{\delta \sim_{st} \gamma_0}\Delta[\mf{w}^{M^*_b}, \varrho_{M^*_b, p}^{-1}(z^b)z^{\iso}_p](\gamma_{H_{M_b}}, \delta) e(I^{J_b}_{\delta})O^{J_b(\Q_p)}_{\delta}(\phi^0_p)\\
     &=|D^H_{H_{M_b}}(\gamma_H)|^{-\frac{1}{2}}_p SO^{H_{M_b}(\mathbb{Q}_p)}_{\gamma_{H_{M_b}}}(\phi^{H_{M_b}}_p).
\end{align*}

It remains to rewrite the righthand side of the above in terms of stable orbital integrals on $H$. To this end, we now apply \cite[Lemma 3.9]{Shi3} to the triple $(H, H_{M_b}, \nu)$ and the function $\phi^{H_{M_b}}_p \cdot \delta^{- \frac{1}{2}}_{P(\nu)}$ to get a function $\tilde{\phi^{J^{\mf{e}}_b}} \in C^{\infty}_c(H(\Q_p))$ satisfying 
\begin{equation}
    O^{H(\Q_p)}_{\gamma_H}(\tilde{\phi^{J^{\mf{e}}_b}}) = \delta^{-\frac{1}{2}}_{P(\nu)} O^{H_{M_b}(\Q_p)}_{\gamma_{H_{M_b}}}(\phi^{H_{M_b}}_p),
\end{equation}
for each semisimple $\gamma_H \in H(\Q_p)$ such that there exists a $\nu$-acceptable element $\gamma_{H_{M_b}} \in H_{M_b}(\Q_p)$ conjugate to $\gamma_H$ in $H(\Q_p)$ and $0$ otherwise. We remark that we believe there is a typo in \cite[Lemma 3.9]{Shi3} in the character twists. In particular, we claim that our $\tilde{\phi^{J^{\mf{e}}_b}} \in C^{\infty}_c(H(\Q_p))$ satisfies
\begin{equation}
    \tr(J^H_{P(\nu)^{op}}(\pi) \mid \phi^{H_{M_b}}_p) = \tr( \pi \mid \tilde{\phi^{J^{\mf{e}}_b}}).
\end{equation}

We need to justify why it is valid to apply that lemma to this function. In particular, it suffices to show that $\phi^{H_{M_b}}_p$ can be chosen to be supported on $\nu$-acceptable elements with connected centralizer. To deal with the connected centralizer condition we may as well assume that $\phi^{H_{M_b}}_p$ is supported on strongly regular semisimple elements as this will not change the semisimple orbital integrals.  Let $U \subset H_{M_b}(\Q_p)_{sr}$ be the set of $\nu$-acceptable elements, noting that the stable conjugate of any element of $U$ is also in $U$. Note that if a $(G,H)$-regular semisimple element of $H(\Q_p)$ transfers to a $\nu_b$-acceptable element of $M_b(\Q_p)$, then it is contained in $U$. Consider the function $1_U \cdot \phi^{H_{M_b}}_p$. The support of this function is contained in $U$. Moreover, since $\phi^{H_{M_b}}_p$ is the transfer of a function $\phi^0_p$ supported on the set of $\nu_b$-acceptable elements, it follows that the stable orbital integrals of $1_U \cdot \phi^{H_{M_b}}_p$ are the same as those of $\phi^{H_{M_b}}_p$ (the integral of elements inside $U$ does not change and the integral of elements outside $U$ is already $0$). Hence $1_U\phi^{H_{M_b}}_p$ is also a transfer of $\phi^0_p$ so we are free to choose this function instead.

We therefore get that 
\begin{equation}
     O_p(\gamma_0, \lambda, \phi_p) = |D^H_{H_{M_b}}(\gamma_{\mb{H}})|^{-\frac{1}{2}}_p SO^{H_{M_b}(\mathbb{Q}_p)}_{\gamma_{H_{M_b}}}(\phi^{H_{M_b}}_p) = SO^{H(\Q_p)}_{\gamma_{\mb{H}}}(\tilde{\phi^{J^{\mf{e}}_b}}).
\end{equation}

We now define $h^{\mb{H}}_p$ by
\begin{equation}
    h^{\mb{H}}_p := \sum\limits_{J^{\mf{e}}_b \in X^{J^{\mf{e}}_b}} \tilde{\phi^{J^{\mf{e}}_b}_p}.
\end{equation}
The key fact is that
\begin{lemma}
We have for every $(G,H)$-regular semisimple $\gamma_H \in H(\Q_p)$ that
\begin{equation}
   SO^{H(\Q_p)}_{\gamma_H}(h^{\mb{H}}_p)= O_p(\gamma_0, \lambda, \phi_p)
\end{equation}
if $(H,s, \eta, \gamma_H) \in \mc{EQ}^r_{eff}(G)$ and $0$ otherwise.
\end{lemma}
\begin{proof}
There are two cases to consider. Suppose first that $(H, s, \eta, \gamma_H) \notin \mc{EQ}^r_{eff}(G)$. We show by contradiction that $SO^{H(\Q_p)}_{\gamma_H}(\tilde{\phi^{J^{\mf{e}}_b}})=0$ for each $J^{\mf{e}}_b \in X^{J^{\mf{e}}_b}$. If one of these integrals is nonzero, then we must have for some $J^{\mf{e}}_b \in X^{J^{\mf{e}}_b}$, a $\nu$-acceptable element $\gamma_{H_{M_b}} \in H_{M_b}(\Q_p)$ such that $\gamma_{H_{M_b}}$ and $\gamma_H$ are conjugate in $H$ (since otherwise the integral is $0$ by \cite[Lemma 3.9]{Shi3}). By (1) of Lemma \ref{stabigusalem}, we know that $(H, H_{M_b}, s, \eta', \gamma_{H_{M_b}})$ does not project to an element of $\mc{EQ}^e_{eff}(J_b, G, H)$. Hence if $(\gamma_0, \lambda)$ is the corresponding element of $\mc{SS}^r(M_b)$, we know that either $\gamma_{H_{M_b}}$ doesn't transfer to $J_b$ or it transfers to an element that isn't $\nu_b$ acceptable. Either way, we have that $SO^{H(\Q_p)}_{\gamma_H}(\tilde{\phi^{J^{\mf{e}}_b}})=0$.

We now consider the case where $(H, s, \eta, \gamma_H) \in  \mc{EQ}^r_{eff}(G)$. By the computations we have done above, it suffices to show there is a unique $J^{\mf{e}}_b \in X^{J^{\mf{e}}_b}$ such that $SO^{H(\Q_p)}_{\gamma_H}(\tilde{\phi^{J^{\mf{e}}_b}}) \neq 0$. If this integral is nonzero, then there exists a $\gamma_{H_{M_b}} \in H_{M_b}(\Q_p)$ such that $\gamma_{H_{M_b}}$ and $\gamma_H$ are stably conjugate in $H(\Q_p)$.

Suppose we have two such tuples $(H, H_{M_b}, s, \eta_1, \gamma_{H_{M_b}})$ and $(H, H'_{M_b}, s, \eta_2, \gamma_{H'_{M_b}})$. These tuples must project to elements of $\mc{EQ}^e_{eff}(J_b, G;H)$ since otherwise the corresponding integrals are $0$ by the previous part of the proof. Let $(\gamma_0, \lambda)$ and $(\gamma'_0, \lambda')$ denote representatives of the classes in $\mc{SS}^r_{ef}(G)$ that we get via the map $\mc{EQ}^e_{eff}(J_b,G;H) \to \mc{EQ}^r_{ef}(G) \to \mc{SS}^r_{ef}(G)$. Then, since both $\gamma_{H_{M_b}}$ and $\gamma_{H'_{M_b}}$ are stably conjugate to $\gamma_H$ in $H(\Q_p)$, we must have by Lemma \ref{SSEQMGGcomm} that $(\gamma_0, \lambda)$ and $(\gamma'_0, \lambda')$ are in the same class of $\mc{SS}^r_{ef}(G)$. In particular, $\gamma_0, \gamma'_0$ are stably conjugate in $G(\Q_p)$. Further, we have that $\gamma_0, \gamma'_0$ are $\nu_b$-acceptable and so by \cite[Lemma 3.5]{Shi3}, they are stably conjugate in $M_b(\Q_p)$. This implies that $(\gamma_0, \lambda)$ and $(\gamma'_0, \lambda')$ are in the same class of $\mc{SS}^e_{eff}(M_b, G)$ and hence that $(H, H_{M_b}, s, \eta_1, \gamma_{H_{M_b}})$ and $(H, H'_{M_b}, s, \eta_2, \gamma_{H'_{M_b}})$ project to the same class of $\mc{EQ}^e_{eff}(M_b, G;H)$. Then by (2) of Lemma \ref{stabigusalem}, we must have that $(H, H_{M_b}, s, \eta_1)$ and $(H, H'_{M_b}, s, \eta_2)$ are the same representative of $X^{J^{\mf{e}}_b}$.
\end{proof}

Now, we define $h^{\mathbf{H}}=h^{\mathbf{H},p}h^{'\mathbf{H}}_{\infty}h^{\mathbf{H}}_p \in C^{\infty}_c(\mathbf{H}(\mathbb{A}))$. Then Shin's final formula is as follows.
\begin{theorem}[Shin]
Let $\phi \in C^{\infty}_c(\mathbf{G}(\mathbb{A}^{p, \infty} \times J_b(\mathbb{Q}_p))$ be an acceptable function. For each $\mathcal{H}^{\mathfrak{e}} \in \mathcal{E}^{ell}(\mathbf{G})$ we get a function $h^{\mathbf{H}}$ as constructed above and have the equality
\begin{equation}
\mathrm{tr}(\phi | H^*_c(\mathrm{Ig}_{\Sigma_b}, \mathcal{L}_{\xi}))= \sum\limits_{X^{\mf{e}}} \iota(\mathbf{G}, \mathbf{H})ST^{\mathbf{H}}_{ell}(h^{\mathbf{H}}).
\end{equation}
\end{theorem}

Equivalently, we have

\begin{equation}
    \mathrm{tr}(\phi | H^*_c(\mathrm{Ig}_{\Sigma_b}, \mathcal{L}_{\xi}))= \sum\limits_{\mathcal{H}^{\mathfrak{e}} \in X^{\mf{e}}} \iota_r(\mathbf{G}, \mathbf{H})ST^{\mathbf{H}}_{r, ell}(h^{\mathbf{H}}).
\end{equation}

\subsection{Destabilization of the trace formula for Igusa varieties}

As in the case of Shimura varieties, we discuss how to rewrite the trace of the cohomology of Igusa varieties as the trace of a representation of $\mathbf{G}(\mathbb{A}^{p, \infty}) \times J_b(\mathbb{Q}_p)$. We note that the cohomology of Igusa varieties does not carry an action of $\Gamma_E$. For this reason, some aspects of the destabilization are simpler than the Shimura variety case. However, we encounter additional complications related to endoscopy of Levi subgroups at $p$.

As before, away from $p, \infty$ we have
\begin{align*}
    &\sum\limits_{\pi^{\mathbf{H}, p, \infty} \in \Pi_{\psi^{\mathbf{H}, p, \infty}}(\mathbf{H}^{p, \infty}, 1)} \langle \pi^{\mathbf{H},p, \infty}, s_{\psi^{\mathbf{H},p, \infty}} \rangle \mathrm{tr}(\pi^{\mathbf{H},p, \infty} \mid h^{\mathbf{H},p, \infty})\\
    &=e(\mathbf{G}^{p, \infty}) \sum\limits_{\pi^{p, \infty} \in \Pi_{\psi^{p, \infty}}(\mathbf{G}^{p, \infty}, \varrho^{iso, p ,\infty})} \langle \pi^{p, \infty}, \overline{s}s_{\psi^{p, \infty}} \rangle \mathrm{tr}(\pi^{p, \infty} \mid \phi^p).
\end{align*}

At $\infty$, we have
\begin{align*}
    &\sum\limits_{\pi^{\mathbf{H}}_{\infty} \in \Pi_{\psi^{\mathbf{H}}_{\infty}}(\mathbf{H}_{\infty}, 1)} \langle \pi^{\mathbf{H}}_{\infty}, s_{\psi^{\mathbf{H}}_{\infty}} \rangle \mathrm{tr}( \pi^{\mathbf{H}}_{\infty} \mid h^{',\mathbf{H}}_{\xi})\\
    &=e(\mathbf{G}_{\infty})(-1)^{q(\mathbf{G})}\langle \lambda_{\pi_{\infty}}, ss_{\psi} \rangle\langle \pi_{\infty}, \overline{s} \rangle.
\end{align*}

The situation at $p$ is of course quite different from the case of Shimura varieties. Recall that we have fixed a set $X^{\mf{e}}$ of refined elliptic endoscopic data.  We also have a set of representatives $X^{\mf{e}}_{J_b}$ of $\mc{E}^i_{eff}(J_b, G; H)$. We recall that elements of this set are \emph{inner} equivalence classes of embedded endoscopic data whose embedded equivalence class is an element of $\mc{E}^e_{eff}(J_b, G; H)$. According to Lemma \ref{torendlem}, this latter set consists of isomorphism classes of embedded endoscopic data in $Y^{-1}(\mc{H}^{\mf{e}})$ such that there exist maximal tori $T_{H_{M_b}}, T_{M_b}, T_{J_b}$ of their respective groups such that each is the transfer of the others.

We recall the functions defined in the stabilization of the cohomology of Igusa varieties.

\begin{itemize}
    \item $\phi_p \in C^{\infty}_c(J_b(\Q_p))$ supported on $\nu_b$-acceptable elements.
    \item $\phi^0_p := \phi_p \cdot \ov{\delta}^{1/2}_{P(\nu_b)}$.
    \item $\phi^{H_{M_b}}_p$ a matching function of $\phi^0_p \in C^{\infty}_c(H_{M_b}(\Q_p))$ for $J^{\mf{e}}_b=(H, H_{M_b}, s, \eta) \in X^{\mf{e}}_{J_b}$.
    \item $\tilde{\phi}^{J^{\mf{e}}_b}_p$ constructed by applying \cite[Lemma 3.9]{Shi3} to the triple $(H, H_{M_b}, \nu)$ and function $\phi^{H_{M_b}}_p$ with a twist of $\delta^{-1/2}_{P(\nu)}$.
    \item $h^{\mb{H}}_p := \sum\limits_{J^{\mf{e}}_b \in X^{\mf{e}}_{J_b}} \tilde{\phi}^{J^{\mf{e}}_b}_p$.
\end{itemize}

Now for $\pi \in \Gr(H(\Q_p))$, we have by definition and also applying \cite[Lemma 3.9 (ii)]{Shi3}
\begin{align*}
    \tr( \pi \mid h^{\mb{H}}_p) & = \sum\limits_{J^{\mf{e}}_b \in X^{\mf{e}}_{J_b}} \tr( \pi \mid \tilde{\phi}^{J^{\mf{e}}_b}_p)\\
    & = \sum\limits_{J^{\mf{e}}_b \in X^{\mf{e}}_{J_b}} \tr( J^H_{P(\nu)^{op}}(\pi) \mid \phi^{H_{M_b}}_p).
\end{align*}

Then for $\pi \in \Gr(H(\Q_p))^{st}$, we have
\begin{equation}
    \tr( \pi \mid h^{\mb{H}}_p) = \sum\limits_{J^{\mf{e}}_b \in X^{\mf{e}}_{J_b}} \tr( \Trans^{H_{M_b}}_{J_b}(J^H_{P(\nu)^{op}}(\pi)) \otimes \ov{\delta}^{1/2}_{P(\nu_b)} \mid \phi_p).\\
\end{equation}

Note that for this to make sense, we first of all need \cite[Lemma 3.3]{Hir1} which states that the Jacquet module of a representation with stable character will again have stable character so that applying $\Trans$ makes sense.

Motivated by this computation, we have the following definition.
\begin{definition}
We define the map $\mathrm{Red}^{\mathcal{H}^{\mathfrak{e}}}_b: \mathbb{C}[\mathrm{Irr}(H(\mathbb{Q}_p))]^{st} \to \mathbb{C}[\mathrm{Irr}(J_b(\mathbb{Q}_p)]$ by
\begin{equation}
    \pi \mapsto  \sum\limits_{J^{\mf{e}}_b \in X^{\mf{e}}_{J_b}}  \mathrm{Trans}^{H_{M_b}}_{J_b}(J^H_{P(\nu)^{op}} (\pi)) \otimes \overline{\delta}^{\frac{1}{2}}_{P(\nu_b)}.
\end{equation}
\end{definition}
In particular, we have that for $ \pi \in \mathbb{C}[\mathrm{Irr}(H(\mathbb{Q}_p))]^{st}$,
\begin{equation}
    \mathrm{tr}(\pi \mid h^{\mathbf{H}}_p)=\mathrm{tr}(\mathrm{Red}^{\mathcal{H}^{\mathfrak{e}}}_b(\pi) \mid \phi_p).
\end{equation}

We can now carry out the destabilization of the cohomology of Igusa varieties. We make Assumption \ref{STELLA} as we did for Shimura varieties. We have
\begin{align*}
    &\mathrm{tr}(\phi | H^*_c(\mathrm{Ig}_{\Sigma_b}, \mathcal{L}_{\xi}))\\
    & = \sum\limits_{\mathcal{H}^{\mathfrak{e}} \in X^{\mf{e}}} \iota_r(\mathbf{G}, \mathbf{H})ST^{\mathbf{H}}_{r, ell}(h^{\mathbf{H}})\\
    &=\sum\limits_{\mathcal{H}^{\mathfrak{e}} \in X^{\mf{e}}} \iota_r(\mathbf{G}, \mathbf{H})ST^{\mathbf{H}}_{r, disc}(h^{\mathbf{H}})\\
    &=\sum\limits_{\mathcal{H}^{\mathfrak{e}} \in X^{\mf{e}}} \iota_r(\mathbf{G}, \mathbf{H}) \sum\limits_{[\psi_{\mathbf{H}}] \in X^{sp}_{\mc{H}^{\mf{e}}}} \frac{1}{|\mathcal{S}_{\psi^{\mathbf{H}}}|} \sum\limits_{\pi^{\mathbf{H}} \in \Pi_{\psi^{\mathbf{H}}}} \epsilon_{\psi^{\mathbf{H}}}(s_{\psi^{\mathbf{H}}}) \langle \pi^{\mathbf{H}}, s_{\psi^{\mathbf{H}}} \rangle \mathrm{tr}( \pi \mid h^{\mathbf{H}})\\
    &=\sum\limits_{\mathcal{H}^{\mathfrak{e}} \in X^{\mf{e}}} \frac{\tau(\mathbf{G})}{\tau(\mathbf{H})|\mathrm{Out}_r(\mathcal{H}^{\mathfrak{e}})|} \sum\limits_{[\psi_{\mathbf{H}}] \in X^{sp}_{\mc{H}^{\mf{e}}}} \frac{1}{|\mathcal{S}_{\psi^{\mathbf{H}}}|} \sum\limits_{\pi^{\mathbf{H}} \in \Pi_{\psi^{\mathbf{H}}}} \epsilon_{\psi^{\mathbf{H}}}(s_{\psi^{\mathbf{H}}}) \langle \pi^{\mathbf{H}}, s_{\psi^{\mathbf{H}}} \rangle \mathrm{tr}( \pi \mid h^{\mathbf{H}})\\
    &=\sum\limits_{\mathcal{H}^{\mathfrak{e}} \in X^{\mf{e}}} \,\ \sum\limits_{[[\psi_{\mathbf{H}}]]: [\psi_{\mb{H}}] \in X^{sp}_{\mc{H}^{\mf{e}}}} \frac{1}{|  _{\mathbf{G}}S_{\psi^{\mathbf{H}}}| \cdot |\mathrm{Out}_r(\mathcal{H}^{\mathfrak{e}})|} \sum\limits_{\pi^{\mathbf{H}} \in \Pi_{\psi^{\mathbf{H}}}} \epsilon_{\psi^{\mathbf{H}}}(s_{\psi^{\mathbf{H}}}) \langle \pi^{\mathbf{H}}, s_{\psi^{\mathbf{H}}} \rangle \mathrm{tr}( \pi \mid h^{\mathbf{H}})\\
    &=\sum\limits_{(\psi,s) \in X_{\Psi}}  \,\ \sum\limits_{\mathcal{H}^{\mathfrak{e}} \in X^{\mf{e}}} \,\ \sum\limits_{(\mc{H}^{\mf{e}}, [[\psi_{\mathbf{H}}]]) \in Y^{-1}(\psi, s)} \frac{1}{|  _{\mathbf{G}}S_{\psi^{\mathbf{H}}}| \cdot |\mathrm{Out}_r(\mathcal{H}^{\mathfrak{e}})|} \sum\limits_{\pi^{\mathbf{H}} \in \Pi_{\psi^{\mathbf{H}}}} \epsilon_{\psi^{\mathbf{H}}}(s_{\psi^{\mathbf{H}}}) \langle \pi^{\mathbf{H}}, s_{\psi^{\mathbf{H}}} \rangle \mathrm{tr}( \pi \mid h^{\mathbf{H}})
     \end{align*}
    \begin{align*}
    &=\sum\limits_{(\psi,s) \in X_{\Psi}}  \,\ \sum\limits_{\mathcal{H}^{\mathfrak{e}} \in X^{\mf{e}}} \,\ \sum\limits_{(\mc{H}^{\mf{e}}, [[\psi_{\mathbf{H}}]]) \in Y^{-1}(\psi, s)} |\frac{Z(\widehat{\mb{G}})^{\Gamma_{\Q}}}{\eta(C_{\psi^{\mb{H}}})}| \cdot \frac{1}{|\mathrm{Out}_r(\mathcal{H}^{\mathfrak{e}})|} \sum\limits_{\pi^{\mathbf{H}} \in \Pi_{\psi^{\mathbf{H}}}} \epsilon_{\psi^{\mathbf{H}}}(s_{\psi^{\mathbf{H}}}) \langle \pi^{\mathbf{H}}, s_{\psi^{\mathbf{H}}} \rangle \mathrm{tr}( \pi \mid h^{\mathbf{H}})\\
    &=\sum\limits_{(\psi, s) \in X_{\Psi}} \,\ \sum\limits_{\mathcal{H}^{\mathfrak{e}} \in X^{\mf{e}}} \,\ \sum\limits_{(\mc{H}^{\mf{e}}, [[\psi_{\mathbf{H}}]]) \in Y^{-1}(\psi, s)} |\frac{Z(\widehat{\mb{G}})^{\Gamma_{\Q}}}{\eta(C_{\psi^{\mb{H}}})}|\cdot \frac{1}{|\mathrm{Out}_r(\mathcal{H}^{\mathfrak{e}})|} \sum\limits_{\pi^{p,\infty} \in \Pi_{\psi^{\infty}}(\mathbf{G}, z^{iso, \infty})} \epsilon_{\psi}(\overline{s}s_{\psi}) \langle \pi^p, \overline{s}s_{\psi} \rangle \\
    & \cdot (-1)^{q(\mathbf{G})}\langle \lambda_{\pi_{\infty}}, ss_{\psi}\rangle e(\mathbf{G}^{p, \infty})e(\mathbf{G}_{\infty})\mathrm{tr}( \pi^{p, \infty} \mid \phi^{p})\mathrm{tr}(\mathrm{Red}^{\mathcal{H}^{\mathfrak{e}}}_b ( \sum\limits_{\pi^{\mathbf{H}}_p \in \Pi_{\psi^{\mathbf{H}}_p}} \langle \pi^{\mathbf{H}}_p, s_{\psi^{\mathbf{H}}} \rangle \pi^{\mathbf{H}}_p) \mid \phi_p).\\
\end{align*}
To complete the destabilization, we need to show that $\mathrm{Red}^{\mathcal{H}^{\mathfrak{e}}}_b(\sum\limits_{\pi^{\mathbf{H}}_p \in \Pi_{\psi^{\mathbf{H}}_p}} \langle \pi^{\mathbf{H}}_p, s_{\psi^{\mathbf{H}}} \rangle \pi^{\mathbf{H}}_p )$ only depends on $\psi, s, b$. This follows from Proposition \ref{EPeqSP} which shows we can recover $(\mc{H}^{\mf{e}}, [[\psi^H]])$ from $(\psi, s)$.

Hence we use the notation $\mathrm{Red}_b(\psi, s) := \mathrm{Red}^{\mathcal{H}^{\mathfrak{e}}}_b(\sum\limits_{\pi^{\mathbf{H}}_p \in \Pi_{\psi^{\mathbf{H}}_p}} \langle \pi^{\mathbf{H}}_p, s_{\psi^{\mathbf{H}}} \rangle \pi^{\mathbf{H}}_p)$ to indicate this dependence.

Then the formula for the trace of the cohomology of Igusa varieties becomes:
\begin{align*}
    &\sum\limits_{(\psi, s) \in X_{\Psi}} \,\ \sum\limits_{\mathcal{H}^{\mathfrak{e}} \in X^{\mf{e}}} \,\ \sum\limits_{(\mc{H}^{\mf{e}},[[\psi_{\mathbf{H}}]]) \in Y^{-1}(\psi, s)} |\frac{Z(\widehat{\mb{G}})^{\Gamma_{\Q}}}{\eta(C_{\psi^{\mb{H}}})}|\cdot \frac{1}{|\mathrm{Out}_r(\mathcal{H}^{\mathfrak{e}})|} \sum\limits_{\pi^{p, \infty} \in \Pi_{\psi^{p, \infty}}(\mathbf{G}, z^{iso, p,\infty})} \epsilon_{\psi}(\overline{s}s_{\psi}) \langle \pi^p, \overline{s}s_{\psi} \rangle \\
    & \cdot (-1)^{q(\mathbf{G})}\langle \lambda_{\pi_{\infty}}, ss_{\psi}\rangle e(\mathbf{G}^{p, \infty})e(\mathbf{G}_{\infty})\mathrm{tr}( \pi^{p, \infty} \mid \phi^{p})\mathrm{tr}(\mathrm{Red}_b(\psi, s) \mid \phi_p)\\
    &=\sum\limits_{(\psi, s) \in X_{\Psi}} |\Out_r(\mc{H}^{\mf{e}})| \cdot |\frac{\eta(C_{\psi^{\mb{H}}})}{Z_{C_{\psi}}(s)}|\cdot|\frac{Z(\widehat{\mb{G}})^{\Gamma_{\Q}}}{\eta(C_{\psi^{\mb{H}}})}|\cdot \frac{1}{|\mathrm{Out}_r(\mathcal{H}^{\mathfrak{e}})|}\sum\limits_{\pi^{p, \infty} \in \Pi_{\psi^{p, \infty}}(\mathbf{G}, z^{iso, p,\infty})} \epsilon_{\psi}(\overline{s}s_{\psi}) \langle \pi^p, \overline{s}s_{\psi} \rangle \\
    &  \cdot (-1)^{q(\mathbf{G})}\langle \lambda_{\pi_{\infty}}, ss_{\psi}\rangle e(\mathbf{G}^{p, \infty})e(\mathbf{G}_{\infty})\mathrm{tr}( \pi^{p, \infty} \mid \phi^{p})\mathrm{tr}(\mathrm{Red}_b(\psi, s) \mid \phi_p)\\
    &=\sum\limits_{(\psi, s) \in X_{\Psi}}  |\frac{Z(\widehat{\mb{G}})^{\Gamma_{\Q}}}{Z_{C_{\psi}}(s)}|\sum\limits_{\pi^{p, \infty} \in \Pi_{\psi^{p, \infty}}(\mathbf{G}, z^{iso, p,\infty})} \epsilon_{\psi}(\overline{s}s_{\psi}) \langle \pi^p, \overline{s}s_{\psi} \rangle \\
    &  \cdot (-1)^{q(\mathbf{G})}\langle \lambda_{\pi_{\infty}}, ss_{\psi}\rangle e(\mathbf{G}^{p, \infty})e(\mathbf{G}_{\infty})\mathrm{tr}( \pi^{p, \infty} \mid \phi^{p})\mathrm{tr}(\mathrm{Red}_b(\psi, s) \mid \phi_p)\\
    &=\sum\limits_{([\psi], [s]) \in X_{\Psi}}  \sum\limits_{s \in [s]} \frac{1}{|\mc{S}_{\psi}|}\sum\limits_{\pi^{p, \infty} \in \Pi_{\psi^{p, \infty}}(\mathbf{G}, z^{iso, p,\infty})} \epsilon_{\psi}(\overline{s}s_{\psi}) \langle \pi^p, \overline{s}s_{\psi} \rangle \\
    &  \cdot (-1)^{q(\mathbf{G})}\langle \lambda_{\pi_{\infty}}, ss_{\psi}\rangle e(\mathbf{G}^{p, \infty})e(\mathbf{G}_{\infty})\mathrm{tr} (\pi^{p, \infty} \mid \phi^{p})\mathrm{tr}(\mathrm{Red}_b(\psi, s) \mid \phi_p)\\ 
\end{align*}
This implies by \cite[Lemma 6.4 ]{Shi4} (up to semisimplification) the following formula for the cohomology of Igusa varieties. Notice that since there is no Galois action, we also do not see characters of $S_{\psi}$ in the above as in the Shimura variety case. We also note that since $\mathbf{G}_p$ is assumed to be quasisplit, $e(\mathbf{G}_p)=1$ so that $e(\mathbf{G}^{p, \infty})e(\mathbf{G}_{\infty})=e(\mathbf{G})=1$.
\begin{equation}
    H^*_c(\mathrm{Ig}_{\Sigma_b}, \mathcal{L}_{\xi})
\end{equation}
\begin{equation*}
    =\sum\limits_{[\psi], s : ([\psi], [s]) \in X_{\Psi}} \frac{1}{|\mathcal{S}_{\psi}|}\sum\limits_{\pi^{p, \infty} \in \Pi_{\psi^{p, \infty}}(\mathbf{G}, z^{iso, p,\infty})} \epsilon_{\psi}(\overline{s}s_{\psi}) \langle \pi^p, \overline{s}s_{\psi} \rangle (-1)^{q(\mathbf{G})}\langle \lambda_{\pi_{\infty}}, ss_{\psi}\rangle
    \pi^{p, \infty}\boxtimes \mathrm{Red}_b(\psi, s).  
\end{equation*}

\section{Cohomology of Rapoport--Zink spaces}
In this section we discuss Rapoport--Zink spaces and their cohomology and describe the Mantovan formula relating the cohomology of Rapoport--Zink spaces, Igusa varieties, and Shimura varieties. We largely follow \cite{RV1}. However, we use the covariant Dieudonne functor so our normalization looks like that of \cite{Shi1}. We then combine the results of the previous sections to derive a formula for the cohomology of Rapoport--Zink spaces.

\subsection{Local Shimura varieties}

\begin{definition}
A local Shimura datum over $\mathbb{Q}_p$ is a tuple $(G, b, \{\mu \})$ such that 
\begin{enumerate}
    \item $G$ is a connected reductive group over $\mathbb{Q}_p$\\
    \item $\{ \mu \}$ is a conjugacy class of minuscule cocharacters $\mu: \mathbb{G}_{m, \overline{\mathbb{Q}_p}} \to G_{\overline{\mathbb{Q}_p}}$\\
    \item $b \in \mathbf{B}(\Q_p, G, -\mu)$
\end{enumerate}
\end{definition}

We denote by $E_{\mu}$ the field of definition of the conjugacy class $\{\mu\}$ inside $\overline{\mathbb{Q}_p}$. There exists a tower of rigid spaces $(\bb{M}_{G, b, \mu, K})_K$ over $\breve{E}_{\mu}$. The tower is indexed by compact open subgroups $K \subset G(\mathbb{Q}_p)$. We fix a decent representative $\tilde{b}$ of $b$ and let $J_b=J_{\tilde{b}}$. Then each $\bb{M}_{G, b, \mu, K}$ has an action by $J_b(\mathbb{Q}_p)$ and the tower $(\bb{M}_{G, b, \mu, K})_{K}$ has an action of $G(\mathbb{Q}_p)$ by Hecke correspondences. The tower $(\bb{M}_{G, b, \mu, K})_{K}$ is equipped with a $W_{E_{\mu}}$-descent datum.

In this paper we only consider local Shimura varieties of PEL-type. These are the spaces that arise in the $p$-adic uniformization of PEL-type Shimura varieties and are known to exist by \cite{RZ1}.

We denote by $H^i_c(\bb{M}_{G, b, \mu, K}, \ov{\Q_{\ell}})$ the $\ell$-adic cohomology with compact supports of $\bb{M}_{G, b, \mu, K}$. Let $\rho$ be an admissible representation of $J_b(\mathbb{Q}_p)$. Then following \cite{RV1}, we define 
\begin{equation*}
H^{i,j}(G, b, \mu)[\rho] := \varinjlim_K \mathrm{Ext}^j_{J_b(\mathbb{Q}_p)}(H^i_c(\bb{M}_{G, b, \mu, K}, \ov{\Q_{\ell}}), \rho).
\end{equation*}

Finally, we define the homomorphism $\mathrm{Mant}_{G, b, \mu}: \mathrm{Groth}(J_b(\mathbb{Q}_p)) \to \mathrm{Groth}(G(\mathbb{Q}_p) \times W_{E_{\{\mu\}_G}}$ by
\begin{equation}
    \mathrm{Mant}_{G, b, \mu}(\rho) := \sum\limits_{i, j} (-1)^{i+j} H^{i,j}(G, b, \mu)[\rho](- \mathrm{dim}\bb{M}).
\end{equation}
\subsection{Mantovan's formula}
We now return to the setup of \S3. In particular, we have a fixed PEL-type Shimura datum $(\mb{G}, X)$ and $\mb{G}$ is assumed to be anisotropic modulo center and unramified at $p$.

By \cite[Theorem 22]{Man1}, we have the following equality in $\mathrm{Groth}(\mathbf{G}(\mathbb{A}^{\infty}) \times W_{E_{\mu}})$

\begin{equation}
    H^*_c(\mathrm{Sh}, \mathcal{L}_{\xi})=\sum\limits_{b \in \mathbf{B}(G, -\mu)} \mathrm{Mant}_{G, b, \mu}(H^*_c(\mathrm{Ig}_{\Sigma_b}, \mathcal{L}_{\xi})). 
\end{equation}
This formula has been established for more general $\mb{G}$ in non-compact PEL cases by \cite{LS2018} and for Hodge type case by \cite{HK}.

\subsection{Cohomology formulas for Rapoport--Zink spaces}
In this subsection we combine the results of the previous sections to derive formulas for the cohomology of Rapoport--Zink spaces. We continue with the assumptions of \S3.

We substitute our formula for the cohomology of Shimura varieties and Igusa varietes  in the previous sections to get the following equality in $\mathrm{Groth}(\mathbf{G}(\A_f) \times W_{E_{\{\mu\}_G}})$
\begin{align*}
&\sum\limits_{[\psi] \in X_{\Psi}} \, \sum\limits_{\nu}\sum\limits_{\pi^{\infty} \in \Pi_{\psi^{\infty}}(\mathbf{G}, z^{iso, \infty})} m(\pi^{\infty}, \nu)\nu(s_{\psi})(-1)^{q(\mathbf{G})}(\pi^{\infty}) \boxtimes V(\psi, \nu)_{\lambda}\\
&=\sum\limits_{b \in \mathbf{B}(G, -\mu)} \mathrm{Mant}_{G,b,\mu}(\sum\limits_{[\psi], s : ([\psi], [s]) \in X_{\Psi}} \frac{1}{|\mathcal{S}_{\psi}|}\sum\limits_{\pi^{p, \infty} \in \Pi_{\psi^{p, \infty}}(\mathbf{G}, z^{iso, p,\infty})} \epsilon_{\psi}(\overline{s}s_{\psi})\\
&\cdot \langle \pi^p, \overline{s}s_{\psi} \rangle (-1)^{q(\mathbf{G})}\langle \lambda_{\pi_{\infty}}, ss_{\psi}\rangle
    \pi^{p, \infty}\boxtimes \mathrm{Red}_b(\psi, s)).  
\end{align*}

Fix a representation $\pi_p \in \mathrm{Irr}(G(\mathbb{Q}_p))$ such that there exists a representation $\pi^{\infty} \in \mathrm{Irr}(\mb{G}(\mathbb{A}^{\infty}))$ appearing in the cohomology of Shimura varieties and such that the local factor at $p$ is isomorphic to $\pi_p$.
\begin{assumption}{\label{unicityassump}}
Choose $\pi^{\infty} \in \mathrm{Irr}(\mb{G}(\mathbb{A}^{\infty}))$ appearing in the cohomology of Shimura varieties and such that the local factor at $p$ is isomorphic to $\pi_p$. We assume that there is at most one global $A$-parameter $\psi$ appearing in the cohomology of Shimura varieties with sheaf $\mathcal{L}_{\xi}$ such that $\pi^{\infty}$ is the away from $\infty$ piece of some $\pi \in \Pi_{\psi}(\mb{G}, \varrho)$.
\end{assumption}

With $\pi^{\infty}$ and $\psi$ fixed as in the assumption, the $\pi^{p, \infty}$-isotypic piece of the formula above is the following (Keeping in mind that the $\mathrm{Mant}_{G,b,\mu}$ acts as the identity away from $p$). 
\begin{align*}
&\sum\limits_{\nu}\sum\limits_{\pi_p \in \Pi_{\psi_p}(\mathbf{G}, \varrho_p)} m(\pi^{\infty}, \nu)\nu(s_{\psi})(\pi^{p, \infty}) \boxtimes \pi_p \boxtimes V(\psi, \nu)_{\lambda}\\
&=\sum\limits_{b \in \mathbf{B}(G, -\mu)} \mathrm{Mant}_{G,b,\mu}\left(\sum\limits_{s : ([\psi], [s]) \in X_{\Psi}} \frac{1}{|\mathcal{S}_{\psi}|} \epsilon_{\psi}(\overline{s}s_{\psi}) \langle \pi^p, \overline{s}s_{\psi} \rangle \langle \lambda_{\pi_{\infty}}, ss_{\psi}\rangle
    \pi^{p, \infty}\boxtimes \mathrm{Red}_b(\psi, s)\right).  
\end{align*}

We now substitute in the formula for $m(\pi_{\infty}, \nu)$ to get

\begin{align*}
&\sum\limits_{\nu}\sum\limits_{\pi_p \in \Pi_{\psi_p}(\mathbf{G}, \varrho_p)} \frac{1}{|\mathcal{S}_{\psi}|}\sum\limits_{s : ([\psi],[s]) \in X_{\Psi}} \epsilon_{\psi}(\overline{s})\langle \pi, \overline{s} \rangle \nu(s)^{-1}\langle \lambda_{\pi_{\infty}}, s \rangle\nu(s_{\psi})(\pi^{p, \infty}) \boxtimes \pi_p \boxtimes V(\psi, \nu)_{\lambda}\\
&=\sum\limits_{b \in \mathbf{B}(G, -\mu)} \mathrm{Mant}_{G,b,\mu}\left(\sum\limits_{s : ([\psi], [s]) \in X_{\Psi}} \frac{1}{|\mathcal{S}_{\psi}|} \epsilon_{\psi}(\overline{s}s_{\psi}) \langle \pi^p, \overline{s}s_{\psi} \rangle \langle \lambda_{\pi_{\infty}}, ss_{\psi}\rangle
    \pi^{p, \infty}\boxtimes \mathrm{Red}_b(\psi, s)\right).  
\end{align*}
We now argue that $s$-dependence of the right-hand side depends only on $\ov{s}$. In particular, we need to show that if we change $s$ to some $sz$ such that $z \in Z(\widehat{\mb{G}})^{\Gamma_{\Q}}$, then this will not change the value of the right-hand side.  The term $\langle \lambda_{\pi_{\infty}}, z\rangle$ is equal to $\mu(z)$. On the other hand, the terms $\Red_b(\psi,sz)$ and $\Red_b(\psi,s)$ differ by $\kappa(b)(z)=\mu(z)^{-1}$ since $b \in \mb{B}(G, - \mu)$. 

We can now translate the Igusa sum by $s_{\psi}$ to get

\begin{align*}
&\sum\limits_{\nu}\sum\limits_{\pi_p \in \Pi_{\psi_p}(\mathbf{G}, \varrho_p)} \frac{1}{|\mathcal{S}_{\psi}|}\sum\limits_{s : ([\psi],[s]) \in X_{\Psi}} \epsilon_{\psi}(\overline{s})\langle \pi, \overline{s} \rangle \nu(s)^{-1}\langle \lambda_{\pi_{\infty}}, s \rangle\nu(s_{\psi})(\pi^{p, \infty}) \boxtimes \pi_p \boxtimes V(\psi, \nu)_{\lambda}\\
&=\sum\limits_{b \in \mathbf{B}(G, -\mu)} \mathrm{Mant}_{G,b,\mu}(\sum\limits_{s : ([\psi], [s]) \in X_{\Psi}} \frac{1}{|\mathcal{S}_{\psi}|} \epsilon_{\psi}(\overline{s}) \langle \pi^p, \overline{s} \rangle \langle \lambda_{\pi_{\infty}}, s\rangle
    \pi^{p, \infty}\boxtimes \mathrm{Red}_b(\psi, ss_{\psi})).  
\end{align*}

In particular, we must have equality ``at $p$'' above, which gives
\begin{align*}
&\sum\limits_{\nu}\sum\limits_{\pi_p \in \Pi_{\psi_p}(\mathbf{G}, \varrho_p)} \frac{1}{|\mathcal{S}_{\psi}|}\sum\limits_{s : ([\psi],[s]) \in X_{\Psi}} \epsilon_{\psi}(\overline{s})\langle \pi, \overline{s} \rangle \nu(s)^{-1}\langle \lambda_{\pi_{\infty}}, s \rangle\nu(s_{\psi}) \pi_p \boxtimes V(\psi, \nu)_{\lambda}\\
&=\sum\limits_{b \in \mathbf{B}(G, -\mu)} \mathrm{Mant}_{G,b,\mu}(\sum\limits_{s : ([\psi], [s]) \in X_{\Psi}} \frac{1}{|\mathcal{S}_{\psi}|} \epsilon_{\psi}(\overline{s}) \langle \pi^p, \overline{s} \rangle \langle \lambda_{\pi_{\infty}}, s\rangle
    \mathrm{Red}_b(\psi, ss_{\psi})).  
\end{align*}

Now, we have a natural embedding $i_p: C_{\psi} \hookrightarrow C_{\psi_p}$. We would like to rewrite the above equation purely locally such that the only dependence on $\psi$ is the subgroup $i_p(C_{\psi}) \subset C_{\psi_p}$. We let $X_{\psi}$ denote the set of $s \in C_{\psi}$ such that $([s], [\psi]) \in X_{\Psi}$.

We wish to replace the sum over $\nu \in X^*(C_{\psi})$ by a sum only depending on $\psi_p$. We decompose $V=\oplus_{\rho} V_{\rho}$ such that each $\rho$ is an irreducible representation of $\mathcal{L}_{E_{\mf{p}}} \times SL_2(\mathbb{C})$ appearing in the semisimple decomposition of $V$ and $V_{\rho}$ is the $\rho$-isotypic piece of $V$. Then fix $V_{\rho}$ and $s \in X_{\psi}$.  The element $s$ must stabilize $V_{\rho}$ since its action commutes with the one coming from $\psi_p$. In particular, since $s$ is assumed to be semisimple, we have the following claim: we can decompose $V_{\rho}$ as a sum of copies of $\rho$ such that $s$ stabilizes each copy.

Assuming the claim for the moment, we now decompose $V$ as a sum of irreducible $\mathcal{L}_{E_{\mf{p}}} \times SL_2(\mathbb{C})$-representations such that $s$ stabilizes each. Observe that for a fixed copy of an irreducible representation $\rho$, the $s$-action can only have a single eigenvalue since the subspace with eigenvalue $\lambda$ is a $\mathcal{L}_{E_{\mf{p}}} \times \SL_2(\mathbb{C})$-sub module and $\rho$ is irreducible. Then in the Grothendieck group of $\mathcal{L}_{E_{\mf{p}}}$-representations, we have 
\begin{equation*}
\sum\limits_{\nu} \nu(s^{-1}s_{\psi})V(\psi, \nu)_{\lambda}=  \sum\limits_{\rho} \frac{\tr( ss_{\psi} | V_\rho)}{\dim \rho}[\rho \otimes | \cdot |^{-\langle \rho_G, \mu \rangle}].
\end{equation*}

We now prove the claim. We have a group $G$ and a finite dimensional representation $V$ of $G$ which is isomorphic to a sum of copies of a fixed irreducible $G$-representation. We have a semisimple automorphism $s$ of $V$ that commutes with the action of $G$ and we claim that we can decompose $V$ into irreducible $G$ representations in such a way that $s$ fixes each irreducible constituent. Pick a decomposition of $V$ into irreducible $G$-representations. By Schur's lemma, $s$ acts as a diagonalizable matrix on these constituents. One can then diagonalize this matrix to achieve the desired decomposition. 

Then our earlier equation becomes
\begin{align*}
&\sum\limits_{\rho}\sum\limits_{\pi_p \in \Pi_{\psi_p}(\mathbf{G}, \varrho_p)} \frac{1}{|\mathcal{S}_{\psi}|}\sum\limits_{s \in X_{\psi}} \epsilon_{\psi}(\overline{s})\langle \pi, \overline{s} \rangle \langle \lambda_{\pi_{\infty}}, s \rangle \frac{\tr( ss_{\psi} | V_\rho)}{\dim \rho} \pi_p \boxtimes [\rho \otimes | \cdot |^{-\langle \rho_G, \mu \rangle}]\\
&=\sum\limits_{b \in \mathbf{B}(G, -\mu)} \mathrm{Mant}_{G,b,\mu}\left(\sum\limits_{s \in X_{\psi}} \frac{1}{|\mathcal{S}_{\psi}|} \epsilon_{\psi}(\overline{s}) \langle \pi^p, \overline{s} \rangle \langle \lambda_{\pi_{\infty}}, s\rangle
    \mathrm{Red}_b(\psi, ss_{\psi})\right).  
\end{align*}

We now let $X_{\psi_p}$ consist of the set of $s \in \widehat{\mb{G}}$ such that there exists a global parameter $\psi$ whose $p$ component is equivalent to $\psi_p$ and such that $(s, \psi) \in X_{\Psi}$.

We make the following assumption:

\begin{assumption}{\label{liftingassump}}
We assume that we have enough different lifts $\psi$ of $\psi_p$ that we can separate the contributions for different elements of $X_{\psi_p}$ in the above equation.
\end{assumption}

In particular, if the above assumption holds, then for each $x \in X_{\psi_p}$, we expect an equation of the form

\begin{align*}
&\sum\limits_{\rho}\sum\limits_{\pi_p \in \Pi_{\psi_p}(\mathbf{G}, \varrho_p)}  \epsilon_{\psi}(x)\langle \pi, x\rangle\langle \lambda_{\pi_{\infty}}, x \rangle\frac{\tr( xs_{\psi_p} | V_\rho)}{\dim \rho} \pi_p \boxtimes [\rho \otimes | \cdot |^{-\langle \rho_G, \mu \rangle}]\\
&=\sum\limits_{b \in \mathbf{B}(G, -\mu)} \mathrm{Mant}_{G,b,\mu} \left(  \epsilon_{\psi}(x) \langle \pi^p, x \rangle \langle \lambda_{\pi_{\infty}}, x\rangle
    \mathrm{Red}_b(\psi, xs_{\psi_p})\right),
\end{align*}

which simplifies to

\begin{align*}
&\sum\limits_{\rho}\sum\limits_{\pi_p \in \Pi_{\psi_p}(\mathbf{G}, \varrho_p)}  \langle \pi_p, x\rangle \frac{\tr( x | V_\rho)}{\dim \rho} \pi_p \boxtimes [\rho \otimes | \cdot |^{-\langle \rho_G, \mu \rangle}]\\
&=\sum\limits_{b \in \mathbf{B}(G, -\mu)} \mathrm{Mant}_{G,b,\mu} \left(
    \mathrm{Red}_b(\psi_p, x) \right).
\end{align*}

Finally, we rewrite $ \mathrm{Red}_b(\psi_p, x)$ as  $\mathrm{Red}^{\mc{H}^{\mf{e}}}_b(\sum\limits_{\pi^{\mb{H}_p} \in \Pi_{\psi^{\mb{H}}_p}} \langle \pi^{\mb{H}_p}, s^{\mb{H}}_{\psi_p} \rangle \pi^{\mb{H}_p})$ to get our final formula.

\begin{theorem}{\label{finalformula}}
We expect the following formula for the cohomology of unramified PEL-type Rapoport--Zink spaces.
\begin{equation}
\sum\limits_{\rho}\sum\limits_{\pi_p \in \Pi_{\psi_p}(\mathbf{G}, \varrho_p)}  \langle \pi_p, x\rangle \frac{\tr( x | V_\rho)}{\dim \rho} \pi_p \boxtimes [\rho \otimes | \cdot |^{-\langle \rho_G, \mu \rangle}]
\end{equation}
\begin{equation*}
=\sum\limits_{b \in \mathbf{B}(G, -\mu)} \mathrm{Mant}_{G,b,\mu}(
    \mathrm{Red}^{\mc{H}^{\mf{e}}}_b(\sum\limits_{\pi^{\mb{H}_p} \in \Pi_{\psi^{\mb{H}}_p}} \langle \pi^{\mb{H}_p}, s^{\mb{H}}_{\psi_p} \rangle \pi^{\mb{H}_p})).
\end{equation*}
\end{theorem}

\section{Endoscopic cocharacter data}

In this section we sketch the theory of endoscopic cocharacter data in analogy with the theory of cocharacter pairs of \cite{abm1}.

We use the notation of $\S{2}$ of \cite{abm1}. In particular, let $G$ be a quasisplit connected reductive group over $\mathbb{Q}_p$. Fix a maximal torus $T$ with maximal split rank and a Borel subgroup $B$ so that $T \subset B$. Let $A$ be the maximal split subtorus of $T$. We let $\Delta \subset X^*(A)$ be the set of relative simple roots determined by the above data. There is an order-preserving bijection $S \mapsto M_S$ between subsets of $\Delta$ and standard Levi subgroups of $G$. The set of cocharacter pairs $\mathcal{C}_G$ consists of pairs $(M_S, \mu_S)$ where $M_S$ is a standard Levi subgroup and $\mu_S \in X_*(T)$ is an $M_S$-dominant cocharacter. 

\begin{definition}
An endoscopic cocharacter datum for $G$ is a tuple $(\mathcal{H}^{\mathfrak{e}}_S, M_S, \mu_S)$ where $(M_S, \mu_S) \in \mathcal{C}_G$ and $\mathcal{H}^{\mathfrak{e}}_S$ is an equivalence class of embedded endoscopic data for $M_S$ relative to $G$. We denote the set of endoscopic cocharacter data for $G$ by $\mc{C}^{\mf{e}}_G$.

\end{definition}
Note there is a natural embedding $\mathcal{C}_G \hookrightarrow \mathcal{C}^{\mathfrak{e}}_G$ whereby we attach the class of the trivial embedded endoscopic datum $(G, M_S, 1, \mathrm{id})$ to $(M_S, \mu_S)$. There is also a natural projection $\mathcal{C}^{\mathfrak{e}}_G \to \mathcal{C}_G$ given by forgetting the extra data.

Recall from \cite[Definition 2.1.4, 2.1.5]{abm1}, the map $\theta_{M_S}: X_*(T) \to \mathfrak{A}_{\mathbb{Q}}$ and poset structure on $\mathcal{C}_G$. We now define a poset structure on $\mathcal{C}^{\mathfrak{e}}_{G}$. We first define for $S \subset S'$ a map 
\begin{equation*}
Y^e_{M_S, M_{S'},G}: \mc{E}^{e}(M_S, G) \to \mc{E}^{e}(M_{S'}, G),
\end{equation*}
as the unique map such that the following diagram commutes (where the horizontal maps are bijections and were defined in Proposition \ref{refemb})
\begin{equation*}
\begin{tikzcd}
\mc{E}^e(M_{S'}, G)  \arrow[r, "X"] & \mc{E}^r(M_{S'}) \\
\mc{E}^e(M_{S}, G) \arrow[u, "{Y^e_{M_S, M_{S'}, G}}"] \arrow[r, "X"] & \mc{E}^r(M_S) \arrow[u, swap, "Y"].
\end{tikzcd}    
\end{equation*}

We then define a partial order on $\mc{C}^{\mf{e}}_G$ by  stipulating that $(\mathcal{H}^{\mathfrak{e}}_S, M_S, \mu_S) \leq (\mathcal{H}^{\mathfrak{e}}_{S'}, M_{S'}, \mu_{S'})$ exactly when $(M_S, \mu_S) \leq (M_{S'}, \mu_{S'})$ and $Y^e_{M_S, M_{S'},G}(\mc{H}^{\mf{e}}_S)=\mc{H}^{\mf{e}}_{S'}$.

\begin{definition}
We define the subset $\mathcal{SD}^{\mf{e}} \subset \mathcal{C}^{\mathfrak{e}}_G$ as the set of endoscopic cocharacter data $(\mathcal{H}^{\mathfrak{e}}_S, M_S, \mu_S)$ such that $(M_S, \mu_S) \in \mathcal{SD}$. We define $\mathcal{SD}^{\mathfrak{e}}_{(\mathcal{H}^{\mathfrak{e}}, \mu)}$ in analogy with \cite[Definition 2.1.7]{abm1} to be the elements $(\mc{H}^{\mf{e}}_S, M_S, \mu_S) \in \mc{SD}$ such that $(\mc{H}^{\mf{e}}_S, M_S, \mu_S) \leq (\mc{H}^{\mf{e}}, G, \mu)$
\end{definition}

The $\mathcal{C}^{\mathfrak{e}}_G$-analogues of the results of \S2.3 of \cite{abm1} go through essentially unchanged.  We remark that the ``extension'' (as in \cite[Proposition 2.3.8]{abm1}) of $(\mc{H}^{\mf{e}}_S, M_S, \mu_S)$ to $M_{S'}$ is defined as $(Y^e_{M_S, M_{S'}, G}(\mc{H}^{\mf{e}}_S), M_{S'}, \mu_{S'})$ where $(M_{S'}, \mu_{S'})$ is the extension of $(M_S, \mu_S)$.  Extending these arguments requires the following obvious transitivity property of $Y^e$.

\begin{lemma}
Suppose $M_{S_3} \subset M_{S_2} \subset M_{S_1}$ are standard Levi subgroups. Then $Y^e_{M_{S_2}, M_{S_1},G} \circ Y^e_{M_{S_3}, M_{S_2},G} =Y^e_{M_{S_3}, M_{S_1},G}$.
\end{lemma}
\begin{proof}
\end{proof}

In \cite[Proposition 2.4.3]{abm1}, we defined a map 
\begin{equation*}
    \mc{T}: \mc{SD} \to \mb{B}(\Q_p, G),
\end{equation*}
such that $\mc{T}(\mc{SD}_{\mu}) \subset \mb{B}(\Q_p, G, \mu)$. We pre-compose this map with the projection $\mc{C}^{\mf{e}}_G \to \mc{C}_G$ to get a map 
\begin{equation*}
    \mathcal{T}^{\mathfrak{e}}: \mathcal{SD}^{\mathfrak{e}} \to \mathbf{B}(\Q_p, G),
\end{equation*}
such that $\mathcal{T}^{\mathfrak{e}}(\mathcal{SD}^{\mathfrak{e}}_{\mathcal{H}^{\mathfrak{e}}, \mu}) \subset \mathbf{B}(\Q_p, G, \mu)$.


Finally, we can define $\mathcal{T}^{\mathfrak{e}}_{G,\mathcal{H}^{\mathfrak{e}},b, \mu}, \mathcal{R}_{G, \mathcal{H}^{\mathfrak{e}}, b, \mu}, \mathcal{M}_{G, \mathcal{H}^{\mathfrak{e}}, b, \mu}$ in analogy with \cite[Definitions 2.5.1, 2.5.2]{abm1}. 

In particular:
\begin{itemize}
    \item $\mc{T}_{G, \mc{H}^{\mf{e}}, b, \mu} := \mc{T^{\mf{e}}}^{-1}(b) \cap \mc{SD}^{\mf{e}}_{\mc{H}^{\mf{e}}, \mu}$,
    \item 
    \begin{equation*}
        \mc{R}_{G, \mc{H}^{\mf{e}}, b, \mu} :=
    \end{equation*}
    \begin{equation*}
          \{ (\mc{H}^{\mf{e}}_{S_1} , M_{S_1}, \mu_{S_1}): (\mc{H}^{\mf{e}}_{S_1} , M_{S_1}, \mu_{S_1}) \leq  (\mc{H}^{\mf{e}}_{S_2}, M_{S_2}, \mu_{S_2}) \text{ where } (\mc{H}^{\mf{e}}_{S_2}, M_{S_2}, \mu_{S_2}) \in \mc{T}_{G,\mc{H}^{\mf{e}}, b, \mu}   \},
    \end{equation*}
    \item $\mc{M}_{G, \mc{H}^{\mf{e}}, b, \mu} := \sum\limits_{(\mc{H}^{\mf{e}}_S , M_S, \mu_S) \in \mc{R}_{\mc{H}^{\mf{e}}, G, \mu}} (-1)^{L_{M_S, M_b}} (\mc{H}^{\mf{e}}_S, M_S, \mu_S)$.
\end{itemize}
We can prove the following endoscopic sum formula:
\begin{theorem}{\label{endsum}}
We have
\[
\sum\limits_{b \in \mathbf{B}(\Q_p, G, \mu)} \mathcal{M}_{G, \mathcal{H}^{\mathfrak{e}},b, \mu} = (\mathcal{H}^{\mathfrak{e}}, G, \mu),
\]
which is an equality in $\mathbb{Z} \langle \mathcal{C}^{\mathfrak{e}}_{G} \rangle$.
\end{theorem}
\begin{proof}
The proof is analogous to that of \cite[Theorem 2.5.4]{abm1}.
\end{proof}

We now describe how to generalize the induction formula (\cite[Corollary 2.5.8]{abm1}). Fix standard Levi subgroups $M_{S_2} \subset M_{S_1} \subset G$. By Proposition \ref{refemb}, we have an isomorphism $Z^e_{M_{S_2}, M_{S_1}, G}: \mc{E}^e(M_{S_2}, M_{S_1}) \to \mc{E}^e(M_{S_2}, G)$ defined such that the following diagram commutes:
\begin{equation*}
\begin{tikzcd}
 \mc{E}^e(M_{S_2} , G)  \arrow[r, "X"] & \mc{E}^r(M_{S_2})\\
 \mc{E}^e(M_{S_2}, M_{S_1}) \arrow[r, "X"] \arrow[u, "{Z^e_{M_{S_2}, M_{S_1}, G}}"]& \mc{E}^r(M_{S_2}) \arrow[u, equals]
\end{tikzcd}    
\end{equation*}
In particular, we get a natural injection $i^G_{M_{S_1}}: \mc{C}^e_{M_{S_1}} \hookrightarrow \mc{C}^e_G$. 

As in the paragraph before \cite[Definition 2.5.5]{abm1}, for each $b \in \mb{B}(\Q_p, G)$ such that $M_b \subset M_{S_1}$,  we have a canonical element $b_{S_1}$ of $\mb{B}(\Q_p, M_{S_1})$ that maps to $b$ under the natural map $\mb{B}(\Q_p, M_{S_1}) \to \mb{B}(\Q_p, G)$.

\begin{definition}
For a fixed $(\mc{H}^{\mf{e}}, G, \mu) \in \mc{C}^{\mf{e}}_G$ and $b \in \mb{B}(\Q_p, G, \mu)$ and standard Levi subgroup $M_S$ such that $M_b \subset M_S$, we define the set $\mc{I}^{\mc{H}^{\mf{e}}, G, \mu}_{M_S, b_S}$ to equal
\begin{equation*}
     \{ (\mc{H}^{\mf{e}}_S, M_S, \mu_S) : b_S \in \mb{B}(\Q_p, M_S, \mu_S), \mu_S \sim_G \mu, Z^e_{M_S, M_S, G}(\mc{H}^{\mf{e}}_S)= \mc{H}^{\mf{e}}\}.
\end{equation*}
\end{definition}

We now have the following proposition (cf \cite[Proposition 2.5.6]{abm1})
\begin{proposition}
Fix $(\mc{H}^{\mf{e}}, G, \mu) \in \mathcal{C}^{\mf{e}}_G$ and $b \in \mathbf{B}(\Q_p, G, \mu)$. Suppose $M_{S_2}$ and $M_{S_1}$ are standard Levi subgroups of $G$ such that $M_b \subset M_{S_2} \subset M_{S_1}$. Then
\begin{equation*}
    \mathcal{I}^{\mc{H}^{\mf{e}}, G, \mu}_{M_{S_2}, b_{S_2}} =
\end{equation*}
\begin{equation*}
    \{ (\mc{H}^{\mf{e}}_{S_2}, M_{S_2}, \mu_{S_2}) \in \mathcal{C}^{\mf{e}}_{M_{S_2}} : (\mc{H}^{\mf{e}}_{S_2}, M_{S_2}, \mu_{S_2}) \in \mathcal{I}^{\mc{H}^{\mf{e}}_{S_1}, M_{S_1},  \mu_{S_1}}_{M_{S_2}, b_{S_2}} \,\ \text{for a} \,\ (\mc{H}^{\mf{e}}_{S_1}, M_{S_1}, \mu_{S_1}) \in \mathcal{I}^{\mc{H}^{\mf{e}}, G, \mu}_{M_{S_1}, b_{S_1}} \}.
\end{equation*}
\end{proposition}
\begin{proof}
Follows from Proposition \cite[2.5.6]{abm1} and the transitivity of $Z^{\mf{e}}$.
\end{proof}
Now we have the analogue of \cite[Proposition 2.5.7]{abm1}.
\begin{proposition}
Fix $M_S, \mu, b$ as before. The map $i^G_{M_S}: \mc{C}^{\mf{e}}_{M_S} \to \mc{C}^{\mf{e}}_G$ induces a bijection
\begin{equation*}
    \coprod\limits_{(\mc{H}^{\mf{e}}_S, M_S, \mu_S) \in \mc{I}^{\mc{H}^{\mf{e}}, G, \mu}_{M_S, b_S}} \mc{T}_{\mc{H}^{\mf{e}}_S, M_S, b_S, \mu_S} \cong  \mc{T}_{\mc{H}^{\mf{e}}, G, b, \mu}
\end{equation*}
\end{proposition}
\begin{proof}
Same as \cite[Proposition 2.5.7]{abm1}.
\end{proof}

The map $i^G_{M_S}$ induces a map:
\begin{equation*}
    i^G_{M_S} : \Z \langle C^{\mf{e}}_{M_S}\rangle  \hookrightarrow \Z \langle  C^{\mf{e}}_{G} \rangle,
\end{equation*}
on the free abelian groups generated by the sets of endoscopic cocharacter data.

\begin{corollary}
The inclusion $i^G_{M_S}$ induces an equality
\begin{equation*}
   i^G_{M_S} \left( \sum\limits_{(\mc{H}^{\mf{e}}_S, M_S, \mu_S) \in \mc{I}^{\mc{H}^{\mf{e}}, G, \mu}_{M_S, b_S}} \mc{M}_{\mc{H}^{\mf{e}}_S, M_S, b_S, \mu_S} \right) = \mc{M}_{\mc{H}^{\mf{e}}, G, b, \mu}
\end{equation*}
in $\Z \langle \mc{C}^{\mf{e}}_G \rangle$.
\end{corollary}
\begin{proof}
Analogous to \cite[Corollary 2.5.8]{abm1}.
\end{proof}

We now relate the above to the cohomology of Rapoport--Zink spaces. Our first step is to associate to elements of $\mc{C}^{\mf{e}}_G$ certain maps of Grothendieck groups. Pick $(\mc{H}^{\mf{e}}_S, M_S, \mu_S)$ and an embedded endoscopic datum $(H, H_{M_S}, s, \eta)$ representing $\mc{H}^{\mf{e}}_S$. We define
\begin{equation*}
    [\mc{H}^{\mf{e}}_S, M_S, \mu_S]: \Groth^{st}(H(\Q_p)) \to \Groth(G(\Q_p) \times W_{E_{\{\mu_S\}_{M_S}}}),
\end{equation*}
as follows. We first use the theory of \S 2 to pick a set $X_{\mc{H}^{\mf{e}}_S}$ of embedded endoscopic data for the various inner classes whose embedded class equals $\mc{H}^{\mf{e}}_S$. We can pick these representatives to all have the same group $H$ as the endoscopic group of $G$ and such that each $H_{M_S}$ is a standard Levi subgroup with respect to a fixed pair $(B_H, T_H)$. If $P_S$ is the standard parabolic subgroup of $M_S$, then we get a parabolic subgroup $P^{\mf{e}}_S$ of $H_{M_S}$ compatible with $P_S$ via $\eta$.

We now define
\begin{equation*}
    [\mc{H}^{\mf{e}}, M_S, \mu_S](\pi^{st}) =\sum\limits_{M^{\mf{e}}_S \in X_{\mc{H}^{\mf{e}}_S}} (\Ind^G_{P_S} \circ [\mu_S]_{M^{\mf{e}}_S}  \circ \delta^{1/2}_{P^{\mf{e}}_S} \otimes \Jac^H_{P^{\mf{e}, op}_S})(\pi^{st}) \otimes [1][|\cdot|^{\langle \rho_G, \mu_S - \mu \rangle}],
\end{equation*}
where
\begin{equation*}
    [\mu_S]_{M^{\mf{e}}_S}: \Groth^{st}(H_{M_S}(\Q_p)) \to \Groth(M_S(\Q_p) \times W_{E_{\{\mu_S\}_{M_S}}})
\end{equation*}
is given by a composition of the transfer map
\begin{equation*}
 \Trans^{H_{M_S}}_{M_S} : \Groth^{st}(H_{M_S}(\Q_p)) \to \Groth(M_S(\Q_p))
\end{equation*}
and the ``weighted local Langlands map''
\begin{equation*}
    LL_S: \Groth(M_S(\Q_p)) \to \Groth(M_S(\Q_p) \times W_{E_{\{\mu_S\}_{M_S}}})
\end{equation*}
defined on tempered $\pi$ by
\begin{equation*}
    \pi \mapsto \langle \pi, \eta(s) \rangle \pi \boxtimes \sum\limits_{\rho \in \Irr(r_{- \mu} \circ \psi_{\pi})} \frac{ \tr( \eta(s) | V_{\rho})  }{\dim \rho} [\rho \boxtimes | \cdot |^{- \langle \rho_G, \mu \rangle}].
\end{equation*}
Note that this last construction requires $\pi$ to be tempered (or at least have an associated Arthur parameter).

We now use the above definition to produce representation-theoretic analogues of the combinatorial sum and induction formulas. For the sum formula, this is clear: we replace each endoscopic cocharacter pair by its corresponding map of representations and use \cite[Lemma 3.3.5]{abm1} to consider $\sum\limits_{b \in \mb{B}(G ,\mu)} \mc{M}_{\mc{H}^{\mf{e}}, G, b, \mu} =(\mc{H}^{\mf{e}}, G, \mu)$ as an equality in $\Groth(G(\Q_p) \times W_{E_{\{\mu\}_G}})$.

We now tackle the induction formula. Fix standard Levi subgroups $M_{S_2} \subset M_{S_1} \subset G$. We recall that there is a bijection between isomorphism classes of embedded data $\mc{E}^e(M_{S_2}, G)$ and $\mc{E}^e(M_{S_2}, M_{S_1})$ since both correspond bijectively to $\mc{E}^r(M_{S_2})$. On the other hand, the number of inner classes in each embedded class need not be the same. In fact, by the comment immediately following Definition \ref{embdat}, the number of inner classes in a given embedded class in $\mc{E}^e(M_{S_2}, M_{S_1})$ will be a multiple of the number of inner classes of the corresponding embedded class in $\mc{E}^{e}(M_{S_2}, G)$.

By Proposition \ref{fibercard}, the number of inner classes in a $\mc{E}^e(M_{S_2}, G)$ (respectively $\mc{E}^e(M_{S_2}, M_{S_1})$) class is $|\Out_r(H,s,\eta)/\Out_r(H_{M_{S_2}} , s, \eta)|$\\ (respectively $|\Out_r(H_{M_{S_1}},s,\eta)/\Out_r(H_{M_{S_2}} , s, \eta)|$). Hence the ratio of the inner classes is $|\Out_r(H,s,\eta)/\Out_r(H_{M_{S_1}} , s, \eta)|$. 

With the above paragraph in mind, let $(\mc{H}^{\mf{e}}_{S_2}, M_{S_2}, \mu_{S_2}) \in \mc{C}^{\mf{e}}_{M_{S_1}}$. Then we get an element $i^G_{M_{S_1}}(\mc{H}^{\mf{e}}_{S_2}, M_{S_2}, \mu_{S_2}) \in \mc{C}^{\mf{e}}_G$ and the corresponding maps of representations are related by
\begin{equation*}
    [i^G_{M_{S_1}}(\mc{H}^{\mf{e}}_{S_2}, M_{S_2}, \mu_{S_2})] = \left|\frac{\Out_r(H_{M_{S_1}} , s, \eta)}{\Out_r(H, s, \eta)} \right|  \Ind^G_{P_{S_1}} \circ [\mc{H}^{\mf{e}}_{S_2}, M_{S_2}, \mu_{S_2}] \circ (\delta_{P^{\mf{e}, op}_{S_1}} \otimes \Jac^H_{P^{\mf{e}, op}_{S_1}}) \otimes | \cdot |^{ \langle \rho_G, \mu_{S_1} - \mu \rangle }.
\end{equation*}

We now define a map $\iota^G_{M_{S_1}}: (\Groth^{st}(H_{M_{S_1}}(\Q_p)) \to \Groth(M_{S_1}(\Q_p) \times W_{E_{\{\mu_{S_2}\}_{M_{S_2}}}}) \to (\Groth^{st}(H(\Q_p)) \to \Groth(G(\Q_p) \times W_{E_{\{\mu_{S_2}\}_{M_{S_2}}}}) $ given by 

\begin{equation*}
    \iota^G_{M_{S_1}} [\mc{H}^{\mf{e}}_{S_2}, M_{S_2}, \mu_{S_2}] = \left|\frac{\Out_r(H_{M_{S_1}} , s, \eta)}{\Out_r(H, s, \eta)} \right|[i^G_{M_{S_1}} (\mc{H}^{\mf{e}}_{S_2}, M_{S_2}, \mu_{S_2})]. 
\end{equation*}

Then we get 

\begin{equation*}
    \sum\limits_{(\mc{H}^{\mf{e}}_S, M_S, \mu_S) \in \mc{I}^{\mc{H}^{\mf{e}}, G, \mu}_{M_S, b_S}}  \left|\frac{\Out_r(H, s, \eta)}{\Out_r(H_{M_{S}} , s, \eta)} \right|\iota^G_{M_S}([\mc{M}_{\mc{H}^{\mf{e}}_S, M_S, b_S, \mu_S}])=[\mc{M}_{\mc{H}^{\mf{e}}, G, b, \mu}].
\end{equation*}

Finally, we show how to relate this to the cohomology of Rapoport--Zink spaces. Clearly, the endoscopic sum formula and our endoscopic averaging formula are analogous. We claim that the induction formula is analogous to the Harris-Viehmann conjecture. To see this, we need to show how $\Red^{\mc{H}^{\mf{e}}}_b$ relates to $\Red^{\mc{H}^{\mf{e}}_S}_b$ for $M_b \subset M_S \subset G$. To do this, we need to show how $X^{\mf{e}}_{M_b}$ relative to $G$ relates to $X^{\mf{e}}_{M_b}$ relative to $M_S$. The answer is that the term $(H, H_M, s, \eta)$ appears $|\Out_r(H,s,\eta)/\Out_r(H_M , s, \eta)|$ many more times in the second set than the first. Note that we can use $X^{\mf{e}}_{M_b}$ instead of $X^{\mf{e}}_{J_b}$ since the extra terms in the $M_b$ set have trivial endoscopic transfer. Then the Harris-Viehmann conjecture \cite[Conjecture 3.2.1]{abm1} implies the formula:
\begin{equation*}
    \Mant_{G,b,\mu} \circ \Red^{\mc{H}^{\mf{e}}}_b = \sum\limits_{(\mc{H}^{\mf{e}}_S, M_S, \mu_S) \in \mc{I}^{\mc{H}^{\mf{e}}, G, \mu}_{M_S, b_S}} \left|\frac{\Out_r(H, s, \eta)}{\Out_r(H_{M_{S}} , s, \eta)} \right|  \Ind^G_{P_S} \left( \Mant_{M_S, b_S, \mu_S} \circ \Red^{\mc{H}^{\mf{e}}_S}_b \right) \otimes | \cdot |^{ \langle \rho_G, \mu_S - \mu \rangle },
\end{equation*}
which we note is entirely analogous to our above combinatorial formula. Hence, arguing by induction as in \cite[\S3.3]{abm1}, one expects the following final formula:
\begin{conjecture}{\label{RZcon}}
We have
\begin{equation*}
    \Mant_{G,b,\mu} \circ \Red^{\mc{H}^{\mf{e}}}_b = [\mc{M}_{\mc{H}^{\mf{e}}, G, b, \mu}],
\end{equation*}
in $\Groth(G(\Q_p) \times W_{E_{\{\mu_G\}}})$.
\end{conjecture}

\appendix

\section{Endoscopic triples and endoscopic quadruples}{\label{appendixA}}
In this appendix, we explain the equivalence between our definition of endoscopic triple and the definition of endoscopic datum appearing in \cite[\S 2.1]{KS}. We point out that the results of this section are due to unpublished work of Kottwitz as well as \cite{KS}.

\begin{definition}{\cite[\S 2.1]{KS}}{\label{KSend}}
An \emph{endoscopic datum for $G$} is a quadruple $(H, \mathcal{H}, s, \eta)$ such that
\begin{itemize}
    \item $H$ is a quasisplit reductive group over $F$,\\
    \item $\mathcal{H}$ fits into a split exact sequence
    \begin{equation} 0 \to \widehat{H} \to \mathcal{H} \to W_F \to 0  \end{equation} such that the map $\rho_{\mathcal{H}}: W_F \to \mathrm{Out}(\widehat{H})$ agrees with $\rho_H$.\\
    \item $\eta: \mathcal{H} \hookrightarrow \, ^LG$ is an $L$-homomorphism,\\
    \item $s$ is a semisimple element of $\widehat{G}$ such that $\eta(\widehat{H})=Z_{\widehat{G}}(s)^0$,\\
    \item We have $\mathrm{Int}(s) \circ \eta= a \cdot \eta$ where $a: W_F \to Z(\widehat{G})$ is a $1$-cocycle that is trivial if $F$ is local or locally trivial if $F$ is global.
\end{itemize}
An isomorphism between endoscopic data $(H, \mathcal{H}, s, \eta)$ and $(H', \mathcal{H}', s', \eta')$ is an element $g \in \widehat{G}$ such that 
\begin{itemize}
\item we have $g \eta(\mathcal{H})g^{-1}=\eta'(\mathcal{H}')$,\\
\item and $gsg^{-1}=s' \, \mathrm{mod} Z(\widehat{G})$.
\end{itemize}
We define $\mathrm{Out}(H, \mathcal{H}, s, \eta)$ to be $\mathrm{Aut}(H, \mathcal{H}, s, \eta)/\eta( \widehat{H})$.
\end{definition}
\begin{construction}
We claim first that we can construct in a natural way an endoscopic triple as in \ref{endtrip} from an endoscopic datum as in \ref{KSend}.
\end{construction}
\begin{proof}
Given an endoscopic datum $(H, \mathcal{H}, s, \eta)$ we get a triple $(H, \eta^{-1}(s), \eta|_{\widehat{H}})$. We must show this is indeed an endoscopic triple. We first check that the $\widehat{G}$-conjugacy class of $\eta$ is fixed by $\Gamma_F$. Fix $\gamma \in \Gamma_F$. We need to show that there exists $g \in \widehat{G}$ such that 
\begin{equation}
    \gamma \circ \eta \circ \gamma^{-1}=g\eta g^{-1}.
\end{equation}
There exists a finite extension $K$ of $F$ such that the  actions of $\Gamma_F$ on $\widehat{H}$ and $\widehat{G}$ factor through $\mathrm{Gal}(K/F)$. Hence it suffices to assume $\gamma \in \mathrm{Gal}(K/F)$. Pick a lift $w \in W_F$ of $\gamma$. We claim there exist  $g_w \in \, ^LG$ and $h_w \in \, \mathcal{H}$ such that $\mathrm{Int}(g_w)|_{\widehat{G}}$ equals the action of $\gamma$ on $\widehat{G}$ and $\mathrm{Int}(h_w)|_{\widehat{H}}$ equals the action of $\gamma$ on $\widehat{H}$ and both elements project to $w \in W_F$. Assuming the claim, we get 
\begin{equation}
    \gamma \circ \eta \circ \gamma^{-1}=\mathrm{Int}(g_w\eta(h_w)^{-1}) \circ \eta,
\end{equation}
and we observe that $g_w\eta(h_w)^{-1} \in \widehat{G}$ as desired.

For $^LG$, the element $(1, w)$ satisfies the claim. For $\mathcal{H}$, we pick a splitting $c: W_F \to \mathcal{H}$ and observe that by the second bullet in \ref{KSend}, the action of $c(w)$ on $\widehat{H}$ differs from that of $\gamma$ by an inner automorphism of $\widehat{H}$. This proves the claim.

Finally we need to check that the image of $\eta^{-1}(s)$ in $Z(
\widehat{H})/Z(\widehat{G})$ is $\Gamma_F$-invariant and that the image of $\eta^{-1}(s)$ in $H^1(F, Z(\widehat{G}))$ is (locally) trivial. We denote $\eta^{-1}(s)$ by $s_H$ for convenience. Then as before, pick $\gamma \in \mathrm{Gal}(K/F)$ and choose $h \in \mathcal{H}$ such that $\mathrm{Int}(h)$ acts by $\gamma$ on $\widehat{H}$. In particular $s_Hhs^{-1}_H=s_H\gamma(s_H)^{-1}h$. Then the fifth bullet in \ref{KSend} gives 
\begin{equation}
    \eta(s_H\gamma(s_H)^{-1}h)=\eta(s_Hhs^{-1}_H)=\eta(h) \mathrm{mod} Z(\widehat{G}),
\end{equation}
which implies $s_H=\gamma(s_H) \mathrm{mod} Z(\widehat{G})$ as desired. Moreover, the above equation implies that the $W_F$-cocycle valued in $Z(\widehat{G})$ given by $s_H\gamma(s_H)^{-1}$ is equal to $a$ which is (locally) trivial by assumption.
\end{proof}
\begin{lemma}
The above construction induces a map of isomorphism classes of endoscopic data to isomorphism classes of triples.
\end{lemma}
\begin{proof}
Suppose we have a $g \in \widehat{G}$ inducing an isomorphism between $(H, \mathcal{H}, s, \eta)$ and $(H', \mathcal{H}', s', \eta')$. Then $\mathrm{Int}(g)$ induces an isomorphism $\beta: \widehat{H} \to \widehat{H}$. We claim that this isomorphism comes from an isomorphism $\alpha: H \to H'$ defined over $F$.

To prove the claim, we first show that $\beta$ is $\Gamma_F$-equivariant up to conjugacy and hence induces an equivariant map of canonical based root data. As before, pick $\gamma \in \mathrm{Gal}(K/F)$ and let $h_{\gamma} \in \mathcal{H}$ be such that $\mathrm{Int}(h_{\gamma})|_{\widehat{H}}$ and $\gamma$ coincide on $\widehat{H}$. Then for any $h\in \widehat{H}$, we have 
\begin{equation}
    \beta(\gamma(h))=\mathrm{Int}(\beta(h_{\gamma})) \circ \beta(h).
\end{equation}
Now, the action of $\mathrm{Int}(\beta(h_{\gamma}))$ on $\widehat{H'}$ agrees with that of $\gamma$ up to an inner automorphism as desired.

Now, we pick splittings of $H$ and $H'$ defined over $F$  and let $\alpha: H' \to H$ be the unique isomorphism dual to $\beta$ and mapping the splitting of $H'$ to that of $H$. Then by uniqueness, $\alpha$ is defined over $F$.

It is clear that $\alpha$ induces an isomorphism of triples $(H, s_H, \eta), (H', s'_H, \eta')$.
\end{proof}
Conversely, we must show that given an endoscopic triple $(H, s, \eta)$ we can construct an endoscopic datum.
\begin{construction}{\label{triptodat}}
Let $(H,s, \eta)$ be an endoscopic triple of $G$. We construct an endoscopic datum $(H', \mathcal{H}', s', \eta')$ from $(H, s, \eta)$ in a natural way. Moreover, this construction gives a map of isomorphism classes of the respective data.
\end{construction}
\begin{proof}
Define the following subgroup $\mathcal{H}$ of $^LG$. Let $x \in \, ^LG$ and suppose $x$ projects to $w \in W_F$. Then $x \in \mathcal{H}$ precisely if there exists a $y \in \, ^LH$ that also projects to $w$ and such that
\begin{equation}
    \mathrm{Int}(x) \circ \eta= \eta \circ \mathrm{Int}(y).
\end{equation}

We claim that $\mathcal{H}$ is an extension of $W_F$ by $\eta(\widehat{H})$. Pick $w \in W_F$ such that $w$ projects to $\gamma \in \mathrm{Gal}(K/F)$. Then $(1, w) \in \, ^LH$ acts on $\widehat{H}$ by $\gamma$. By definition, there exists a $g_{\gamma} \in \widehat{G}$ such that $\mathrm{Int}(g_{\gamma}) \circ \eta=\gamma \cdot \eta$. Then
\begin{equation}
    \eta \circ \Int(1, w) = \mathrm{Int}(\gamma^{-1}(g_{\gamma}), w) \circ \eta,
\end{equation}
which implies $\mathcal{H}$ surjects onto $W_F$.

Now the kernel of $\mathcal{H} \to W_F$ consists of $x \in \widehat{G}$ such that there exists $y \in \widehat{H}$ and $\mathrm{Int}(x) \circ \eta= \eta \circ \mathrm{Int}(y)$. Clearly $\eta(\widehat{H})$ is contained in this set. Conversely, we have that $\mathrm{Int}(x^{-1} \eta(y))$ acts trivially on $\widehat{H}$. In particular, $x^{-1}\eta(y)$ must centralize a maximal torus $\widehat{T}_H$ of $\eta(\widehat{H})$. Then $\widehat{T}_H$ is maximal in $\widehat{G}$ as well so $x^{-1} \eta(y) \in \widehat{T}_H \subset \eta(\widehat{H})$. Hence $x \in \eta( \widehat{H})$.

We now prove that the extension
\begin{equation}
1 \to \eta(\widehat{H}) \to \mathcal{H} \to W_F \to 1
\end{equation}
is split. We proceed as follows. Let $\widehat{T} \subset \widehat{B}$ be maximal torus and Borel of $\widehat{H}$ and let $\mathcal{T}$ be the subgroup of $\mathcal{H}$ of elements preserving the pair $(\eta(\widehat{T}), \eta(\widehat{B}))$. Then $\mathcal{T}$ is an extension of $W_F$ by $\eta(\widehat{T})$.

Then \cite[Lemma 4]{Lan1} says that if there exists a field $K$ that is a finite Galois extension of $F$ such that the action of $W_F$ on $\widehat{T}$ factors through $\mathrm{Gal}(K/F)$, then $\mathcal{T}$ is split. Since this is the case, $\mathcal{T}$ is split so we can take a splitting $c: W_F \to \mathcal{T}$. Then this is also a splitting of $\mathcal{H}$.

We define $(H', \mathcal{H}', s', \eta')$ so that $H'=H$ and $\mathcal{H}'$ is defined to be equal to $\mathcal{H}$ and $\eta'$ is the natural embedding $\mathcal{H} \subset \, ^LG$. Finally, let $s'$ equal $\eta(s)$. We need to check that $\rho_{\mathcal{H}'}: W_F \to \mathrm{Out}(\widehat{H}')$ agrees with $\rho_{H'}$ and that $\mathrm{Int}(s') \circ \eta'=a\cdot \eta'$ where $a: W_F \to Z(\widehat{G})$ is a (locally) trivial cocycle. 

For the first property, we pick $w \in W_F$. Then we need to show that the action of $\mathrm{Int}(c(w))$ on $\eta'(\widehat{H})$ is equal up to conjugation by an element of $\eta(\widehat{H})$ to the action given by $\eta(h) \mapsto \eta(w(h))$. By definition, there exists $(y,w) \in {}^LH$ so that 
\begin{equation}
   \mathrm{Int}(c(w)) \circ \eta(h)=\eta(yw(h)y^{-1}),
\end{equation}
hence the actions differ by  conjugation by $\eta(y)$ as desired.

For the second property, we need to show that $(1,w) \mapsto s' \eta'(1,w)s'^{-1}\eta'(1,w)^{-1}$ is a cocycle valued in $Z(\widehat{G})$. By definition, there exists $(y,w) \in \, ^LH$ such that the above equals
\begin{equation}
    \eta(s(y,w)s^{-1}(y,w)^{-1})=\eta(syw(s^{-1})y^{-1})=\eta(sys^{-1}zy^{-1})=\eta(z)
\end{equation}
for some $z \in Z(\widehat{G})$. To check the cocycle condition, we observe that this map is given by $w \mapsto s w(s^{-1})$.

Finally, we check that this construction gives a map of isomorphism classes. Suppose $(H_1, s_1, \eta_1), (H_2, s_2, \eta_2)$ are isomorphic endoscopic triples and $\alpha: H_1 \to H_2$ gives an isomorphism between them. Then there exists $g \in \widehat{G}$ so that $\eta_1 \circ \widehat{\alpha}= \mathrm{Int}(g) \circ \eta_2$. Now pick some $(x,w) \in \mathcal{H}_1$. Then there exists $(y,w) \in \, ^LH$ such that 
\begin{equation}
    \mathrm{Int}(x,w) \circ \eta_1 = \eta_1 \circ \mathrm{Int}(y,w).
\end{equation}
Now we post-compose with $\widehat{\alpha}$ and use that $\widehat{\alpha}$ extends naturally to a map $^LH_2 \to \, ^LH_1$ by Lemma \ref{alphalem}. Let $(y', w)= \widehat{\alpha}^{-1}(y,w)$. Then we get 
\begin{equation}
    \mathrm{Int}(x,w) \circ \mathrm{Int}(g) \circ \eta_2= \mathrm{Int}(x,w) \circ \eta_1 \circ \widehat{\alpha}=\eta_1 \circ \mathrm{Int}(y,w) \circ \widehat{\alpha}=\eta_1 \circ \widehat{\alpha} \circ \mathrm{Int}(y',w)= \mathrm{Int}(g) \circ \eta_2 \circ \mathrm{Int}(y',w).
\end{equation}
This implies that $g^{-1}$ conjugates $\mathcal{H}_1$ to $\mathcal{H}_2$. To get the desired isomorphism, we just need to check that $g^{-1}\eta_1(s_1)g=\eta_2(s_2) \mathrm{mod} Z(\widehat{G})$. But modulo $Z(\widehat{G})$, we have $\widehat{\alpha}(s_2)=s_1$ so $g \eta_2(s_2)g^{-1}=\eta_1(s_1)$ modulo $Z(\widehat{G})$ as desired.
\end{proof}

Finally, it remains to check that the two constructions are inverses of each other. Take an endoscopic triple $(H,s, \eta)$ and get the datum $(H', \mathcal{H}', s', \eta')$ and then get the triple $(H', \eta'^{-1}(s'), \eta'|_{\widehat{H}'})$. We need to show this triple is isomorphic to $(H,s,\eta)$. Consider the map $\eta^{-1} \circ \eta' : \widehat{H}' \to \widehat{H}$. This map is by definition $\Gamma_F$-equivariant, hence induces an isomorphism $\alpha: H \to H'$. Moreover we get $\widehat{\alpha}(s')=s$ as desired.

Now take a datum $(H, \mathcal{H}, s, \eta)$. Then we need to show this is isomorphic to the datum $(H', \mathcal{H}', s', \eta')$ induced by the triple $(H, \eta^{-1}(s), \eta|_{\widehat{H}})$. We first check that $\eta(\mathcal{H}) \subset \mathcal{H}'$. This follows because if $(g,w) \in \eta(\mathcal{H})$ then let $(x,w)$ by the pre-image of $(g,w)$ under $\eta$. Then $\eta \circ \mathrm{Int}(x,w)= \mathrm{Int}(g,w) \circ \eta$ so that $(g, w) \in \mc{H}'$. But now since both $\eta(\mathcal{H})$ and $\mathcal{H}'$ are split extensions of $W_F$, this implies they are equal. Furthermore, we have $s=s'$ by construction.

This completes our demonstration of the equivalence of the two definitions of endoscopic data.

\printbibliography

\end{document}